\documentclass[12,reqno]{amsart}
\usepackage{amsmath, amssymb, esint}
\usepackage[colorlinks=true]{hyperref}
\usepackage[pdftex]{graphicx,color}
\pdfoutput=1

\newtheorem*{conj}{Conjecture}
\newtheorem*{problem}{Problem $\star$}
\newtheorem{definition}{Definition}[section]
\newtheorem{theorem}{Theorem}[section]
\newtheorem{lemma}{Lemma}[section]
\newtheorem{remark}{Remark}[section]

\newtheorem{prop}{Proposition}[section]
\newtheorem{coro}{Corollary}[section]

\newcommand{\real}{\mathbb{R}}

\newcommand{\nat}{\mathbb{N}}

\newcommand{\eps}{\epsilon}

\newcommand{\chrc}[1]{\mathrm{1}_{#1}}
\newcommand{\del}{\partial}

\title{A free boundary problem inspired by a conjecture of De Giorgi}
\date{}
\author{Nikola Kamburov}
\thanks{Supported in part by NSF grant DMS-1069225.}

\numberwithin{equation}{section}

\begin{document}

\begin{abstract}
We study global monotone solutions of the free boundary problem that
arises from minimizing the energy functional $I(u) = \int |\nabla
u|^2 + V(u)$, where $V(u)$ is the characteristic function of the
interval $(-1,1)$. This functional is a close relative of the scalar
Ginzburg-Landau functional $J(u) = \int |\nabla u|^2 + W(u)$, where
$W(u) = (1-u^2)^2/2$ is a standard double-well potential. According
to a famous conjecture of De Giorgi, global critical points of $J$
that are bounded and monotone in one direction have level sets that
are hyperplanes, at least up to dimension $8$. Recently, Del Pino,
Kowalczyk and Wei gave an intricate fixed-point-argument
construction of a counterexample in dimension $9$, whose level sets
``follow" the entire minimal non-planar graph, built by Bombieri, De
Giorgi and Giusti (BdGG). In this paper we turn to the free boundary
variant of the problem and we construct the analogous example; the
advantage here is that of geometric transparency as the interphase
$\{|u| < 1\}$ will be contained within a unit-width band around the
BdGG graph. Furthermore, we avoid the technicalities of Del Pino,
Kowalczyk and Wei's fixed-point argument by using barriers only.
\end{abstract}

\maketitle
\bibliographystyle{plain}
\tableofcontents

\section{Introduction.}
In this paper we study the following free boundary problem:
\begin{align}
    \Delta u & = 0 \quad \text{in}\quad \Omega_{\text{in}} := \{x\in \Omega: |u(x)|< 1\} \notag \\
    u & = \pm 1 \quad \text{in} \quad \Omega \setminus \Omega_{\text{in}} \label{FBP} \\
    |\nabla u| & = 1 \quad \text{on} \quad \del
\Omega_{\text{in}}\cap \Omega = F^+(u) \sqcup F^{-}(u) \notag
\end{align}
where $\Omega\subseteq \real^n$ is a domain and the free boundary of
$u$ consists of two pieces
\begin{equation*}
     F^{+}(u) : = \del \{u=1\}\cap\Omega \qquad  F^{-}(u) := \del \{u = -1\}\cap\Omega.
\end{equation*}
In particular, we are interested in global solutions
($\Omega=\real^n$), which are monotonically increasing in the last
coordinate $x_n$. We pose the following question:
\begin{problem}
Let $n=9$. Does there exist a global solution to \eqref{FBP},
monotonically increasing in $x_9$, such that its level sets are not
hyperplanes?
\end{problem}

The question above should be read in the context of the prominent
\emph{De Giorgi's conjecture} concerning global solutions of the
Allen-Cahn equation
\begin{equation}\label{AC}
    \Delta u = (1-u^2)u \quad \text{in}~\real^n,
\end{equation}
namely:
\begin{conj}[De Giorgi \cite{DeGio}]
If $u \in C^2(\real^n)$ is a global solution of \eqref{AC} such that
$|u|\leq 1$ and $\del_{x_n} u> 0$, then the level sets $\{u=\lambda
\}$ are hyperplanes, at least for dimensions $n\leq 8$.
\end{conj}
The common nature of the PDE's \eqref{FBP} and \eqref{AC} is rooted
in the fact that they arise as Euler-Lagrange equations for the
closely related energy functionals $I$ and $J$, respectively
\begin{align}
I(u, \Omega) & = \int_{\Omega} |\nabla u|^2 + \mathcal{V}(u)
\quad \text{for} \quad u:\Omega \rightarrow [-1,1] \label{I} \\
J(u,\Omega) & = \int_{\Omega} |\nabla u|^2 +
\mathcal{W}(u),
\end{align}
where $\mathcal{V}(u) := \chrc{(-1,1)}(u)$ is a singular version of
the standard double-well potential $\mathcal{W}(u) :=
\frac{(1-u^2)^2}{2}$.

De Giorgi's conjecture has been motivated by a fascinating and deep
connection between the theory of semilinear elliptic PDE and the
theory of minimal surfaces. The connection was first rigorously
stated through the notion of $\Gamma$-convergence in the work of
Modica \cite{modica}. Assuming that $u$ minimizes $J$ in a large
ball $B_{1/\eps}$, $u_{\eps}(x) = u(x/\eps)$ minimizes the rescaled
energy
\begin{equation*}
    J_{\eps}(v, B_1) = \eps\int_{B_1} |\nabla v|^2 + \frac{1}{\eps}
    \int_{B_1}\mathcal{W}(v)
\end{equation*}
in the unit ball. What Modica proved was that as $\eps\to 0$, a
subsequence of minimizers $u_{\eps_k}$ of $J_{\eps_k}(\cdot, B_1)$
of uniformly bounded energy converges to
\begin{equation*}
    u_{\eps_k} \to \chrc{E} - \chrc{B_1\setminus E} \qquad \text{in}~
    L^1_{\text{loc}}(B_1),
\end{equation*}
where $E$ has a perimeter minimizing boundary in $B_1$, i.e. $\del
E$ is a minimal hypersurface. The convergence is in fact stronger:
Caffarelli and Cordoba (\cite{CafCordo1}, \cite{CafCordo2}) later
showed that the level sets $\{u_{\eps} = \lambda \}$ for
$-1<\lambda<1$ converge uniformly on compacts to the minimal
hypersurface $\del E$. Intuitively speaking therefore, the level
sets of a global minimizer of $J$ look like a minimal hypersurface
at large scales. The analogous statements can, of course, be made
for global minimizers of $I$.

The monotonicity assumption $\del_{x_n} u > 0$ in De Giorgi's
conjecture implies that $u$ is not only a stable critical point for
$J$, but that $u$ is, in fact, a global minimizer of $J$ in a
certain sense (see \cite{AlbAmbCabre}). Under the natural assumption
\begin{equation}\label{SavinAsmp}
    \lim_{x_n \to \pm \infty} u(x',x_n) = \pm 1.
\end{equation}
the level sets of $u$ are also graphs over $\real^{n-1}$ in the
$e_n$-direction. Therefore, the preceding discussion combined with
\emph{Bernstein's theorem}, which states that an entire minimal
graph in $\real^{n-1}\times \real$ is a hyperplane for $n\leq 8$
(cf. Simons \cite{Simons}), is what gives plausibility to De
Giorgi's conjecture. Moreover, the existence of a non-planar minimal
graph in dimension $n=9$, constructed by Bombieri, De Giorgi and
Giusti in \cite{BdGG}, strongly suggests that the conjecture is
likely to be false for $n\geq 9$.

There has been a lot of recent work which has almost completely
resolved De Giorgi's conjecture. It was fully established in
dimensions $n=2$ by Ghoussoub and Gui \cite{GhouGui} and $n=3$ by
Ambrosio and Cabr\'e \cite{AmbroCabre}, while Savin \cite{Savin}
managed to prove it for dimensions $2\leq n\leq 8$ under the
additional assumption \eqref{SavinAsmp}. Savin's approach has a
broader scope and applies to monotone global minimizers of the
functional $I$, as well (see \cite{SavinCDM}).



Recently Del Pino, Kowalczyk and Wei \cite{PKW} successfully
constructed a counterexample in dimension $n=9$. Roughly speaking,
their strategy is based upon the derivation of a sufficiently good
ansatz whose level sets ``follow" the Bombieri--De Giorgi--Giusti
(BdGG) minimal graph. This allows them to carry out an intricate
fixed point argument.

It was a desire to gain a better understanding of precisely what
geometric ingredients are responsible for the existence of this
important counterexample that led us to formulate and resolve in the
affirmative the alternative Problem $\star$.
\begin{theorem}\label{theBread}
There exists a solution $u:\real^9\to \real$ of \eqref{FBP} which is
monotonically increasing in $x_9$ and whose free boundary $F(u) =
F^+(u) \sqcup F^-(u)$ consists of two non-planar smooth graphs.
\end{theorem}

The study of the free boundary problem has an obvious geometric
advantage. In this setting, the interphase $\{|u|<1\}$ will be
contained within a unit-width band around the BdGG minimal graph, so
that ones does not have to worry about capturing a non-trivial
behaviour of the solution away from the band. We will use the method
of barriers which is elementary in nature and allows for a
transparent and relatively precise description of the solution (we
will be able to trap the solution quite tightly between the two
barriers). This way we avoid using fixed point arguments which are
arguably the main culprit for the level of technical complexity of
the construction by Del Pino, Kowalczyk and Wei.

To construct a solution to \eqref{FBP} once we are in possession of
a supersolution lying above a subsolution (we will define these
notions shortly), we adopt the strategy developed by De Silva
\cite{DS} in her study of global free boundary graphs that arise
from monotone solutions to the classical one-phase free boundary
problem:
\begin{equation}\label{FBP_0}
\begin{array}{rl}
    \Delta u = 0 & \text{in}\quad \Omega_p(u):=\{x\in \Omega:
u(x)>0\} \\
    |\nabla u| = 1 & \text{on}\quad F_{p}(u) := \Omega \cap
\del\Omega_p(u).
\end{array}
\end{equation}
Namely, a global solution to \eqref{FBP_0} is constructed as the
limit of a sequence of local minimizers of the one-phase energy
functional:
\begin{equation}\label{I0}
    I_0(u, \Omega) = \int_{\Omega} |\nabla u|^2 + \chrc{\{u>0\}}
\end{equation}
constrained to lie between a fixed \emph{strict} subsolution and a
fixed \emph{strict} supersolution to \eqref{FBP_0}. The strictness
condition ensures that the free boundary of each minimizer doesn't
touch the free boundaries of the barriers. Following the classical
ideas of Alt, Cafarelli and Friedman (\cite{AC}, \cite{ACF}), De
Silva shows that $u$ is a global energy minimizing \emph{viscosity}
solution, which is locally Lipschitz continuous and has
non-degenerate growth along its free boundary; moreover, if one
assumes that the two barriers are monotonically increasing in $x_n$,
the global solution can also be chosen to be monotonically
increasing in $x_n$ after a rearrangement. The harder part is the
regularity theory: De Silva's key observation is that the positive
phase of the minimizer is locally an NTA (non-tangentially
accessible) domain (\cite{JK}) which allows the application of the
powerful boundary Harnack principle. By comparing the solution with
a vertical translate, she rules out the possibility that the free
boundary contains any vertical segments, so that it is a graph, and
then she shows that the graph is, in fact, continuous. A more
sophisticated comparison argument by De Silva and Jerison \cite{DSJ}
establishes a Lipschitz bound on the free boundary graph. Hence, by
the classical result of Caffarelli \cite{Caf}, the free boundary is
locally a $C^{1,\alpha}$ graph, so that the global minimizer is
indeed a classical solution to \eqref{FBP_0}.

Obviously, the functionals $I$ and $I_0$ are close relatives: if $u$
minimizes $I(\cdot, \Omega)$ and $D\subset \Omega$ is a (nice
enough) subdomain, such that $D\cap \{u=1\} = \emptyset$, then $u+1$
minimizes $I_0(\cdot, D)$; similarly, if $D\cap \{u=-1\} =
\emptyset$, $1-u$ minimizes $I_0(\cdot, D)$. So, after we construct
a global minimizer $u$ to $I$, we will be in a position to directly
apply the regularity theory for the one-phase energy minimizers from
the discussion above to the free boundary of $u$.

Let us now give a brief outline of the arguments and the structure
of our paper.

\section{Outline of strategy.} First, let us recall the
definition of a classical super/subsolution to the one-phase problem
\eqref{FBP_0} (see for example \cite{Caf}).

\begin{definition}\label{defClaSuper0}
A classical supersolution (resp. subsolution) to \eqref{FBP_0} is a
non-negative continuous function $w$ in $\Omega$ such that
\begin{itemize}
\item $w\in
C^{2}(\overline{\Omega_p(w)})$;
\item $\Delta w \leq 0$ (resp. $\Delta w \geq 0$) in $\Omega_p(w)$;
\item The free boundary $F_p(w)$ is a $C^2$ surface and
\begin{equation*}
    0<|\nabla w| \leq 1 \quad \text{(resp.~ } |\nabla
    w|\geq 1) \quad \text{on}\quad F_p(w).
\end{equation*}
If the inequality above is strict, we call $w$ a \textbf{strict}
super (resp. sub) solution.
\end{itemize}
\end{definition}

The appropriate notion of a classical super/subsolution to our free
boundary problem \eqref{FBP} is, therefore, the following:

\begin{definition}\label{defClaSuper}
A classical supersolution (resp. subsolution) to \eqref{FBP} is a
non-negative continuous function $w$ in $\Omega$ such that
\begin{itemize}
\item $w\in
C^{2}(\overline{\Omega_{\text{in}}(w)})$;
\item $\Delta w \leq 0$ (resp. $\Delta w \geq 0$) in $\Omega_{\text{in}}(w)$;
\item The free boundary $F(w)=F^{+}(w)\sqcup F^-(w)$ consists of two $C^2$ surfaces and
\begin{align*}
    0<|\nabla w| & \leq 1 \quad \text{(resp.~ } |\nabla
    w|\geq 1) \quad \text{on}\quad F^{-}(w), \\
    |\nabla w| & \geq 1 \quad \text{(resp.~ } |\nabla
    w|\leq 1) \quad \text{on}\quad F^{+}(w).
\end{align*}
If the inequalities above are strict, we call $w$ a \textbf{strict}
super (resp. sub) solution to \eqref{FBP}.
\end{itemize}
\end{definition}


As mentioned in the introduction, the driving intuition is that the
level surfaces of a solution to \eqref{FBP} should follow the shape
of the BdGG entire minimal graph $\Gamma = \{(x', x_9)\in
\real^8\times\real: x_9 = F(x')\}$. The function $F:\real^8
\rightarrow\real$ satisfies the minimal surface equation $H[F] = 0$
where $H[\cdot]$ is the mean curvature operator (MCO)
\begin{align*}
    H[F] := \nabla \cdot \left(\frac{\nabla F}{\sqrt{1 + |\nabla
    F|^2}}\right).
\end{align*}
Note that there is a whole one-parameter family of such entire
minimal non-planar graphs, obtained by rescaling $\Gamma$:
\begin{equation*}
    \Gamma_{\alpha} := \alpha^{-1}\Gamma = \{x_9 = F_{\alpha}(x')\}
\end{equation*}
where $\alpha > 0$ and $F_{\alpha}(x'):= \alpha^{-1}F(\alpha x')$.
``Following the shape" should be interpreted in the sense that we
would like the solution $u$ and thus the trapping super/subsolution
$W$ and $V$ to behave asymptotically at infinity like the signed
distance to $\Gamma_{\alpha}$ for some $\alpha>0$:
\begin{equation*}
    V(x) \leq u(x) \leq W(x) \approx \text{signed dist}(x,{\Gamma_{\alpha}})
\end{equation*}
within their interphases (for the definition of the signed distance:
we take the sign to be positive if the point $x\in\real^9$ lies
``above the graph", i.e. if $x_9>F_{\alpha}(x')$, and negative
otherwise). This suggests that the coordinates
\begin{equation}
 \real^9 \ni x \rightarrow (y,z)\in \Gamma_{\alpha}\times \real,  \qquad  x = y + z
\nu_{\alpha}(y), \label{Fermialpha}
\end{equation}
where $\nu_{\alpha}(y)$ is the unit normal to $\Gamma_{\alpha}$ at
$y$ with $\nu_{\alpha}(y)\cdot e_9 >0$, will be particularly
well-suited to the problem. We will later show in Lemma
\ref{fermilemma} that the coordinates
\begin{equation}
 \real^9 \ni x \rightarrow (y,z)\in \Gamma_{1}\times \real,  \qquad  x = y + z
\nu_{1}(y), \label{Fermi}
\end{equation}
are well-defined in a thin band around $\Gamma = \Gamma_1$
\begin{equation*}
    \mathcal{B}_{\Gamma}(d) = \{x\in\real^9: \text{dist}(x,\Gamma) < d\}
\end{equation*}
for $d>0$ small enough. By taking blow-ups of space $x \to
\alpha^{-1} x$ we see the coordinates \eqref{Fermialpha} with
respect to the blow-up $\Gamma_{\alpha}$ will be well defined in the
band $\mathcal{B}_{\Gamma_{\alpha}}(\alpha^{-1}d)$. Thus we can
ensure that the coordinate system \eqref{Fermialpha} is well defined
in a unit-size band $\mathcal{B}_{\Gamma_{\alpha}} =
\mathcal{B}_{\Gamma_{\alpha}}(2)$ for all $\alpha>0$ small enough.
The trick of scaling will prove quite useful in what comes later, as
well. Its effect on the geometry in the unit-width band is described
in \S \ref{SecGeoUWB}, Lemma \ref{lemma_scaling}.

Not surprisingly, we look for a supersolution/subsolution that is a
polynomial in $z$:
\begin{equation*}
    w(y,z) = \sum_{k=0}^m h_k^{\alpha}(y)z^k,
\end{equation*}
where $h_1^{\alpha} \approx 1$ to main order at infinity, whereas
all the other coefficient decay appropriately to $0$. As it turns
out, $m=5$ suffices.

The Euclidean Laplacian is given by the following key formula:
\begin{equation*}
\Delta = \Delta_{\Gamma_{\alpha}(z)} + \del_{z}^2 -
    H_{\Gamma_{\alpha}(z)}(y)\del_z,
\end{equation*}
where $\Delta_{S}$ denotes the Laplace-Beltrami operator on a
surface $S$,
\begin{equation*}
    \Gamma_{\alpha}(z) = \{y+ z\nu_{\alpha}(y): y\in \Gamma\}
\end{equation*}
is a level set for the signed distance to $\Gamma_{\alpha}$ and
$H_{\Gamma_{\alpha}(z)}(y)$ is its mean curvature at $y+z\nu(y)$.
Note that if $k_i^{\alpha}(y)$ denote the principal curvatures of
$\Gamma_{\alpha}$ at $y$
\begin{equation}\label{hgammaz}
H_{\Gamma_{\alpha}(z)} =
\sum_{i=1}^{8}\frac{1}{(k_i^{\alpha})^{-1}-z} =
\sum_{i=1}^{8}\frac{k_i^{\alpha}}{1-k_i^{\alpha}z}.
\end{equation}
Expanding \eqref{hgammaz} in $z$, we formally have
\begin{equation}\label{Hgammaz_series}
    H_{\Gamma_{\alpha}(z)}(y)= (H_{1,\alpha} = 0) + \sum_{l=2}^{\infty} z^{l-1} H_{l,\alpha},
\end{equation}
where $H_{l,\alpha} = \sum_{i=1}^8 (k_i^{\alpha})^l$ is the sum of
the $l$-th powers of the principal curvatures of $\Gamma_{\alpha}$.
It turns out that the principal curvatures
\begin{equation*}
    k_{i}^{\alpha}(y) = O(\alpha(1+\alpha |x'|)^{-1}), \qquad y = (x',
    F_{\alpha}(x'))\in \Gamma_{\alpha}
\end{equation*}
meaning in particular that the series \eqref{Hgammaz_series}
converges rapidly for $(y,z)\in \mathcal{B}_{\Gamma_{\alpha}}$ and
$\alpha$ small enough. However, a more refined understanding of the
asymptotics of the quantities $H_{k} := H_{k,1}$ will be needed.  

Del Pino, Kowalczyk and Wei faced the exact same issue in
\cite{PKW}. For the purpose, they introduce the model graph
$\Gamma_{\infty}$, which has an explicit coordinate description and
which matches $\Gamma$ very well at infinity. This allows one to
approximate geometric data and geometric operators on $\Gamma$ with
their counterparts on $\Gamma_{\infty}$: for example, the second
fundamental form, the quantities $H_{l}$, the intrinsic gradient and
the Laplace-Beltrami operator. We will briefly present their results
concerning the geometry of $\Gamma$ and the closeness between
$\Gamma$ and $\Gamma_{\infty}$, and use their framework to prove
additional relevant estimates in Section \ref{SecGeoUWB}.

Choosing the coefficients $h_k^{\alpha}(y)$ so that $w$ meets the
supersolution conditions is the subject of Section \ref{TheMeat}. In
fact, most of the choices will be imposed on us (see Remark
\ref{forcedchoi}): they will be given in terms of geometric
quantities like $H_{l,\alpha}$ and their covariant derivatives. The
upshot is that (see Lemmas \ref{lemma_lap} and \ref{supGrad})
\begin{equation}\label{upshot}
\begin{aligned}
    \Delta w & = J_{\Gamma_{\alpha}} h_0^{\alpha} + h_2'^{\alpha} - z^2 H_{3, \alpha}
    + \text{~lower order terms} \qquad
    \text{in}~\mathcal{B}_{\Gamma_{\alpha}},\\
    |\nabla w| & = 1 \pm h_2'^{\alpha} + \text{~lower order terms}
    \qquad \text{on}~ \{w =\pm 1\},
\end{aligned}
\end{equation}
where
\begin{equation*}
J_{\Gamma_{\alpha}} = \Delta_{\Gamma_{\alpha}} + |A_{\alpha}|^2
\end{equation*}
is the Jacobi operator on $\Gamma_{\alpha}$ and $h_2'^{\alpha} = 2
h_2^{\alpha} - |A_{\alpha}|^2 h_0^{\alpha}$. Thus, by varying
$h_0^{\alpha}$ we can satisfy the superharmonicity condition and by
varying $h_2'^{\alpha}$ -- the gradient condition on the free
boundary. So, $h_2'^{\alpha}$ needs to be positive and we will also
require that $h_0^{\alpha}>0$. That way, we only have to flip the
signs of the coefficients $h_0^{\alpha}$ and $h_2^{\alpha}$ in the
ansatz (and leave the remaining ones unchanged) in order to produce
a \emph{subsolution} ansatz that automatically lies underneath the
supersolution.

So, we want the function $h_0^{\alpha}$ to be a positive
supersolution for $J_{\Gamma_{\alpha}}$ that satisfies an
appropriate differential inequality. It turns out that $J_{\Gamma}$
admits nonnegative supersolutions $h$ of the following types:
\begin{itemize}
\item Type $1$ is such that $J_{\Gamma}h$ can absorb terms that decay like $r^{-k}$ for
$k>4$. See Proposition \ref{h'}. This is useful when dealing with
the lower order terms in \eqref{upshot}.
\item Type $2$ can take care of the $|H_{3}|$-term which
is globally $O(r^{-3})$ but has the important additional property
that it vanishes on the Simons cone $S = \{(\vec{u}, \vec{v})\in
\real^4 \times \real^4: |\vec{u}|=|\vec{v}|\}$.
\end{itemize}
The Type $1$ supersolution is readily provided by Del Pino,
Kowalczyk and Wei's paper \cite[Proposition 4.2(b)]{PKW}. We build
the Type $2$ supersolution ourselves in Section
\ref{subsec_Supersols} (asymptotically in Lemma \ref{h"_infty} and
globally in Proposition \ref{h"}) and the construction involves a
very delicate patching of two supersolutions in a region of the
graph over the Simons cone. The ingredients are contained in the
analysis of the linearized mean curvature operator $H'[F_{\infty}]$
around $\Gamma_{\infty}$ carried out in \cite[\S 7]{PKW}, Lemma 7.2
and 7.3; for the reader's convenience we state these results in
Appendix \ref{PKWsupsinfinity}. The operators $J_{\Gamma_{\infty}}:=
\Delta_{\Gamma_{\infty}} + |A_{\infty}|^2$ and $H'[F_{\infty}]$ are
closely related (see \eqref{JacobiandLinearizedMC}) and in turn
$J_{\Gamma_{\infty}}$ is asymptotically close to $J_{\Gamma}$ (see
\eqref{Jacobiclo}). This allows one to first build a (weak)
supersolution away from the origin, which can then be upgraded to a
global smooth supersolution via elliptic theory.

Having these two types of barriers for $J_{\Gamma}$ we will be able
to satisfy the free boundary supersolution conditions far away from
the origin. In order to satisfy them globally, we employ the trick
of scaling by $\alpha$. That way we can also ensure that both the
supersolution and the subsolution are monotonically increasing in
$x_9$ for all small enough $\alpha>0$.

Once we have obtained the monotone sub-super solution pair $V\leq W$
we proceed to construct the solution $u$ to \eqref{FBP} as a global
minimizer $u$ of the functional $I$, constrained to lie in-between
$V\leq u\leq W$. This is the subject of Section
\ref{Section_FBexireg}.

\section*{Acknowledgements.} The author would like to thank his
advisor, David Jerison, for suggesting this problem and for his
exquisite guidance.

\section{The Bombieri-De Giorgi-Giusti graph $\Gamma$ and its approximation.}\label{SecGeoUWB}

In this section we will describe several results concerning the
asymptotic geometry of the entire minimal graph in $9$ dimensions,
constructed by Bombieri, De Giorgi and Giusti in \cite{BdGG}. Some
of them have been covered by the analysis of the graph $\Gamma$,
carried out by Del Pino, Kowalczyk and Wei in \cite{PKW} (Lemmas
\ref{fermilemma}, \ref{lemma_levelComp}, \ref{lemma_scaling} and
\ref{TcoorandsizeofG} below). We will state those in a form suitable
for our later computations. Furthermore, we will establish the
important estimates for the covariant derivatives of $H_l$ in Lemma
\ref{lemma_Hlest}.

Let us first set notation. Recall, we denote the entire minimal
graph by
\begin{equation*}
 \Gamma = \{(x',x_9)\in \real^8\times \real: x_9 = F(x')\} \subset \real^8\times \real,
\end{equation*}
where $F:\real^8 \rightarrow \real$ is an entire solution to the
minimal surface equation (MSE):
\begin{equation}\label{MSE}
    \nabla\cdot\left(\frac{\nabla F}{\sqrt{1+|\nabla F|^2}}\right) =
    0 \quad \text{in} \quad \real^8.
\end{equation}
The graph enjoys certain nice symmetries. Write
$x'=(\vec{u},\vec{v})\in \real^8$, where $\vec{u},\vec{v} \in
\real^4$ and $u=|\vec{u}|$, $v=|\vec{v}|$. Then $F$ satisfies
\begin{equation}\label{sym}
\begin{aligned}
& F \text{ is radially symmetric in
both } \vec{u}, \vec{v}, \text{
i.e. } F = F(u,v) \\
& F(u,v) = - F(v,u),
\end{aligned}
\end{equation}
so that $F$ vanishes on the \emph{Simons cone}
\begin{equation*}
S = \{u = v\} = \{x_1^2 + \cdot + x_4^2 = x_5^2 + \cdot + x_8^2 \}
\subset
    \real^8.
\end{equation*}



\subsection{Geometry of the unit-width band around
$\Gamma$.}

We will use powers of
\begin{equation*}
    r(x) = |x'| \quad \text{for} \quad x=(x',x_9) \in
    \real^8\times\real=\real^9.
\end{equation*}
to measure the decay rate of various quantities at infinity. As
mentioned in the beginning, we are interested in the domain of
definition for the coordinates \eqref{Fermi}:
\begin{equation*}
    \real^9 \ni x \to (y,z)\in\Gamma \times \real \qquad x = y + z
    \nu(y),
\end{equation*}
where $\nu(y)$ is the unit normal to $\Gamma$ at $y\in\Gamma$ such
that $\nu(y)\cdot e_9 >0$. In particular, we are considering a type
of domains which is a thin band around the graph $\Gamma$:
\begin{equation*}
    \mathcal{B}_{\Gamma}(d) = \{x\in\real^9: \text{dist}(x,\Gamma) <
    d\}.
\end{equation*}

\begin{lemma}[cf. Remark 8.1 in \cite{PKW}]\label{fermilemma}
There exists a small enough $d >0$ such that the coordinates
\eqref{Fermi} are well-defined in $\mathcal{B}_{\Gamma}(d)$.
\end{lemma}

The level surfaces for the signed distance to $\Gamma$
\begin{equation*}
    \Gamma(z) = \{x \in \real^9:~ \text{signed dist}(x,\Gamma) = z\}
\end{equation*}
are prominent in our analysis as it will be necessary to estimate
derivative operators on $\Gamma(z)$, for small $z$. According to
Lemma \ref{fermilemma}, the coordinates \eqref{Fermi} are
well-defined in a thin-enough band $\mathcal{B}_{\Gamma}(d)$.
Equivalently, the orthogonal projection onto $\Gamma$ is
well-defined in $\mathcal{B}_{\Gamma}(d)$:
\begin{equation*}
    \pi_{\Gamma}: \mathcal{B}_{\Gamma}(d) \rightarrow \Gamma \qquad
    \pi_{\Gamma}(x) = y
\end{equation*}
and provides a diffeomorphism between $\Gamma(z)$ and $\Gamma$ for
each $z$, $|z|<d$. Thus, one can identify functions $f$ defined on
$\Gamma$ with functions $\tilde{f}$ defined on $\Gamma(z)$ via the
projection $\pi_{\Gamma}$:
\begin{equation}\label{pilift}
    \tilde{f} = f\circ \pi_{\Gamma}.
\end{equation}
We will use the following lemma, quantifying the proximity of the
gradient and the Laplace-Beltrami operator on $\Gamma(z)$ acting on
$\tilde{f}$ to the gradient and the Laplace-Beltrami operator on
$\Gamma$ acting on $f$.
\begin{lemma}[cf. \S 3.1 in \cite{PKW}]\label{lemma_levelComp} There exists $0<d<2$ small enough so that
\begin{itemize}
\item The coordinates \eqref{Fermi} are well-defined;
\item If we fix $z$ with $|z|<d$ and assume $\tilde{f}\in
C^2(\Gamma(z))$ and $f\in C^2(\Gamma)$ are related via
\eqref{pilift}, we have the following comparison of their gradients,
viewed as Euclidean vectors:
\begin{equation}\label{GradComp}
 |\nabla_{\Gamma(z)}\tilde{f} - (\nabla_{\Gamma}f) \circ \pi_{\Gamma}| = O \Big(z\frac{|D_{\Gamma}f|}{1+r}\Big)\circ \pi_{\Gamma},
\end{equation}
while the Laplace-Beltrami operators on $\Gamma(z)$ and $\Gamma$ are
related by:
\begin{equation}\label{DeltaComp}
\Delta_{\Gamma(z)}\tilde{f} = \left(\Delta_{\Gamma} f + O\Big(z
\frac{|D_{\Gamma}^2 f|}{1+r} + z\frac{|D_{\Gamma} f|}{1 + r^2} \Big)
\right)\circ \pi_{\Gamma}.
\end{equation}
\end{itemize}
\end{lemma}

Recall that we are interested in the decay rate of the quantities
\begin{equation*}
    H_{l} = \sum_{i=1}^8 k_i^l,
\end{equation*}
where $\{k_i\}_{i=1}^8$ are the principal curvatures of $\Gamma$ --
the eigenvalues of the second fundamental form. We prove the
following useful estimates.

\begin{lemma}\label{lemma_Hlest} The $k$-th order intrinsic derivatives of the quantity
$H_l$ are bounded by
\begin{equation}\label{Hl_FullEst}
    |D_{\Gamma}^k H_{l}(y)| \leq \frac{C_{kl}}{1+ r^{l+k}(y)}.
\end{equation}
for some numerical constants $C_{kl}>0$.
\end{lemma}

Finally, we would like to investigate how scaling space $x
\rightarrow \alpha^{-1} x$ affects the estimates in Lemmas
\ref{lemma_levelComp} and \ref{lemma_Hlest}. For a function $f$ on
$\Gamma$, define $f_{\alpha}$ to be the corresponding function on
$\Gamma_{\alpha} = \alpha^{-1}\Gamma$:
\[
    f_{\alpha}(y) = f(\alpha y).
\]
Also, denote $\{k_{i}^{\alpha}\}_{i=1}^8$ to be the principal
curvatures of $\Gamma_{\alpha}$ and
\begin{equation*}
    H_{l,\alpha} = \sum_{i=1}^8 (k_{i}^{\alpha})^l.
\end{equation*}

\begin{lemma}\label{lemma_scaling} Scaling space $x \rightarrow
\alpha^{-1}x$ has the following effects:
\begin{itemize}
\item The intrinsic $k$-th order derivatives of $f_{\alpha}$, $k=0,1,2,\ldots,$ scale like
\begin{equation}\label{der_scaling}
    D_{\Gamma_{\alpha}}^k f_{\alpha}(y) = \alpha^k (D_{\Gamma}^k f)(\alpha
    y)\qquad y\in \Gamma_{\alpha};
\end{equation}

\item The quantities
\begin{align}\label{Hl_scaling}
    D_{\Gamma_{\alpha}}^k H_{l,\alpha}(y) & = \alpha^{k+l} (D_{\Gamma}^k H_{l})(\alpha
    y) = O\left(\frac{\alpha^{l+k}}{1+ (\alpha
    r(y))^{l+k}}\right) \qquad y\in \Gamma_{\alpha};
\end{align}

\item If the coordinates \eqref{Fermi} are well-defined in a band
$\mathcal{B}_{\Gamma}(d)$, the coordinates
\begin{equation*}
    \real^9 \ni x \rightarrow (y,z)\in \Gamma_{\alpha}\times \real \qquad x
    = y + z \nu_{\alpha}(y),
\end{equation*}
where $\nu_{\alpha}(y) = \nu(\alpha y)$ is the unit-normal to
$\Gamma_{\alpha}$ at $y$, will be well-defined in the band
$\mathcal{B}_{\Gamma_{\alpha}}(d/\alpha)$. Thus, the orthogonal
projection onto $\Gamma_{\alpha}$ is well-defined in
$\mathcal{B}_{\Gamma_{\alpha}}(d/\alpha)$:
\begin{equation*}
    \pi_{\Gamma_{\alpha}}: \mathcal{B}_{\Gamma_{\alpha}}(d/\alpha)
    \to \Gamma_{\alpha} \qquad \pi_{\Gamma_{\alpha}}(x) = y;
\end{equation*}

\item If $|z| < d/\alpha$ is small enough and $\tilde{f}_{\alpha} \in C^2(\Gamma_{\alpha}(z))$ and $f_{\alpha}\in
C^2(\Gamma_{\alpha})$ are related via $\tilde{f}_{\alpha} =
f_{\alpha}\circ \pi_{\Gamma_{\alpha}}$, the estimates corresponding
to \eqref{GradComp} and \eqref{DeltaComp} take the form
\begin{align}
     & |\nabla_{\Gamma_{\alpha}(z)}\tilde{f}_{\alpha} - (\nabla_{\Gamma_{\alpha}}f_{\alpha}) \circ \pi_{\Gamma_{\alpha}}| = O \Big(\frac{\alpha z|D_{\Gamma_{\alpha}}f_{\alpha}|}{1+\alpha r}\Big)\circ \pi_{\Gamma_{\alpha}},
\label{GradCompAlpha} \\
    & \Delta_{\Gamma_{\alpha}(z)}\tilde{f}_{\alpha} = \left(\Delta_{\Gamma_{\alpha}} f_{\alpha} + O\Big(\alpha z
\frac{|D_{\Gamma_{\alpha}}^2 f_{\alpha}|}{1+\alpha r} + \alpha^2
z\frac{|D_{\Gamma_{\alpha}} f_{\alpha}|}{1 + (\alpha r)^2} \Big)
\right)\circ \pi_{\Gamma_{\alpha}}. \label{DeltaCompAlpha}
\end{align}
\end{itemize}
\end{lemma}

Let us now turn to the proofs of the aforementioned lemmas.

We will take advantage of the following local representation of the
minimal graph $\Gamma$. At each $y = (x_0', F(x_0'))\in \Gamma$
denote by $T = T(y)$ the tangent hyperplane to $\Gamma$ at $y$. A
simple consequence of the Implicit Function Theorem states that
$\Gamma$ can locally be viewed as a smooth (minimal) graph over a
neighbourhood in the tangent hyperplane $T$. Concretely, if
$\{\tilde{e}_i\}_{i=1}^8$ is an orthonormal basis for $T$ and
$\tilde{e}_9 = \nu(y)$ is the unit normal to $T$, there exists a a
small enough $a = a(y)$ and a smooth function $G: T\cap B_{a}(y)
\rightarrow \real$ so that in a neighbourhood of $x_0'$,
\begin{equation}\label{defGt}
(x', F(x')) = (x_0', F(x_0')) + \sum_{i=1}^8 t_i \tilde{e}_i +
G(t)\tilde{e}_9 \qquad \forall |t|<a.
\end{equation}
Moreover, $G(t)$ satisfies the MSE $H[G]=0$ in $\{|t|<a\}$. In
\cite{PKW} the authors establish the following key estimates for
$G$, which provide the basis for the lemmas, stated above.

\begin{lemma}[cf. Proposition 3.1 and \S 8.1 in \cite{PKW}]\label{TcoorandsizeofG}
Fix $y\in \Gamma$ and let $\rho = 1 + r(y)$. There exists a
 constant $\beta> 0$, independent of $y$, such that the local
representation \eqref{defGt} is defined in a neighbourhood $\{|t| <
a(y)\} \subset T$ with $a(y) = \beta \rho$. Moreover,
\begin{align}\label{Gderivatives}
  |D_t G(t)| \leq \frac{c|t|}{\rho}, \qquad |D_t^k G(t)|& \leq
  \frac{c_k}{\rho^{k-1}} \qquad \text{in} \quad  |t|\leq \beta \rho
\end{align}
for $k\in\nat$ and some numerical constants $c, c_k>0$. Also, the
unit normal $\nu$ to $\Gamma$ doesn't tilt significantly over the
same neighbourhood:
\begin{equation}\label{normtilt}
    |\nu(t,G(t)) - \nu(y)| \leq \frac{c|t|}{\rho} \qquad |t|\leq
\beta \rho.
\end{equation}
\end{lemma}
The proof of Lemma \ref{TcoorandsizeofG} is based on Simon's
estimate for the second fundamental form of minimal graphs that
admit tangent cylinders at infinity, \cite[Thm.4, p.673]{Simon},
\[
    |A|^2(y) \leq \frac{c}{1 + r(y)^2},
\]
and employs standard MSE estimates applied on an appropriate rescale
of $G$.

Lemma \ref{fermilemma}, the possibility to define the coordinates
\eqref{Fermi} in a thin enough band around $\Gamma$, is a corollary
of \eqref{normtilt}.

\begin{proof}[Proof of Lemma \ref{fermilemma}]
Assume the contrary: that there doesn't exist a $d>0$ for which the
coordinates \eqref{Fermi} are well-defined in
$\mathcal{B}_{\Gamma}(d)$. The coordinates will fail to represent a
point $x\in \mathcal{B}_{\Gamma}(d)$ uniquely when there exist two
points $y_1\neq y_2 \in \Gamma$ such that
\begin{equation*}
    |x-y_1| = |x-y_2| = \text{dist}(x,\Gamma) < d,
\end{equation*}
i.e. if
\begin{equation*}
    x = y_1 + z \nu(y_1) = y_2 + z\nu(y_2).
\end{equation*}
We have $|y_1 - y_2| = |z| |\nu(y_2) - \nu(y_1)| \leq 2 d$. This
means that if $d$ is sufficiently small ($d \leq
\frac{1}{2}\beta(1+r(y_1))$, for example), $y_2$ lies in the portion
of $\Gamma$ which is a graph over $T(y_1) \cap \{|t|\leq \beta
(1+r(y_1))\}$ -- a neighbourhood of the tangent hyperplane at $y_1$.
But then \eqref{normtilt} gives us
\begin{equation*}
    |y_1 - y_2| \leq d |\nu_1 - \nu_2| \leq \frac{cd}{1+ r(y_1)} |y_1 -
y_2| \leq cd |y_1 - y_2|
\end{equation*}
which is impossible whenever $d$ is small enough so that $cd < 1$.
\end{proof}

\begin{proof}[Proof of Lemma \ref{lemma_levelComp}]
Let $\tilde{y}\in \Gamma(z)$, $y = \pi_{\Gamma}(\tilde{y})$ and let
$T=T(y)$ be the tangent hyperplane to $\Gamma$ at $y$.  Use the
Euclidean coordinates $t$ on $T$ to parameterize $\Gamma(z)$ near
$\tilde{y}$:
\begin{equation*}
    t \rightarrow (t, G(t)) + z \nu(y(t)),
\end{equation*}
where $\nu(y(t)) = \frac{(-D_t G, 1)}{\sqrt{1+|D_t G|^2}}$ is the
unit normal to $\Gamma$ at $y(t)=(t,G(t))$. Note that in these
coordinates
\begin{equation*}
    \tilde{f}(t) = f(t).
\end{equation*}
As before, set $\rho = 1 + r(y)$ and use Einstein index notation.
Because of \eqref{Gderivatives}, the metric tensor $g(z)$ on
$\Gamma(z)$ computes to:
\begin{align*}
    g_{ij}(z) & = \delta_{ij} + G_iG_j + z(\tilde{e}_i + G_i \nu)\cdot
\del_j \nu + z(\tilde{e}_j + G_j \nu)\cdot \del_i \nu + z^2 \del_i
\nu\cdot
\del_j \nu \\
& = g_{ij}(0) - 2z G_{ij} + z^2 \del_i \nu \cdot \del_j \nu =
g_{ij}(0) + O(z\rho^{-1}) + O(z^2\rho^{-2})
\end{align*}
while its inverse
\begin{equation*}
    g^{ij}(z) = g^{ij}(0) + O(z\rho^{-1})
\end{equation*}
for $|z|\leq d$ small enough.  Noting that $g_{ij} = g_{ij}(0)$ is
the metric tensor on $\Gamma$ in $t$-coordinates, we see that the
difference of gradients, viewed as vectors in Euclidean space:
\begin{align*}
 & \big|\nabla_{\Gamma(z)}\tilde{f}(\tilde{y}) - \nabla_{\Gamma}f(y)\big| =  \\
=& \big|g^{ij}(z)\del_j f(0)(\tilde{e}_i + G_i(0)\tilde{e}_9 + z
\del_i\nu(0))
- g^{ij}(0) \del_j f(0)(\tilde{e}_i + G_i(0)\tilde{e}_9)\big| = \\
 =& O(z|D_{\Gamma}f(y)|\rho^{-1}).
\end{align*}
The derivatives of the metric tensor satisfy:
\begin{equation*}
    \del_k [g_{ij}(z)] = \del_k g_{ij} + O(z\rho^{-2})
\end{equation*}
so that
\begin{equation*}
    \del_k [g^{ij}(z)] = -g^{il}(z)\del_k [g_{lm}(z)] g^{mj}(z) = \del_k
g^{ij} + O(z\rho^{-1}).
\end{equation*}
Thus,
\begin{align*}
    \Delta_{\Gamma(z)}\tilde{f}(\tilde{y}) & = \frac{1}{\sqrt{|g(z)|}}\del_i(g^{ij}(z) \sqrt{|g(z)|}\del_j
f) =  \\ & = g^{ij}(z) \del^2_{ij} f + \del_ig^{ij}(z) \del_j f +
\frac{g^{ij}(z)}{2} \frac{\del_i |g(z)|}{|g(z)|} \del_j f =\\
& = \Delta_{\Gamma} f(y) + O(z\rho^{-1}|D_{\Gamma}^2 f(y)|) +
O(z\rho^{-2}|D_{\Gamma} f(y)|) .
\end{align*}

\end{proof}

\begin{proof}[Proof of Lemma \ref{lemma_Hlest}]
Set $\rho = 1 + r(y)$ and write the metric tensor of $\Gamma$ around
$y$ in the coordinates $t$, \eqref{defGt}:
\begin{equation*}
    g_{ij} = \delta_{ij} + G_i G_j = \delta_{ij} + O(\beta^2) \qquad
|t|\leq \beta \rho.
\end{equation*}
Its inverse takes the form
\begin{equation*}
    g^{ij} = \delta_{ij} -\frac{G_i G_j}{1+ |\nabla G|^2} =
\delta_{ij} + O(\beta^2) \qquad |t|\leq \beta \rho.
\end{equation*}
Mind that constants in the $O$-notation are independent of $\rho$.
Taking into account \eqref{Gderivatives} we see that for $|t|\leq
\beta \rho,\quad k=0,1,2\ldots$
\begin{equation}\label{tderTensor}
    |D_t^k g_{ij} (t)| \leq \frac{c_k'}{\rho^k} \qquad |D_t^k g^{ij} (t)| \leq
    \frac{c_k''}{\rho^k}.
\end{equation}
An easy consequence is the fact that the intrinsic $k$-th order
derivative of a function $f$ on $\Gamma$ at $y$ will be majorized by
the $t$-derivatives of $f$ at $t=0$ up to order $k$ as follows:
\begin{equation}\label{commensurate}
    |D_{\Gamma}^k f(y)| \leq \sum_{j=1}^k c_j \rho^{j-k} |D_t^j f(0)|
\end{equation}
for some numerical constants $c_k>0$. Since $g_{ij} = O(1)$ and
$g^{ij} = O(1)$, it suffices to show that
\begin{equation}
\begin{aligned}
    \nabla_{I} f & = \del_{I} f + \sum_{|J|<k}
    c_{J}(t) \del_{J} f, \quad \text{where} \\
    |D_t^m c_{J} (t)| & =
    O(\rho^{-k+|J|-m}) \qquad |t|\leq \beta \rho, \qquad m=0,1,\ldots \label{induHl}
\end{aligned}
\end{equation}
Here, of course, $I, J$ denote multi-indices (e.g. if $I = (i_1,
i_2, \ldots, i_k)$, $|I|=k$), $\nabla$ denotes covariant
differentiation and
\begin{align*}
\del_{I} f & = \del^k_{i_1 i
_2\ldots i_k} f \\
 \nabla_{I} f & = (\nabla^k
f)(\del_{i_1}, \del_{i_2}, \cdots, \del_{i_k}).
\end{align*}
We'll prove \eqref{induHl} by induction on $k = |I|$. When $k=1$,
the statement is obviously true and assume it holds up to $k-1$. For
convenience, define the following transformation on multi-indices of
length $k-1$:
\begin{equation*}
    \sigma_l^j(j_1, j_2, \ldots, j_{k-1}) = (j_1, \ldots, j_{l-1}, j,
j_{l+1}, \ldots, j_{k-1}) \qquad 1\leq l \leq k-1.
\end{equation*}
So, if $|I| = k$ and we write $I = (i_1, I')$, $I' = (i_2, \ldots,
i_{k})$, the covariant differentiation rule gives
\begin{align*}
    \nabla_{I} f = \del_{i_1} (\nabla_{I'}f) - \sum_{\substack{1\leq l\leq k-1 \\ 1\leq j\leq 8}}\Gamma_{i_1
    i_{l+1}}^j
    \nabla_{\sigma_l^j(I')} f,
\end{align*}
where $\Gamma_{ij}^k$ are the Christophel symbols for the metric
tensor $g$ in the coordinates $t$. We only need to check $D_t^m
\Gamma_{ij}^k = O(\rho^{-1-m})$, which follows immediately from
\eqref{tderTensor}. The induction step is complete.

Let us apply \eqref{commensurate} to the second fundamental form
\begin{equation*}
    A_i{}^j = A_{ik}g^{kj} = \frac{G_{ik} g^{kj}}{\sqrt{1+|\nabla
G|^2}} = O(\rho^{-1}) \qquad |t|\leq \beta \rho,
\end{equation*}
(we use the Einstein index notation again). Observe that its
$t$-derivatives decay like
\begin{equation*}
    |D_t^m(A_i{}^j)| = O(\rho^{-1-m})\qquad |t|\leq \beta \rho.
\end{equation*}
Since $H_l = \text{Trace}([A_i{}^j]^l)$,
\begin{equation*}
    D_t^m H_l = \text{Trace}(D_t^m [A_i{}^j]^l) = O(\rho^{-m-l})\qquad |t|\leq \beta \rho.
\end{equation*}
Hence, \eqref{commensurate} implies the desired
\begin{equation*}
     |D_{\Gamma}^k H_{l}(y)| \leq \frac{C_{kl}'}{\rho^{l+k}}.
\end{equation*}
\end{proof}

\begin{proof}[Proof of Lemma \ref{lemma_scaling}]
The results in the lemma are obtained after simple length-scale
considerations. Equation \eqref{der_scaling} is immediate. The
principal curvatures $k_i^{\alpha}(y)$ scale like
$\text{distance}^{-1}$, so that
\begin{equation*}
    k_{i}^{\alpha}(y) = \alpha k_{i}(\alpha y) \qquad y\in \Gamma_{\alpha},
\end{equation*}
and thus
\begin{equation}
    H_{l,\alpha}(y) = \alpha^{l}H_{l}(\alpha y) = \alpha^{l} (H_{l})_{\alpha}(y).
\end{equation}
We invoke \eqref{der_scaling} and \eqref{Hl_FullEst} to obtain the
full estimate \eqref{Hl_scaling}.

We are left to check \eqref{GradCompAlpha} and
\eqref{DeltaCompAlpha}, which follow from \eqref{GradComp} and
\eqref{DeltaComp} right after we note that
\begin{align*}
        \big(\nabla_{\Gamma_{\alpha}(z)} \tilde{f}_{\alpha}\big)(y+z\nu_{\alpha}(y)) & = \alpha \big(\nabla_{\Gamma(\alpha
z)}\tilde{f}\big)(\alpha (y + z \nu(y)) \\
    \big(\Delta_{\Gamma_{\alpha}(z)}\tilde{f}_{\alpha}\big)(y+z\nu_{\alpha}(y)) & = \alpha^2 \big(\Delta_{\Gamma(\alpha
z)}\tilde{f}\big)(\alpha (y + z \nu(y)),
\end{align*}
where $\tilde{f} = f \circ \pi_{\Gamma}$ is the lift of $f$ onto
$\Gamma(\alpha z)$.

\end{proof}

\subsection{Proximity between $\Gamma$ and the model graph.}

A more refined knowledge of the asymptotics of $\Gamma$ is needed in
order to carry out the construction of a supersolution to
\eqref{FBP}. To extract better information about geometry of
$\Gamma$ at infinity, Del Pino, Kowalczyk and Wei \cite[\S 2]{PKW}
introduce a model graph $\Gamma_{\infty}$, which has an explicit
formula and which approximates $\Gamma$ very well at infinity.
Namely, the model graph
\begin{equation*}
    \Gamma_{\infty} = \{(x',x_9)\in \real^9: x_9 = F_{\infty}(x')\}
\end{equation*}
where $F_{\infty}:\real^8 \rightarrow \real$ solves the
``homogenized" MSE:
\begin{equation}\label{hMSE}
    \nabla\cdot\left(\frac{\nabla F_{\infty}}{|\nabla F_{\infty}|}\right) =
    0,
\end{equation}
has the same growth ($\sim r^3$) at infinity as $F$ and shares the
same symmetries \eqref{sym}. This determines the function
$F_{\infty}$ uniquely up to a multiplicative constant: if we use
polar coordinates to write
\begin{equation*}
    u = r \cos\theta \qquad v = r\sin\theta,
\end{equation*}
the function takes the form $F_{\infty}(r,\theta) = r^3 g(\theta)$
where $g(\theta)\in C^2[0,\pi/2]$ is the unique (up to a scalar
multiple) solution to
\begin{align*}
& \frac{21 g \sin^2(2\theta)}{\sqrt{9g^2 + g'^2}} + \left(\frac{g'
\sin^3(2\theta)}{\sqrt{9g^2 + g'^2}} \right)' = 0, \qquad \text{in}
\quad \theta\in [0,\frac{\pi}{2}],
\\
& g'(\frac{\pi}{2}) = 0.
\end{align*}
which is \emph{odd} with respect to $\theta = \pi/4$. For
concreteness, we pick the $g(\theta)$ which satisfies in addition $
g'(\frac{\pi}{4}) = 1$.

Del Pino, Kowalczyk and Wei then prove the following result
quantifying the asymptotic proximity between the BdGG graph and the
model graph.

\begin{theorem}[cf. Theorem 2 in \cite{PKW}]\label{ThmRefined}
There exists a function $F=F(u,v)$, an entire solution to the
minimal surface equation \eqref{MSE} which has the symmetries
\eqref{sym} and satisfies
\begin{equation}\label{refine}
    F_{\infty} \leq F \leq F_{\infty} + \frac{C}{r^{\sigma}}\min\{F_{\infty}, 1\} \quad
    \text{in} \quad \theta\in[\frac{\pi}{4},\frac{\pi}{2}], \quad r> R_0
\end{equation}
for some constants $C, R_0 >0$ and $0<\sigma<1$.
\end{theorem}

The proximity of $\Gamma$ and $\Gamma_{\infty}$ at infinity allows
one to approximate geometric data and geometric operators defined on
the non-explicit $\Gamma$ with their counterparts on the explicit
$\Gamma_{\infty}$. To put it more concretely: one can use the
orthogonal projection $\pi_{\Gamma}$ onto $\Gamma$ to identify
functions $f$ defined on $\Gamma$ with functions $f_{\infty}$
defined on $\Gamma_{\infty}$ far away from the origin:
\begin{equation}\label{lifttomodel}
    f_{\infty} = f \circ \pi_{\Gamma} \qquad \Gamma_{\infty} \cap \{r > R\}
\end{equation}
for $R>0$ large enough. Then one can compare $(\nabla_{\Gamma}f
)\circ\pi_{\Gamma}$ to $\nabla_{\Gamma_{\infty}} f_{\infty}$ and
$(\Delta_{\Gamma} f)\circ\pi_{\Gamma}$ to
$\Delta_{\Gamma_{\infty}}f_{\infty}$. Also, if we denote
$\{(k_{\infty,i}\}_{i=1}^8$ to be the principal curvatures of
$\Gamma_{\infty}$, and
\begin{equation*}
    H_{\infty,l} = \sum_{i=1}^8 (k_{\infty,i})^l,
\end{equation*}
one can use the explicit $H_{\infty,l}$ to approximate the curvature
quantities $H_l$ associated with $\Gamma$. The ultimate goal is to
approximate the Jacobi operator on $\Gamma$,
\begin{equation*}
    J_{\Gamma} f = (\Delta_{\Gamma} + H_2)f
\end{equation*}
asymptotically with the Jacobi operator on $\Gamma_{\infty}$,
\begin{equation*}
    J_{\Gamma_{\infty}} f_{\infty} = (\Delta_{\Gamma_{\infty}} +
    H_{\infty,2})f_{\infty}.
\end{equation*}

On the basis of Theorem \ref{ThmRefined}, Del Pino, Kowalczyk and
Wei establish the following list of results.
\begin{lemma}[cf. \S 8.2, 8.3 in \cite{PKW}]\label{lemma_close} Far enough away from the origin,
$r>R$, for some constant $0<\sigma <1$,
\begin{itemize}
\item One can express $\Gamma_{\infty}$ locally as a graph of a function $G_{\infty}(t)$ over a neighbourhood of the
tangent hyperplane $T=T(y)$ to $y\in \Gamma$ with $r(y)>R$.
Moreover, for some constants $C_k>0$, $k=0,1,\ldots$
\begin{equation*}
    |\left.D_t^k (G-G_{\infty})\right|_{t=0}| \leq \frac{C_k}{r(y)^{k+1+\sigma}},
\end{equation*}
where $(t, G(t))$ is the local parametrization \eqref{defGt} of
$\Gamma$.

\item The Laplace-Beltrami operator on $\Gamma$
can be approximated with the Laplace-Beltrami operator on
$\Gamma_{\infty}$ as follows:
\begin{equation}\label{lapBelt_comparison}
    (\Delta_{\Gamma}f)\circ\pi_{\Gamma} = \Delta_{\Gamma_{\infty}}f_{\infty} +
    O\left(r^{-2-\sigma}|D_{\Gamma_{\infty}}^2 f_{\infty}| +
    r^{-3-\sigma} |D_{\Gamma_{\infty}} f_{\infty}|\right),
\end{equation}
where $f$ and $f_{\infty}$ are related via \eqref{lifttomodel}.
\item The quantities $H_2=|A|^2$ and $H_3$ of $\Gamma$ are
approximated by $H_{\infty, 2}$ and $H_{\infty,3}$, respectively, as
follows:
\begin{align}
    H_2\circ\pi_{\Gamma} & = H_{\infty,2} + O(r^{-4-\sigma}) \label{curv2clo}\\
    H_3\circ \pi_{\Gamma} & = H_{\infty,3} + O(r^{-5-\sigma}). \label{curv3clo}
\end{align}

\item Therefore, the Jacobi operators on $\Gamma$ and
$\Gamma_{\infty}$ are related by
\begin{equation}\label{Jacobiclo}
    (J_{\Gamma}f)\circ\pi_{\Gamma} = J_{\Gamma_{\infty}}f_{\infty} + O\left(r^{-2-\sigma}|D_{\Gamma_{\infty}}^2 f_{\infty}| +
    r^{-3-\sigma} |D_{\Gamma_{\infty}} f_{\infty}| +
    r^{-4-\sigma}|f_{\infty}|\right).
\end{equation}

\end{itemize}
\end{lemma}

\section{Construction of the super and subsolution.}\label{TheMeat}

\subsection{The ansatz.}

Define the $L^{\infty}$ weighted norms
\begin{equation*}
 \|f\|_{k,L^{\infty}(\Omega)} = \|(1+r(y)^k)
f(y)\|_{L^{\infty}(\Omega)}
\end{equation*}
for regions $\Omega\subseteq\Gamma$. Use the short-hand
$\|\cdot\|_{k,\infty}$ when $\Omega = \Gamma$.

Recall that for $\alpha>0$ small enough the coordinates
\eqref{Fermialpha} are well-defined in the band
$\mathcal{B}_{\Gamma_{\alpha}} = \mathcal{B}_{\Gamma_{\alpha}}(2)$.
We will work with the following ansatz $w:
\mathcal{B}_{\Gamma_{\alpha}} \rightarrow \real$:
\begin{equation}\label{ansatz}
    w(y,z) = h_0^{\alpha}(y) + z h_1^{\alpha}(y) + z^2 h_2^{\alpha}(y) + z^3 h_3^{\alpha}(y)
    +z^4  h_4^{\alpha}(y) + z^5 h_5^{\alpha}(y),
\end{equation}
where $h_i^{\alpha} = h_i^{\alpha}(y)$ are functions on
$\mathcal{B}_{\Gamma_{\alpha}}$, independent of the $z$-variable.
The coefficients $h_1^{\alpha}, h_3^{\alpha}, h_5^{\alpha}$ are
explicitly specified in terms of geometric quantities associated
with $\Gamma_{\alpha}$:
\begin{align}
  & h_1^{\alpha} = 1 - \frac{|A_{\alpha}|^2}{2} +  h_1'^{\alpha} \qquad h_1'^{\alpha} = - \frac{5}{24} (\Delta_{\Gamma_{\alpha}} + |A_{\alpha}|^2)|A_{\alpha}|^2 - \frac{H_{4,\alpha}}{4} \notag \\
  & h_3^{\alpha} = \frac{1}{6}(|A_{\alpha}|^2 + h_3'^{\alpha}) \qquad
  h_3'^{\alpha} = \frac{1}{2}(\Delta_{\Gamma_{\alpha}}|A_{\alpha}|^2 - |A_{\alpha}|^4)
  \\
  & h_5^{\alpha} = \frac{1}{20} \Big(\frac{|A_{\alpha}|^4}{2}+H_{4,\alpha} - \frac{\Delta_{\Gamma_{\alpha}}|A_{\alpha}|^2}{6}
  \Big).\notag
\end{align}
According to \eqref{Hl_scaling} of Lemma \ref{lemma_scaling}, the
size of $h = (h_1^{\alpha}-1), h_3^{\alpha}$ and their covariant
derivatives up to second order on $\Gamma$,
\begin{equation}\label{h1-1h3}
h = O\Big(\frac{\alpha^2}{1+(\alpha r)^2}\Big),\quad
|D_{\Gamma_{\alpha}}h| = O\Big(\frac{\alpha^3}{1+(\alpha
    r)^3}\Big), \quad |D^2_{\Gamma_{\alpha}}h| = O\Big(\frac{\alpha^4}{1+(\alpha
    r)^4}\Big),
\end{equation}
while for $h=h_1'^{\alpha}, h_3'^{\alpha}, h_5^{\alpha}$
\begin{equation}\label{h1'h3'h5}
    h = O\Big(\frac{\alpha^4}{1+(\alpha r)^4}\Big),\quad |D_{\Gamma_{\alpha}}h| = O\Big(\frac{\alpha^5}{1+(\alpha
    r)^5}\Big), \quad |D^2_{\Gamma_{\alpha}}h| = O\Big(\frac{\alpha^6}{1+(\alpha
    r)^6}\Big).
\end{equation}

We set $h_4^{\alpha} = 0$. The coefficients $h_0^{\alpha}$ and
$h_2^{\alpha}$ will be specified later so that the ansatz meets the
supersolution conditions in Definition \ref{defClaSuper}, but from
the very start we will require that they satisfy the following
properties:
\begin{itemize}
\item $h_0^{\alpha}>0$ is \emph{strictly positive} and scales like
\begin{equation*}
    h_{0}^{\alpha}(y) = \alpha^p h_0(\alpha y),
\end{equation*}
where $0<p < 1$ and $h_0\in C^{2}(\Gamma)$ is positive with
\begin{equation}\label{sizeh0}
    \|D^2_{\Gamma}h_0\|_{3,\infty}  + \|D_{\Gamma}h_0\|_{2,\infty} + \|h_0\|_{1,\infty} \leq C_1.
\end{equation}
for some positive constant $C_1$. So,
\begin{equation*}
    h_0^{\alpha} = O\Big(\frac{\alpha^{p}}{1+(\alpha r)}\Big),\quad
|D_{\Gamma_{\alpha}}h_0^{\alpha}| =
O\Big(\frac{\alpha^{1+p}}{1+(\alpha
    r)^2}\Big), \quad |D^2_{\Gamma_{\alpha}}h_0^{\alpha}| = O\Big(\frac{\alpha^{2+p}}{1+(\alpha
    r)^3}\Big).
\end{equation*}

\item $h_2^{\alpha}$ equals 
\begin{equation*}
    h_2^{\alpha} = \frac{1}{2}(|A_{\alpha}|^2h_0^{\alpha} +
    h_2'^{\alpha}),
\end{equation*}
where the correction $h_2'^{\alpha}$ scales like
\begin{equation*}
    h_2'^{\alpha}(y) = \alpha^{2+p}h_2'(\alpha y)
\end{equation*}
and $h_2'\in C^2(\Gamma)$ is \emph{positive} with
\begin{equation}\label{sizeh2'}
    \|D^2_{\Gamma}h_2'\|_{5,\infty}  + \|D_{\Gamma}h_2'\|_{4,\infty} + \|h_2'\|_{3,\infty}\leq C_2,
\end{equation}
for some positive constant $C_2$. Thus,
\begin{equation*}
    h_2^{\alpha} = O\Big(\frac{\alpha^{2+p}}{1+(\alpha r)^3}\Big),\quad
|D_{\Gamma_{\alpha}}h_2^{\alpha}| =
O\Big(\frac{\alpha^{3+p}}{1+(\alpha
    r)^4}\Big), \quad |D^2_{\Gamma_{\alpha}}h_2^{\alpha}| = O\Big(\frac{\alpha^{4+p}}{1+(\alpha
    r)^5}\Big).
\end{equation*}

\end{itemize}

\begin{remark}\label{forcedchoi}
At first look, the choices for $h^{\alpha}_{i}$ above may seem
somewhat arbitrary but they are prompted by the supersolution
conditions. The fact that we expect the solution to behave
asymptotically like $z$ suggests that $h_1^{\alpha} \approx 1$ to
main order. Thus, the main order term in
$H_{\Gamma_{\alpha}(z)}\del_z w$ is $ z |A_{\alpha}|^2$ which has to
be cancelled by the $z^1$-term in $\del^2_z w$: thus, $h_3 \approx
\frac{|A_{\alpha}|^2}{6}$. Now $\del_z w \approx h_1^{\alpha} +
z^2\frac{|A_{\alpha}|^2}{2}$ and since $w$ achieves values $\pm 1$
at $z\approx \pm 1$ and $|\nabla w| \approx \del_z w$, the
supersolution gradient condition demands that we refine
$h_1^{\alpha}$ to equal $h_1^{\alpha}\approx 1 -
\frac{|A_{\alpha}|^2}{2}$. The form of $h_2^{\alpha}$ is contingent
upon the fact that $w (y, \cdot)$ attains the values $\pm 1$
asymptotically at $z_{\pm} \approx \pm 1 - h_0^{\alpha}$, so that
\[
\del_z w(y, z_{\pm}) \approx 1 \pm( 2h_2^{\alpha} - |A_{\alpha}|^2
h_0^{\alpha}),
\]
requiring the positivity of $h_2'^{\alpha} = 2h_2^{\alpha} -
|A_{\alpha}|^2 h_0^{\alpha}$. The remaining choices (and further
refinements) are made so that no terms that decay at a rate $r^{-4}$
(and no better) at infinity are present in the expansions of $\Delta
w$ or $\del_z w(y,z_{\pm})$.

All this will become transparent once we carry out the computations
of the Laplacian of $w$ in Lemma \ref{lemma_lap} and of the gradient
of $w$ on $\{w=\pm 1\}$ in Lemmas \ref{lemma_dzw} and \ref{supGrad}
below.
\end{remark}

\paragraph{NB} In what follows the constants in the $O$-notation
depend solely on $p$, $C_1$, $C_2$ and the minimal graph $\Gamma$,
but \textbf{not} on the scaling parameter $\alpha$.

\begin{lemma}\label{lemma_lap} The Laplacian of $w$ in
$\mathcal{B}_{\Gamma_{\alpha}}$ can be estimated by
\begin{align}
\Delta w(y,z) & = (\Delta_{\Gamma_{\alpha}(z)} +
|A_{\alpha}|^2)h_0^{\alpha} + h_2'^{\alpha}- z^2H_{3,{\alpha}} +
O\Big(\frac{\alpha^{4+p}}{1+(\alpha r)^5}\Big) \label{lap_ansatz}.
\end{align}
\end{lemma}
\begin{proof}
Compute in succession:
\begin{align}
    \del_z w & = 1 - \frac{|A_{\alpha}|^2}{2} + h_1'^{\alpha} + z(|A_{\alpha}|^2 h_0^{\alpha} +
    h_2'^{\alpha})+ z^2 \frac{|A_{\alpha}|^2 + h_3'^{\alpha}}{2} +
    5z^4 h_5^{\alpha} \label{delz}\\
    \del_z^2 w & = |A_{\alpha}|^2 h_0^{\alpha} + h_2'^{\alpha}+ z (|A_{\alpha}|^2 + h_3'^{\alpha}) +
    20 z^3 h_5^{\alpha} \label{del2z} \\
    H_{\Gamma_{\alpha}(z)} \del_z w & =  \left(z |A_{\alpha}|^2 + z^2 H_{3, \alpha} + z^3 H_{4, \alpha} +
O\Big(z^4\frac{\alpha^5}{1+(\alpha r)^5}\Big)\right) \del_z w = \notag \\
& = z \left(|A_{\alpha}|^2 - \frac{|A_{\alpha}|^4}{2} +
O\Big(\frac{\alpha^6}{1+ (\alpha r)^6}\Big)\right) +
z^2\left(H_{3,\alpha} + O\Big(\frac{\alpha^{4+p}}{1 + (\alpha
r)^5}\Big)\right)  \label{Hdelz} \\
& + z^3 \left(\frac{|A_{\alpha}|^4}{2}+ H_{4,\alpha} +
O\Big(\frac{\alpha^{5+p}}{1 + (\alpha r)^{6}}\Big)\right) +
O\Big(z^4 \frac{\alpha^{5}}{1+ (\alpha r)^5}\Big) \notag
\end{align}
Because of \eqref{DeltaCompAlpha} and \eqref{h1'h3'h5},
\begin{align*}
    |\Delta_{\Gamma_{\alpha(z)}} h_1'^{\alpha}| + & |\Delta_{\Gamma_{\alpha(z)}}
    h_3'^{\alpha}| + |\Delta_{\Gamma_{\alpha(z)}} h_5^{\alpha}|  =
    O\Big(\frac{\alpha^6}{1+(\alpha r)^6}\Big) \\
    &\Delta_{\Gamma_{\alpha(z)}} h_2^{\alpha} = O\Big(\frac{\alpha^{4+p}}{1+(\alpha
    r)^5}\Big) \\
    &\Delta_{\Gamma_{\alpha(z)}} |A_{\alpha}|^2  = \Delta_{\Gamma_{\alpha}}
    |A_{\alpha}|^2 + O\Big(z \frac{\alpha^5}{1 + (\alpha r)^5}\Big)
\end{align*}
so that
\begin{equation}\label{lap}
    \Delta_{\Gamma_{\alpha}(z)} w =
    \Delta_{\Gamma_{\alpha(z)}}h_0^{\alpha} - z \frac{\Delta_{\Gamma_{\alpha}}
    |A_{\alpha}|^2}{2} + z^3 \frac{\Delta_{\Gamma_{\alpha}}
    |A_{\alpha}|^2}{6} + O\Big(z\frac{\alpha^{4+p}}{1+(\alpha
    r)^5}\Big).
\end{equation}
Combining \eqref{del2z}, \eqref{Hdelz} and \eqref{lap} we derive
that in $\mathcal{B}_{\Gamma_{\alpha}}$
\begin{align*}
    \Delta w & = (\Delta_{\Gamma_{\alpha}(z)} +
    |A_{\alpha}|^2)h_0^{\alpha} + h_2'^{\alpha} - z^2 H_{3,\alpha} \\
    & + z\left(-\frac{\Delta_{\Gamma_{\alpha}}|A_{\alpha}|^2}{2} +
    |A_{\alpha}|^2 + h_3'^{\alpha} - |A_{\alpha}|^2 +
    \frac{|A_{\alpha}|^4}{2}\right) + \\
    &  + z^3 \left(\frac{\Delta_{\Gamma_{\alpha}}
    |A_{\alpha}|^2}{6} + 20 h_5^{\alpha} - \frac{|A_{\alpha}|^4}{2}- H_{4,\alpha}\right) = \\
    & = (\Delta_{\Gamma_{\alpha}(z)} +
    |A_{\alpha}|^2)h_0^{\alpha} + h_2'^{\alpha} - z^2 H_{3,\alpha} +
    O\Big(\frac{\alpha^{4+p}}{1+(\alpha
    r)^5}\Big).
\end{align*}
\end{proof}

Now we would like to determine how far the level surfaces $\{w=\pm
1\}$ stand from the graph $\Gamma_{\alpha}$. Note that for $|z|\leq
2$ and uniformly in $y\in \Gamma_{\alpha}$, we have $w = z +
O(\alpha^p)$ and $\del_z w = 1 + O(\alpha^2)$. Thus for all small
enough $\alpha>0$, $w(y,\cdot)$ is strictly increasing and attains
the values $\pm 1$ for unique $z_{\pm}(y)$ with $|z_{\pm}(y)|\leq
2$.

\begin{lemma}\label{lemma_dzw} For all small enough $\alpha > 0$ (so that $z_{\pm}$ is
well-defined),
\begin{equation*}
   \del_z w(y, z_{\pm}(y)) = 1 \pm h_2'^{\alpha} - \frac{|A_{\alpha}|^2
     (h_0^{\alpha})^2}{2} - h_2'^{\alpha}h_0^{\alpha} +O\Big(\frac{\alpha^{2+3p}}{1 + (\alpha
     r)^5}\Big).
\end{equation*}
\end{lemma}
\begin{proof}
Fix $y\in \Gamma_{\alpha}$ and let us estimate $z_{\pm}(y)$. Write
$z_{\pm} = \pm 1 + \delta_{\pm}$, where $\delta_{\pm} = o(1)$ as
$\alpha \to 0$. We compute
\begin{align*}
    w(y,z_{\pm}) & = h_0^{\alpha} +
    (1-\frac{|A_{\alpha}|^2}{2})(\pm 1 +\delta_{\pm}) + h_2^{\alpha}(\pm 1 +\delta_{\pm})^2
    + \\ & +
    \frac{|A_{\alpha}|^2}{6}(\pm 1 +\delta_{\pm})^3 + O\Big(\frac{\alpha^{4}}{1+(\alpha
    r)^4}\Big) \\
    & = h_0^{\alpha} \pm (1 - \frac{|A_{\alpha}|^2}{2}) + h_2^{\alpha}
    \pm
 \frac{|A_{\alpha}|^2}{6} + \\
 & + \delta_{\pm}(1 \pm 2 h_2^{\alpha}) + \delta_{\pm}^2 (\pm \frac{|A_{\alpha}|^2}{2} + h_2^{\alpha}) +
 \delta_{\pm}^3\frac{|A_{\alpha}|^2}{6} +O\Big(\frac{\alpha^{4}}{1+(\alpha
    r)^4}\Big).
\end{align*}
Thus,
\begin{equation*}
    h_0^{\alpha} \mp \frac{|A_{\alpha}|^2}{3} + h_2^{\alpha} +
    \delta_{\pm}(1 \pm 2h_2^{\alpha}) + \delta_{\pm}^2(\frac{|A_{\alpha}|^2}{2} +
    h_2^{\alpha}) + \delta_{\pm}^3 \frac{|A_{\alpha}|^2}{6} = O\Big(\frac{\alpha^{4}}{1+(\alpha
    r)^4}\Big),
\end{equation*}
so that
\begin{equation*}
    \delta_{\pm} = O\Big(h_0^{\alpha} + h_2^{\alpha} +
    \frac{|A_{\alpha}|^2}{3}\Big) = O\Big(\frac{\alpha^p}{1+\alpha
    r}\Big),
\end{equation*}
which in turn implies
\begin{equation}\label{deltapm}
    \delta_{\pm} = - h_0^{\alpha} \pm \frac{|A_{\alpha}|^2}{3} -
    h_2^{\alpha} + O\Big(\frac{\alpha^{2+2p}}{1+(\alpha
    r)^4}\Big),
\end{equation}
We can now estimate $\del_z w(y, z_{\pm}(y))$:
\begin{align*}
    \del_z w (y,z_{\pm}(y)) & = 1 - \frac{|A_{\alpha}|^2}{2} +
    h_1'^{\alpha} + 2 h_2^{\alpha} (\pm 1 - h_0^{\alpha})  + \\ & + \frac{|A_{\alpha}|^2+h_3'^{\alpha}}{2}(\pm 1 - h_0^{\alpha}
     \pm \frac{|A_{\alpha}|^2}{3})^2 + 5 h_5^{\alpha} + O\Big(\frac{\alpha^{2+3p}}{1 + (\alpha
     r)^5}\Big) = \\
     & = 1  + ( \pm 2
     h_2^{\alpha} \mp |A_{\alpha}|^2
     h_0^{\alpha}) + (- 2 h_2^{\alpha}h_0^{\alpha} + \frac{|A_{\alpha}|^2 (h_0^{\alpha})^2}{2})  + \\
     & + (h_1'^{\alpha} + \frac{h_3'^{\alpha}}{2} + \frac{|A_{\alpha}|^4}{3} + 5h_5^{\alpha}) + O\Big(\frac{\alpha^{2+3p}}{1 + (\alpha
     r)^5}\Big)\\
     & = 1 \pm h_2'^{\alpha} - \frac{|A_{\alpha}|^2
     (h_0^{\alpha})^2}{2} - h_2'^{\alpha}h_0^{\alpha} +O\Big(\frac{\alpha^{2+3p}}{1 + (\alpha
     r)^5}\Big).
\end{align*}

\end{proof}
Straightforward derivative estimates using \eqref{GradCompAlpha}
yield
\begin{lemma}\label{supGrad}
We have
\begin{equation*}
    |\nabla_{\Gamma_{\alpha}(z_{\pm})}w|^2 = |\nabla_{\Gamma_{\alpha}(z_{\pm})}h_0^{\alpha}|^2 + O\Big(\frac{\alpha^{4+p}}{1+(\alpha
    r)^5}\Big)
\end{equation*}
and thus
\begin{align*}
    |\nabla w|^2(y,z_{\pm}(y)) & = (\del_z w)^2 +
    |\nabla_{\Gamma_{\alpha}(z_{\pm})}w|^2 = \\
    & = 1 \pm 2 h_2'^{\alpha}(1 + O(\alpha^p)) -
    |A_{\alpha}|^2(h_0^{\alpha})^2 + |\nabla_{\Gamma_{\alpha}(z_{\pm})}h_0^{\alpha}|^2 + O\Big(\frac{\alpha^{2+3p}}{1+(\alpha
    r)^5}\Big).
\end{align*}
\end{lemma}
Our ansatz has the very nice, extra feature that it is strictly
increasing in $\mathcal{B}_{\Gamma_{\alpha}}$ in the direction of
$e_9$.
\begin{lemma}\label{ansmonotoni}
For all $\alpha$ small enough
\begin{equation*}
    \del_{x_9} w  > 0 \qquad \text{in} \quad
    \mathcal{B}_{\Gamma_{\alpha}}.
\end{equation*}
\end{lemma}
\begin{proof}
Computing in the coordinates \eqref{Fermialpha}
\begin{align}
    \del_{x_9} w(y,z) & = \nabla w \cdot e_9 = (\nabla_{\Gamma_{\alpha}(z)}w + (\del_z w)
    \nu(y))\cdot e_9 \notag \\
    & \geq \frac{\del_z w}{\sqrt{1+|\nabla F_{\alpha}}|^2} -
    |\nabla_{\Gamma_{\alpha}(z)}w| \notag \\
    & =  \frac{1 + O(\alpha^2)}{\sqrt{1+|\nabla F_{\alpha}}|^2} + O\Big(\frac{\alpha^{1+p}}{1 + (\alpha
    r)^2}\Big) \label{w_mono}.
\end{align}
Since
\[
    \frac{1}{\sqrt{1+|\nabla F_{\alpha}}|^2} \geq \frac{c}{1+(\alpha
    r)^2}
\]
for some positive constant $c>0$ (see Remark $8.2$ in \cite{PKW}),
\eqref{w_mono} yields
\[\del_{x_9} w >0\] for all small enough $\alpha >0$.
\end{proof}

\subsection{Supersolutions for the Jacobi
operator.}\label{subsec_Supersols}

As we have noticed from Lemma \ref{lemma_lap}, the sign of $\Delta
w$ depends crucially on whether the Jacobi operator
\begin{equation*}
    J_{\Gamma_{\alpha}} = \Delta_{\Gamma_{\alpha}} +
    |A_{\alpha}|^2
\end{equation*}
admits positive supersolutions that satisfy appropriate differential
inequalities. We will show that the Jacobi operator $J_{\Gamma}$ on
the (non-rescaled) minimal graph $\Gamma$ admits the following two
types of smooth supersolutions:

\begin{itemize}
\item \emph{Type 1} is a positive supersolution $h\in C^{2}(\Gamma)$ such that for some $0<\eps<1$
\begin{equation*}
    J_{\Gamma} h(y) \leq - \frac{1}{1 + r^{4+\eps}(y)};
\end{equation*}

\item \emph{Type 2} is a positive supersolution $h\in C^2(\Gamma)$ such that
\begin{equation*}
    J_{\Gamma} h(y) \leq -\frac{|\theta(y) - \pi/4|}{1 + r^3(y)}.
\end{equation*}
\end{itemize}

The Type 1 supersolution is readily provided by \cite[Proposition
4.2(b)]{PKW} (our Proposition \ref{h'} below is a straightforward
modification). We construct the Type 2 supersolution in Proposition
\ref{h"} and the supporting Lemma \ref{h"_infty}.

\begin{prop}\label{h'}
Let $0<\eps<1$. There exists a \textbf{positive} function $h\in
C^2(\Gamma)$ such that
\[
\|D^2_{\Gamma}h\|_{4+\eps,\infty} + \|D_{\Gamma}h\|_{3+\eps,\infty}
+ \|h\|_{2+\eps,\infty} < \infty
\]
and
\begin{equation*}
    J_{\Gamma} h \leq  - \frac{1}{1+ r^{4+\eps}}.
\end{equation*}
\end{prop}

\begin{prop}\label{h"}
There exists a \textbf{non-negative} function $h\in C^2(\Gamma)$
such that
\begin{equation*}
\|D^2_{\Gamma}h\|_{3,\infty} + \|D_{\Gamma}h\|_{2,\infty}
  +  \|h\|_{1,\infty} < \infty
\end{equation*}
and
\begin{equation*}
    J_{\Gamma} h \leq - \frac{|\theta - \pi/4|}{1+ r^3}.
\end{equation*}
Moreover, there is a $\frac{1}{2} < \tau < \frac{2}{3}$ (e.g. $\tau
= \frac{5}{8}$) such that for every $0\leq \delta < \delta'\leq
\frac{3}{2}$
\begin{equation}\label{refinedecay}
    \|h\|_{1+\delta\tau, L^{\infty}(S(-\delta'))} + \|D_{\Gamma}h\|_{2+\delta\tau, L^{\infty}(S(-\delta'))} + \|D^2_{\Gamma}h\|_{3+\delta\tau, L^{\infty}(S(-\delta'))} < \infty
\end{equation}
where $S(-\delta') = \{|\theta-\frac{\pi}{4}| \leq
(1+r)^{-\delta'}\} \subset \Gamma.$
\end{prop}

Before we venture into proving these two propositions, recall that
\begin{align*}
    J_{\Gamma} h & \leq f \quad \text{in the weak sense if} \\
    (J_{\Gamma} h -f )[\phi] & \leq 0 \quad \text{for all
    non-negative}~\phi\in C^{1}_c(\Gamma).
\end{align*}
Above we have used the notation
\begin{align*}
    f[\phi] &= \int_{\Gamma}f\phi \qquad \text{for}\quad f\in L^1_{\text{loc}}(\Gamma) \\
    (J_{\Gamma}h)[\phi] &= \int_{\Gamma} - \nabla_{\Gamma}h\cdot
    \nabla_{\Gamma} \phi + |A|^2\phi
\end{align*}
where test functions $\phi \in C^1_c(\Gamma)$.

Let us make the important remark that the operator $J_{\Gamma}$
satisfies the maximum principle.

\begin{remark}[Maximum principle for $J_{\Gamma}$]\label{maxPrin}
Since $h_0 := \frac{1}{\sqrt{1+|\nabla F|^2}} > 0$ solves
$J_{\Gamma}h_0 = 0$ (see \eqref{jacoandLinMCmini}), the elliptic
operator
\begin{equation*}
     L := h_0 \Delta_{\Gamma} + 2\nabla_{\Gamma}h_0\cdot \nabla_{\Gamma}
\end{equation*}
satisfies
\begin{equation*}
    J_{\Gamma} h = L(h/h_0).
\end{equation*}
Thus, if $h$ is a supersolution for $J_{\Gamma}$ (in the weak sense)
in a bounded domain $U\subset \Gamma$ and $h\in C(\overline{U})$,
\begin{equation*}
    0\geq J_{\Gamma}h = L(h/h_0) \qquad \text{in} \quad U,
\end{equation*}
so that the quotient $h/h_0$ doesn't achieve its minumum at an
interior point of $U$ unless $h/h_0$ is constant in $U$.
\end{remark}

In fact, we will construct the supersolutions in Propositions
\ref{h'} and \ref{h"} as solutions to appropriate elliptic
differential equations rather than inequalities. This approach will
pay off, because in the end we will automatically possess global
smooth supersolutions, whose first and second derivatives will have
the appropriate decay rates at infinity. Specifically, we will
investigate the linear problem
\begin{equation}\label{linearProb}
    J_{\Gamma} h = f \qquad \text{in}\quad \Gamma,
\end{equation}
where $h$ and $f$ are in appropriately weighted H\"older-type
spaces. As usual, we first study the problem \eqref{linearProb} in
bounded domains $\Gamma_{R} := \Gamma \cap \{r<R\}$
\begin{equation}\label{linearProbbounded}
\begin{aligned}
    J_{\Gamma} h_R  & = f \qquad \text{in}\quad \Gamma_R \\
    h_R & = 0 \qquad \text{on} \quad \del \Gamma_R.
\end{aligned}
\end{equation}
Because $J_{\Gamma}$ satisfies the maximum principle, the problem
\eqref{linearProbbounded} is uniquely solvable for all $R$. In order
then to run a compactness argument which takes a sequence $h_{R_n}$,
$R_n\nearrow\infty$ and produces a globally-defined $h:\Gamma \to
\real$ that solves \eqref{linearProb}, we need two important
ingredients -- the existence of suitable global barrier functions
and an a priori estimate (Lemma \ref{apriori_lemma}) for the
solution to \eqref{linearProbbounded}.

We first exhibit functions that are (weak) supersolutions for
$J_{\Gamma}$ far away from the origin. Later, we will be able to
modify and extend them to barrier functions on the whole of
$\Gamma$.


\begin{lemma}[cf. Lemma 7.2 in \cite{PKW}]\label{h'infty}
Let $0<\eps < 1$. There exists a positive function $h$, such that
for some $R>0$ and constants $c, C>0$
\begin{align}
    J_{\Gamma} h(y) & \leq -\frac{1}{1+ r^{4+\eps}} \qquad \text{in}\quad \{r(y)>R\} \label{h'inftydifineq}\\
    \frac{c}{1+ r^{2+\eps}} \leq h(y) & \leq \frac{C}{1+ r^{2+\eps}} \qquad
\text{in}\quad \{r(y) > R\}. \label{h'inftysize}
\end{align}
\end{lemma}
\begin{proof}
It follows from the existence of the Type 1--supersolution
$h_{1,\infty} \in C^2(\Gamma_{\infty})$ for $J_{\Gamma_{\infty}}$
far away from the origin $\{r>R\}$ (See \eqref{diffen1} of Appendix
\ref{append1}):
\begin{equation*}
    J_{\Gamma_{\infty}} h_{1,\infty}(\tilde{y}) \leq -\frac{1}{1+
    r(\tilde{y})^{4+\eps}} \qquad \tilde{y} \in \Gamma_{\infty}\cap \{r>R\}.
\end{equation*}
Use the orthogonal projection $\pi_{\Gamma}$ to lift $h_{1,\infty}$
to a function $h_1$ on $\Gamma\cap \{r>R\}$
\begin{equation*}
    h_1 \circ \pi_{\Gamma} = h_{1,\infty} \qquad \text{on} \quad
    \Gamma \cap \{r > R\}.
\end{equation*}
Then according to \eqref{Jacobiclo} and the gradient and hessian
estimates in Lemma \ref{gradhess} (see Appendix \ref{append2}),
\begin{align*}
    J_{\Gamma} h(\pi_{\Gamma}(\tilde{y})) & =
    J_{\Gamma_{\infty}}h_{1,\infty}(\tilde{y}) + \\ & +
    O\left(r^{-2-\sigma}|D_{\Gamma_{\infty}}^2 h_{1,\infty}| +
    r^{-3-\sigma} |D_{\Gamma_{\infty}} h_{1,\infty}| +
    r^{-4-\sigma}|h_{1,\infty}|\right)(\tilde{y}) \\
    & \leq -\frac{1}{2(1+r^{4+\eps}(\tilde{y}))}
\end{align*}
for $r(\tilde{y})>R$ large enough. Equations \eqref{h'inftydifineq}
and \eqref{h'inftysize} are obtained once we note that, according to
Lemma \ref{lemma_close}, $y :=\pi_{\Gamma}(\tilde{y})$ is very close
to $\tilde{y}$ :
\begin{equation*}
    |y - \tilde{y}| = O(r^{-1-\sigma}(y))
\end{equation*}
in $\{r>R\}$ for a large enough $R$.

\end{proof}
\begin{lemma}\label{h"_infty}
There exists a locally Lipschitz, non-negative function $h$, which
is a weak supersolution for $J_{\Gamma}$ away from the origin and
which satisfies
\begin{equation}\label{weakh_asymp}
    J_{\Gamma} h(y) \leq -\frac{|\theta - \pi/4|}{1 + r^3} \qquad\text{on}\quad
\{r(y)>r_0\}
\end{equation}
for some large enough $r_0>0$. Moreover,
\begin{equation}\label{weakhbound}
    h = O\Big(\frac{|\theta - \pi/4|^{\tau}}{1+r}+ \frac{1}{1+ r^{2+\eps}}\Big)
\end{equation}
for some $\tau \in (\frac{1}{2},\frac{2}{3})$ and some $\eps\in
(0,1)$ (e.g. $\tau = \frac{5}{8}$ and $\eps = \frac{1}{8}$ do the
job).
\end{lemma}
\begin{proof}
The construction of the weak supersolution in this case is achieved
by patching up two smooth supersolutions, defined on overlapping
regions of $\Gamma$, via the $\min$ operation. The resultant
function is obviously locally Lipschitz.

One of the building blocks is the Type 2 supersolution for
$J_{\Gamma_{\infty}}$ at infinity \eqref{differen2}-- call it
$\tilde{h}_{\text{ext}}\in C^2(\Gamma_{\infty}\cap \{\pi/4 < \theta
\leq \pi/2\})$ here:
\begin{equation*}
    \tilde{h}_{\text{ext}}(\tilde{y}) = \frac{rq_2(\theta(\tilde{y}))}{\sqrt{1+|\nabla F_{\infty}|^2}}
\end{equation*}
where $q_2(\theta)$ has the following expansion near $\theta =
\frac{\pi}{4}$:
\begin{equation*}
    q_2(\theta) = (\theta-\frac{\pi}{4})^{\tau}(a_0 + a_2(\theta-\frac{\pi}{4})^2 +
\cdots) \qquad a_0>0,
\end{equation*}
and
\begin{equation*}
    J_{\Gamma_{\infty}} \tilde{h}_{\text{ext}} \leq - \frac{(\theta - \pi/4)^{\tau}}{1+ r^{3}} \qquad
    \theta\in(\frac{\pi}{4},\frac{\pi}{2})\quad \text{and} \quad r>r_0.
\end{equation*}
for some large $r_0>0$.
Define for
\begin{equation}\label{cond1}
    -2\leq \alpha_2 < \alpha_1 < 0
\end{equation}
the following subregions of the model graph $\Gamma_{\infty}$
\begin{align*}
    \Gamma_{\infty,\text{int}} & = \{|\theta - \frac{\pi}{4}| <
r^{\alpha_1}\} \cap \{r > r_0\} \subset \Gamma_{\infty} \\
\Gamma_{\infty,\text{ext}} & = \{|\theta - \frac{\pi}{4}| >
r^{\alpha_2}\}\cap \{r > r_0\}\subset \Gamma_{\infty}.
\end{align*}
Note that $\Gamma_{\infty,\text{int}}$ and
$\Gamma_{\infty,\text{ext}}$ have a non-empty overlap and that they
cover all of $\Gamma_{\infty}\cap\{r>r_0\}$. Define
$\tilde{h}_{\text{int}}\in C^2(\Gamma_{\infty})$ by
\begin{equation}\label{gamma0hint}
\tilde{h}_{\text{int}} = \frac{(\theta -
\pi/4)^2r^{3+\delta}}{\sqrt{1+|\nabla F_{\infty}|^2}},
\end{equation}
for some $0<\delta < 1$, which will be specified later, and extend
$\tilde{h}_{\text{ext}}$ on the whole of
$\Gamma_{\infty}\cap\{r>r_0\}$ so that it is even about $\theta =
\frac{\pi}{4}$:
\begin{equation*}
    \tilde{h}_{\text{ext}}(r,\theta)
=\tilde{h}_{\text{ext}}(r,\frac{\pi}{2}-\theta).
\end{equation*}
The goal is that the lifts of
$\tilde{h}_{\text{int}},\tilde{h}_{\text{ext}}$ onto $\Gamma$:
\begin{equation*}
    h_{\text{int}}(\pi_{\Gamma}(\tilde{y})) = \tilde{h}_{\text{int}}(\tilde{y}) \qquad
    h_{\text{ext}}(\pi_{\Gamma}(\tilde{y})) = \tilde{h}_{\text{ext}}(\tilde{y}) \qquad \tilde{y} \in \Gamma_{\infty}
\end{equation*}
corrected by an appropriate asymptotic supersolution of Type 1
(given by the previous Lemma \ref{h'infty}), will satisfy the
desired differential inequality \eqref{weakh_asymp} in the
respective regions
\begin{equation*}
    \Gamma_{\text{int}} = \pi_{\Gamma}(\Gamma_{\infty,\text{int}})\subset \Gamma \quad \text{and} \quad \Gamma_{\text{ext}} =
\pi_{\Gamma}(\Gamma_{\infty,\text{ext}}) \subset \Gamma.
\end{equation*}

Applying the gradient and hessian estimates of Lemma \ref{gradhess}
from Appendix \ref{append2} and the fact that $-2\leq
\alpha_2<\alpha_1$, we derive that in $\Gamma_{\infty,\text{int}}$
\begin{align*}
    |\tilde{h}_{\text{int}}| & = O(r^{1+\delta+2\alpha_1}) \\
    |D_{\Gamma_{\infty}}\tilde{h}_{\text{int}}| & = O\Big((\theta-\frac{\pi}{4})^2r^{\delta} +
r^{-2+\delta}|\theta-\frac{\pi}{4}|\Big) = O(r^{\delta+2\alpha_1}+ r^{\alpha_1+\delta-2}) = O(r^{\delta+2\alpha_1})\\
    |D^2_{\Gamma_{\infty}}\tilde{h}_{\text{int}}| & =
O\Big((\theta-\frac{\pi}{4})^2r^{-1+\delta} +
r^{-3+\delta}|\theta-\frac{\pi}{4}| + r^{-5+\delta}\Big) =
O(r^{-1+\delta+2\alpha_1}).
\end{align*}
On the other hand, $\tilde{h}_{\text{ext}}$ satisfies in
$\Gamma_{\infty,\text{ext}}$
\begin{align*}
    |\tilde{h}_{\text{ext}}| & = O(|\theta-\frac{\pi}{4}|^{\tau}r^{-1})\\
    |D_{\Gamma_{\infty}}\tilde{h}_{\text{ext}}| & = O\Big(|\theta-\frac{\pi}{4}|^{\tau}r^{-2} +
r^{-4}|\theta-\frac{\pi}{4}|^{\tau-1}\Big) = O(|\theta
-\frac{\pi}{4}|^{\tau} r^{-2}) \\
|D^2_{\Gamma_{\infty}}\tilde{h}_{\text{ext}}| &
=O\Big(|\theta-\frac{\pi}{4}|^{\tau}r^{-3} +
r^{-5}|\theta-\frac{\pi}{4}|^{\tau-1} +
r^{-7}|\theta-\frac{\pi}{4}|^{\tau-2}\Big) = O(|\theta
-\frac{\pi}{4}|^{\tau} r^{-3}).
\end{align*}
Thus, the proximity \eqref{Jacobiclo} between $J_{\Gamma}$ and
$J_{\Gamma_{\infty}}$ implies
\begin{align*}
    J_{\Gamma}h_{\text{int}}(\pi_{\Gamma}(\tilde{y})) & = O(r^{-1+\delta+2\alpha_1}(\tilde{y})) \qquad \tilde{y}\in \Gamma_{\infty,\text{int}}\\
    J_{\Gamma} h_{\text{ext}}(\pi_{\Gamma}(\tilde{y})) & \leq
-\frac{|\theta(\tilde{y}) - \pi/4|^{\tau}}{1+r^3(\tilde{y})} +
O\Big(|\theta(\tilde{y})
-\pi/4|^{\tau} r^{-5-\sigma}(\tilde{y})\Big)\\
& \leq -\frac{1}{2}\frac{|\theta(\tilde{y}) -
\pi/4|^{\tau}}{1+r^3(\tilde{y})}\leq - \frac{C|\theta(\tilde{y}) -
\pi/4|}{1+r^3(\tilde{y})} \qquad \tilde{y}\in
\Gamma_{\infty,\text{ext}}
\end{align*}
for large enough $r(\tilde{y})>r_0$. According to Lemma
\ref{lemma_close}, if $y = \pi_{\Gamma}(\tilde{y})$,
\begin{equation*}
    |\tilde{y}-y| = O(r(y)^{-1-\sigma})
\end{equation*}
for some $0<\sigma <1$ and $r(y)>r_0$ large enough. Therefore, the
pair $(r(\tilde{y}),\theta(\tilde{y}))$ is asymptotically equal to
$(r(y),\theta(y))$:
\begin{equation*}
    |r(y)-r(\tilde{y})| = O(r^{-1-\sigma}(y)) \qquad |\theta(y)-\theta(\tilde{y})| =
O(r^{-2-\sigma}(y)).
\end{equation*}
Thus,
\[
\frac{|\theta(\tilde{y}) - \pi/4|}{1+r^3(\tilde{y})} =
\frac{|\theta(y) - \pi/4|}{1+r^3(y)} + O(r(y)^{-5-\sigma})
\]
so that
\begin{align}
J_{\Gamma}h_{\text{int}}(y) + \frac{C}{2}\frac{|\theta(y) -
\pi/4|}{1+r^3(y)} & =
O(r^{-1+\delta+2\alpha_1}(y)+r^{\alpha_1-3}(y)) = \notag \\
& = O(r^{-1+\delta+2\alpha_1}(y))
\qquad y\in \Gamma_{\text{int}} \label{int}\\
J_{\Gamma} h_{\text{ext}}(y) + \frac{C}{2}\frac{|\theta(y) -
\pi/4|}{1+r^3(y)} & \leq -
\frac{C}{2}\frac{|\theta(y)-\pi/4|}{1+r^3(y)}
+O(r^{-5-\sigma}(y))\quad y\in \Gamma_{\text{ext}}. \label{ext}
\end{align}
Let $h'$ be the supersolution for $J_{\Gamma}$, provided by Lemma
\ref{h'infty}:
\begin{equation}\label{littlehelp}
    J_{\Gamma}h' \leq - \frac{1}{1+r^{4+\eps}} \qquad
r>r_0
\end{equation}
for some $0<\eps<1$ which we'll pick shortly. Below we will define
the functions $h_1$ and $h_2$, which will be supersolutions for
$J_{\Gamma}$ on $\Gamma_{\text{int}}$ and $\Gamma_{\text{ext}}$,
respectively, and patch them into a (weak) supersolution $h$,
defined on $\Gamma \cap \{r>r_0\}$, via the $\min$-operation:
\begin{equation*}
    h = \min(h_1, h_2).
\end{equation*}
In order for the operation to succeed, we have to verify the
following:
\begin{itemize}
\item $h_1 := h_{\text{int}} + h'$ satisfies the differential inequality \eqref{weakh_asymp} in
$\Gamma_{\text{int}}$ for large $r_0$:
\begin{equation*}
    J_{\Gamma} (h_{\text{int}} + h')
+\frac{C}{2}\frac{|\theta - \pi/4|}{1+r^3} \leq 0 \qquad
\text{in}\quad \Gamma_{\text{int}}.
\end{equation*}
Because of \eqref{int} and \eqref{littlehelp}, it suffices
\begin{equation*}
-1 + \delta + 2\alpha_1 < -4-\eps \quad \Leftrightarrow \quad
\alpha_1 < \frac{-3-\eps-\delta}{2}
\end{equation*}

\item $h_2 := h_{\text{ext}} + h'$ satisfies
\eqref{weakh_asymp} in $\Gamma_{\text{ext}}$ for large $r_0$:
\begin{equation}\label{cond2}
    J_{\Gamma} (h_{\text{ext}} + h')
+\frac{C}{2}\frac{|\theta - \pi/4|}{1+r^3} \leq 0 \qquad \text{in}
\quad y\in \Gamma_{\text{ext}}.
\end{equation}
By \eqref{ext} and \eqref{littlehelp} this holds for a sufficiently
large $r_0$.

\item $h_1 < h_2$ on $\Gamma_{\text{int}}\setminus\Gamma_{\text{ext}}$ and
$h_1 > h_2$ in $\Gamma_{\text{ext}}\setminus\Gamma_{\text{int}}$,
i.e. we would like to have $\tilde{h}_{\text{int}} <
\tilde{h}_{\text{ext}}$ on $\Gamma_{\infty,\text{int}}$ and
$\tilde{h}_{\text{int}} > \tilde{h}_{\text{ext}}$ in
$\Gamma_{\infty,\text{ext}}$. This will be the case for large enough
$r_0$ if
\begin{equation}\label{cond3}
    \alpha_2 < -\frac{2+\delta}{2-\tau} <\alpha_1
\end{equation}
\end{itemize}
Collect conditions $\eqref{cond1}$, $\eqref{cond2}$ and
$\eqref{cond3}$ in
\begin{equation}\label{cond}
    -2\leq \alpha_2 <-\frac{2+\delta}{2-\tau} <\alpha_1 <
\frac{-3-\eps-\delta}{2}.
\end{equation}
Moreover, \eqref{cond} needs to be compatible with
\begin{equation*}
    \eps, \delta \in (0,1), \qquad \frac{1}{3} < \tau < \frac{2}{3}.
\end{equation*}
Condition \eqref{cond} is fairly tight, but not void: indeed, for
$\delta = \frac{1}{2}$, $\eps = \frac{1}{8}\in (0,1)$ and $\tau =
\frac{5}{8}\in (\frac{1}{3}, \frac{2}{3})$, we have
\begin{equation*}
    -2\leq \alpha_2 <-\frac{20}{11} < \alpha_1 < -\frac{29}{16}.
\end{equation*}
So setting the parameters appropriately, we can conclude that
\begin{equation}\label{finally}
    h = \left\{\begin{array}{ll}h_1 & \text{in} \quad \Gamma_{\text{int}}\setminus\Gamma_{\text{ext}}\\
    \min(h_1, h_2) & \text{in} \quad
\Gamma_{\text{int}}\cap\Gamma_{\text{ext}} \\
h_2 & \text{in} \quad \Gamma_{\text{ext}} \setminus
\Gamma_{\text{int}}
\\\end{array}\right.
\end{equation}
is a weak, locally Lipschitz, supersolution for $J_{\Gamma}$ in
$r>r_0$ for a large enough $r_0>0$ that satisfies
\eqref{weakh_asymp} and \eqref{weakhbound}.
\end{proof}


The second ingredient is an a priori estimate for the solution $h$
to \eqref{linearProbbounded}. Introduce the H\"older-type norms:
\begin{align*}
 |f|_{C^{\gamma}(\Omega)} & = \sup_{y_1\neq y_2 \in \Omega}\left|\frac{f(y_1)-f(y_2)}{\text{dist}_{\Gamma}(y_1, y_2)^{\gamma}} \right| \\
    \|f\|_{k,C^{\gamma}(\Omega)} & = \|f\|_{k,L^{\infty}(\Omega)} +
\|(1+r^{k+\gamma}) |f|_{C^{\gamma}(\mathcal{C}_{\beta(1+ r)}(y)\cap
\Omega)} \|_{L^{\infty}(\Omega)}
\end{align*}
where $\Omega \subseteq \Gamma$, $k\geq 0$, $0<\gamma<1$,
$\text{dist}_{\Gamma}(y_1, y_2)$ is the intrinsic distance on
$\Gamma$ and
\begin{equation*}
    \mathcal{C}_{r}(y) = \{\sum_{i=1}^8 t_i \tilde{e}_{i} +l\nu(y): |t|<r,~ l\in\real\}
\end{equation*}
is the infinite right cylinder with a base $B'_{r}:=\{|t|< r\}
\subset T_y$ on the tangent plane $T_y$ to $y\in \Gamma$ (see
\eqref{defGt} to recall notation).

We will now establish the following regularity estimate.
\begin{lemma}[compare to Lemma 7.5 in \cite{PKW}]\label{apriori_lemma}
Let $R>0$ be finite or infinite and assume $h\in
C^{2,\gamma}(\Gamma_R)$ is a solution to \eqref{linearProbbounded}
with $f\in C^{\gamma}(\Gamma_R)$. Then
\begin{equation}\label{apriori}
    \|D^2_{\Gamma}f\|_{k+2, C^{\gamma}(\Gamma_{R/2})} + \|D_{\Gamma}f\|_{k+1,L^{\infty}(\Gamma_{R/2})}
\leq C (\|h\|_{k,L^{\infty}(\Gamma_R)} +
\|f\|_{k+2,C^{\gamma}(\Gamma_R)} )
\end{equation}
with a constant $C>0$, independent of $R$.
\end{lemma}
\begin{proof}
The proof is based on a rescaling technique. We may assume
\[
\|h\|_{k,L^{\infty}(\Gamma_R)} + \|f\|_{k+2,C^{\gamma}(\Gamma_R)}
\leq 1.
\]
Pick $y\in \Gamma_{R/2}$, set $\rho = 1 + r(y)$ and express the
operator $J_{\Gamma}$ in $C_{\beta \rho}(y)\cap \Gamma_R $ using the
coordinates $t$ \eqref{defGt}:
\begin{align*}
    J_{\Gamma} h & = \frac{1}{\sqrt{|g|}}\del_i(g^{ij} \sqrt{|g|}\del_j
h) + |A|^2 h = \\
& = g^{ij}\del^2_{ij} h + \del_ig^{ij} \del_j h + \frac{g^{ij}}{2}
\frac{\del_i |g|}{|g|} \del_j h + |A|^2 h= \\
& = a^{ij}\del^2_{ij} h (t) + b^i \del_i h (t) + |A|^2 h (t) = f(t)
\qquad t\in B'_{\beta\rho}.
\end{align*}
Rescaling to size one,
\[
    \bar{h}(t) = \rho^k h(\rho y),\quad \bar{f}(t)= \rho^{k+2}
g(\rho t), \quad \bar{a}^{ij}(t) = a^{ij}(\rho t), \quad \bar{b}^i =
\rho b^i(\rho t),
\]
we get
\begin{equation*}
   \bar{a}^{ij}\del^2_{ij} \bar{h} (t) + \bar{b}^i \del_i \bar{h} (t) + |A_{\rho}|^2(t) \bar{h} (t) =
\bar{f}(t) \qquad \text{in} \quad B'_{\beta}.
\end{equation*}
Recall the standard H\"older norm of a function $q$ defined on a
domain $U\subseteq\real^8$:
\begin{equation*}
    \|q\|_{C^{\gamma}(U)}:= \|q\|_{L^{\infty}(U)} + \sup_{t \neq
    s\in U} \frac{|q(t)-q(s)|}{|t-s|^{\gamma}}.
\end{equation*}
Because of the estimates \eqref{tderTensor} on the metric tensor $g$
and its derivatives, and the estimates \eqref{Hl_FullEst} on the
second fundamental form $|A|^2$ and its derivatives, we can bound
\[
\|\bar{a}^{ij}\|_{C^{\gamma}(B'_{\beta})},
\|\bar{b}^i\|_{C^{\gamma}(B'_{\beta})},
\||A_{\rho}|^2\|_{C^{\gamma}(B'_{\beta})} \leq K
\]
by a universal constant $K$. Thus, by interior Schauder estimates,
\begin{equation*}
    \|D^2_t \bar{h}\|_{C^{\gamma}(B'_{\beta/2})} + \|D_t
\bar{h}\|_{L^{\infty}(B'_{\beta/2})} \leq C(
\|\bar{h}\|_{L^{\infty}(B'_{\beta/2})} +
\|\bar{f}\|_{C^{\gamma}(B'_{\beta/2})}) \leq C'
\end{equation*}
so that
\begin{equation}\label{interior1}
\rho^{k+2+\gamma}|D^2_{\Gamma} h(y) |_{C^{\gamma}(\Gamma_{R}\cap
\mathcal{C}_{\rho\beta/2}(y))} + \rho^{k+2}|D^2_{\Gamma} h (y)| +
\rho^{k+1} |D_{\Gamma} h (y)| \leq C'',
\end{equation}
for each $y\in \Gamma_{R/2}$.
\end{proof}

Let us now show that
\begin{equation*}
J_{\Gamma} h = f \qquad \text{in} \quad \Gamma,
\end{equation*}
where the right-hand side $\|f\|_{k+2,C^{\gamma}(\Gamma)} < \infty$
for some $k>2$, is uniquely solvable when $\|h\|_{k,\infty} <
\infty$.
\begin{prop}
Let $k>2$, $0<\gamma<1$ and $\|f\|_{k+2,C^{\gamma}(\Gamma)} <
\infty$. There exists a unique solution $h\in C^2(\Gamma)$ to
\eqref{linearProb} such that $\|h\|_{k, L^{\infty}(\Gamma)} <
\infty$. Moreover,
\begin{equation}\label{globalest}
    \|D^2_{\Gamma} h\|_{k+2,L^{\infty}(\Gamma)} + \|D_{\Gamma}
h\|_{k+1,L^{\infty}(\Gamma)} + \|h\|_{k,L^{\infty}(\Gamma)} \leq
C\|f\|_{k+2,C^{\gamma}(\Gamma)}.
\end{equation}
\end{prop}
\begin{proof}
Uniqueness follows from the maximum principle (Remark \ref{maxPrin})
and the fact that $|h/h_0|\leq C r^{2-k} \to 0$ as $r\to \infty$.

To establish existence, consider the Dirichlet problem in expanding
bounded domains:
\begin{align*}
    J_{\Gamma} h_n & = f \qquad \text{in}\quad \Gamma_{R_n}  \\
    h_n & = 0 \qquad \text{on} \quad \del\Gamma_{R_n},
\end{align*}
where $R_n \nearrow \infty$. First claim that
\begin{equation}\label{strengthen}
    \|h_n\|_{k, L^{\infty}(\Gamma_{R_n})} \leq C
\|f\|_{k+2,C^{\gamma}(\Gamma_{R_n})}.
\end{equation}
for some constant independent of $n$. Assume not; then there is a
subsequence (call it $R_{n}$ again) such that
\[
\|h_{n}\|_{k, L^{\infty}(\Gamma_{R_{n}})} \geq n
\|f\|_{k+2,C^{\gamma}(\Gamma_{R_n})}.
\]
If we set $\bar{f}_n = f/\|h_{n}\|_{k, L^{\infty}(\Gamma_{R_{n}})}$,
$\bar{h}_n = h_{n}/\|h_n\|_{k, L^{\infty}(\Gamma_{R_{n}})}$, we see
that
\[
    J_{\Gamma} \bar{h}_n = \bar{f}_n
\]
with $\|\bar{h}_n\|_{k, L^{\infty}(\Gamma_{R_{n}})}=1$ and
$\|\bar{f}_n\|_{k+2,C^{\gamma}(\Gamma_{R_n})} \leq 1/n$. The a
priori estimate \eqref{apriori} implies, after possibly passing to a
subsequence, that $h_n$ converge uniformly on compact sets to a
$C^2(\Gamma)$-function $\bar{h}$ with $\|\bar{h}\|_{k,\infty} <
\infty$ which solves
\[
    J_{\Gamma}\bar{h} = 0 \qquad \text{in} \quad \Gamma.
\]
Uniqueness requires that $h = 0$. Let $h_{\infty}'$ be the
supersolution for $J_{\Gamma}$ provided by Lemma \ref{h'infty} with
some $0<\eps < k-2$ and $r_0$ large enough:
\[
    J_{\Gamma} h_{\infty}' \leq - \frac{1}{1+ r^{4+\eps}}, \qquad h_{\infty}'\geq
    \frac{c}{1+r^{2+\eps}} \qquad r>r_0.
\]
Since $h_n \to 0$ uniformly on compact sets, $s_n :=
\sup_{\Gamma_{r_0}}h_n\to 0$. Therefore,
\begin{align*}
    \pm h_n + \mu_n h_{\infty}' & \geq 0 \qquad \text{on} \quad r=r_0 \quad \text{and}
\quad r=R_n \\
J_{\Gamma}(\pm h_n + \mu_n h_{\infty}') & \leq 0 \qquad \text{in}
\quad r_0<r<R_n
\end{align*}
for $\mu_n  = \max\{s_n c^{-1}, \frac{1}{n}\} \to 0$. An application
of the maximum principle yields
\[
    |h_n| \leq \mu_n h_{\infty}' \qquad \text{in} \quad
r_0<r<R_n.
\]
Combine this with the fact that
$\|h_n\|_{k,L^{\infty}(\Gamma_{r_0})} \leq s_nr_0^k$ to conclude
\[
\|h_n\|_{k,L^{\infty}(\Gamma_{R_n})} \to 0 \quad \text{as}\quad
n\to\infty
\]
which is a contradiction. Hence, \eqref{strengthen} holds and the a
priori estimate \eqref{apriori} becomes
\[
\|D^2_{\Gamma}h_n\|_{k+2, C^{\gamma}(\Gamma_{R_n/2})} +
\|D_{\Gamma}h_n\|_{k+1,L^{\infty}(\Gamma_{R_n/2})} + \|h_n\|_{k,
L^{\infty}(\Gamma_{R_n/2})} \leq C
\|f\|_{k+2,C^{\gamma}(\Gamma_{R_n})}
\]
for some constant $C$, independent of $R_n$. Now a standard
compactness argument produces a $C^2(\Gamma)$-function $h$ which
solves \eqref{linearProb} and satisfies the estimate
\eqref{globalest}.
\end{proof}
Proposition \ref{h'} is an immediate corollary.
\begin{proof}[Proof of Proposition \ref{h'}]
Let $h$ solve \eqref{linearProb} with a right-hand side $f = -
\frac{1}{1+r^{4+\eps}}$. We are only left with checking that $h$ is
strictly positive. This is a consequence of the strong maximum
principle (Remark \ref{maxPrin}) and the fact that $|h/h_0(y)| =
O(r^{-\eps}(y)) \to 0$, as $r(y)\to\infty$.
\end{proof}

Now we would like to construct a global barrier function
(not-necessarily smooth) for \eqref{linearProb} with a
right-hand-side
\begin{equation*}
    f = -\frac{|\theta - \pi/4|}{1+r^3}.
\end{equation*}

\begin{lemma}\label{globBarrier"} There exists a globally defined, locally Lipschitz function $h\geq0$
which is a weak supersolution for $J_{\Gamma}$ and which satisfies
\begin{equation*}
    J_{\Gamma}h \leq -\frac{|\theta - \pi/4|}{1+r^3} \quad
    \text{in}\quad \Gamma.
\end{equation*}
Moreover,
\begin{equation}\label{weakhbound}
    h = O\Big(\frac{|\theta - \pi/4|^{\tau}}{1+r}+ \frac{1}{1+ r^{2+\eps}}\Big)
\end{equation}
for some $\tau \in (\frac{1}{2},\frac{2}{3})$ and some $\eps\in
(0,1)$ (e.g. $\tau = \frac{5}{8}$ and $\eps = \frac{1}{8}$).
\end{lemma}

\begin{proof}
Let $h_{\infty}''$ be the weak supersolution for $J_{\Gamma}$ in
$\Gamma_{r_0}^c$, provided by Lemma \ref{h"_infty}:
\begin{equation*}
    (J_{\Gamma}h_{\infty}''-f)[\phi]\leq 0
\end{equation*}
for every non-negative $\phi\in C^1_c(\Gamma_{r_0}^c)$. Now let
$\psi \in C^{\infty}(\Gamma)$ be a non-negative cutoff function such
that
\[
    \psi(y) = 0 \quad \text{for} \quad r(y)\leq r_0 \quad \text{and} \quad \psi(y) = 1 \quad \text{for} \quad r(y)\geq r_0+1.
\]
Define a function $h''$ on the whole of $\Gamma$ by
\begin{equation*}
    h''(y) = \left\{ \begin{array}{lr}0 & \text{in} \quad r<r_0 \\ \psi(y) h_{\infty}''(y) & \text{in} \quad r\geq r_0 \end{array}
\right.
\end{equation*}
Finally set
\begin{equation*}
    h = Ch' + h'',
\end{equation*}
where $h'$ is the supersolution provided by Proposition \ref{h'} and
$C>0$ is some large constant, to be fixed shortly. Now for any
nonnegative $\phi\in C^1_c(\Gamma)$, the fact that
$(J_{\Gamma}h_{\infty}''-f)[\psi \phi] \leq 0$ implies
\begin{align*}
    (J_{\Gamma}h -f)[\phi] & = C(J_{\Gamma}h')[\phi] - f[\phi] +
\int -h_{\infty}''\nabla_{\Gamma} \psi  \cdot \nabla_{\Gamma}\phi -
\nabla_{\Gamma}h_{\infty}'' \cdot (\psi\nabla_{\Gamma}\phi) +
|A^{2}|h_{\infty}''\psi\phi \\
& = C(J_{\Gamma}h')[\phi] - (1-\psi)f[\phi] + (J_{\Gamma}h_{\infty}'')[\psi \phi] - f[\psi\phi] \\
& + \int(2\nabla_{\Gamma}\psi \cdot \nabla_{\Gamma}h_{\infty}''+
h_{\infty}''\Delta_{\Gamma}\psi)\phi \leq -C(J_{\Gamma}h')[\phi] +
k[\phi]
\end{align*}
where $k$ is a bounded function, compactly supported in
$\Gamma_{r_0+1}$. We were able to carry out the integration by
parts, since $h_{\infty}''$ is locally Lipschitz. Taking $C>0$ large
enough we conclude that $J_{\Gamma} h \leq f$ globally, in the weak
sense.
\end{proof}

We now possess all the means to prove Proposition \ref{h"}.
\begin{proof}[Proof of Proposition \ref{h"}]
Pick $f\in C^{0,\gamma}(\Gamma)$ such that $f\leq 0$ and
\begin{equation*}
    f\circ \pi_{\Gamma}(\tilde{y}) = -\frac{|\theta(\tilde{y}) - \pi/4|}{1+r^3(\tilde{y})} \qquad \tilde{y}\in
    \Gamma_{\infty}\cap\{r>r_0\}
\end{equation*}
for a large enough $r_0$. It is not hard to verify that
$\|f\|_{3,C^{\gamma}(\Gamma)} < \infty$ by transferring the
computation onto $\Gamma_{\infty}$ via \eqref{GradComp} and
employing the gradient estimate in Lemma \ref{gradhess}.

Let $h_n$ solve the Dirichlet problem \eqref{linearProbbounded} in
the expanding bounded domains $\Gamma_{R_n}$, $R_n\nearrow \infty$
\begin{align*}
    J_{\Gamma} h_n & = f \qquad \text{in}\quad \Gamma_{R_n}  \\
    h_n & = 0 \qquad \text{on} \quad \del\Gamma_{R_n}.
\end{align*}
Since $f$ is non-positive, the weak maximum principle implies $h_n
\geq 0$. Let $h'$ be a Type 1 supersolution, provided by Proposition
\ref{h'}, and let $h''$ be the weak Type 2 supersolution which we
constructed in Lemma \ref{globBarrier"}. Noting again that
\begin{equation*}
    \frac{|\theta(\tilde{y}) - \pi/4|}{1+r^3(\tilde{y})} =
\frac{|\theta(y) - \pi/4|}{1+r^3(y)} + O(r(y)^{-5-\sigma}) \qquad y
= \pi_{\Gamma}(\tilde{y})
\end{equation*}
for some $\sigma >0$, we obtain
\begin{align*}
    J_{\Gamma}(-h_n + h''+ Ch')\leq 0
\end{align*}
for a large enough $C$. Moreover, since $- h_n+ Ch' + h'' \geq 0$ on
$\del{\Gamma_{R_n}}$, the maximum principle implies
\begin{equation*}
    0\leq h_n \leq h''+Ch' \quad \text{in}\quad \Gamma_{R_n}.
\end{equation*}
Thus, $\|h_n\|_{1,L^{\infty}(\Gamma_{R_n})}\leq C'$ for an absolute
constant $C'$ independent of $n$. We can now employ the a priori
estimate \eqref{apriori} into a standard compactness argument that
yields a non-negative $C^2$--function $\bar{h}$ solving
\begin{equation*}
    J_{\Gamma} \bar{h} = f \quad \text{in}\quad \Gamma
\end{equation*}
with
\begin{equation}\label{sizeofh_prelim}
\begin{aligned}
& 0\leq \bar{h} \leq  h'' + O(h')\\
    & \|D^2_{\Gamma}\bar{h}\|_{3,\infty} + \|D_{\Gamma} \bar{h}\|_{2, \infty} +
\|\bar{h}\|_{1,\infty} < \infty.
\end{aligned}
\end{equation}
After possibly correcting $\bar{h}$ by a supersolution of Type 1,
$h= \bar{h} + ch'$, 
we can conclude that
\begin{equation*}
J_{\Gamma}h(y) = f + c J_{\Gamma} h' \leq -\frac{|\theta(y) -
\pi/4|}{1+r^3(y)}.
\end{equation*}

To establish the second statement in the proposition, namely the
refinement of decay of $h$ near $\theta = \pi/4$, we notice that on
$S(-\delta)$ with  $0< \delta < \frac{3}{2}$
\begin{equation*}
    \bar{h} = O\Big(\frac{|\theta - \pi/4|^{\tau}}{1+r}+
\frac{1}{1+r^{2^+}}\Big) = O(r^{-1-\delta \tau}),
\end{equation*}
as $\delta \tau + 1 < 2 $.
Also, $\|f\|_{3+\delta \tau, C^{\gamma}(S(-\delta))} < \infty$.

Then an argument, based on rescaling and interior elliptic estimates
-- absolutely analogous to the one for the a priori estimate (Lemma
\ref{apriori_lemma}) -- gives us the interior estimate
\eqref{refinedecay} (for $\bar{h}$ and thus for $h$ itself) on
$S(-\delta') \Subset S(-\delta)$. There is a caveat:  the same
argument will carry through to the present situation, once we
ascertain that $S(-\delta)$ contains ``balls" of size $\sim r$,
centered on points in $S(-\delta')$ far away from the origin. More
precisely, we want for some $r_0$ large enough,
\begin{equation}\label{fullballs}
\mathcal{C}_{\beta r(y)}(y) \cap \Gamma \subseteq S(-\delta) \quad
\text{for every } y\in S(-\delta')\cap \{r(y)>r_0\}
\end{equation}
Note that according to Lemma \ref{TcoorandsizeofG}, the fact that
$\Gamma$ is a graph $\{(t, G(t)\}$ over $B'_{\beta r(y)}(y)$ with
\begin{equation*}
    |G(t)| \leq C r(y)
\end{equation*}
implies
\[
    \mathcal{C}_{\beta r(y)}(y)
\cap \Gamma \subseteq B_{c_0r(y)}(y) \cap \Gamma
\]
for a large enough numerical constant $c_0>0$.
Suppose that \eqref{fullballs} is not true: then there exist $y'\in
\del S(-\delta')$ and $y\in \del S(-\delta)$ with $r' = r(y')$,
$r=r(y)$ arbitrarily large such that $|y-y'|<c_0r'$. Denote the
projections of $y'$ and $y$ onto $\real^8$ by $(\vec{u'}, \vec{v'})$
and $(\vec{u}, \vec{v})$, respectively. Obviously, $r/r' \sim 1$ and
for $0\leq \delta<\delta'<2$
\begin{align*}
    |y'-y|^2 & = |\vec{u}-\vec{u'}|^2 + |\vec{v}-\vec{v'}|^2 + |F(r',(1+r')^{-\delta'}) -
F(r,(1+r)^{-\delta})|^2 \\
& \geq (u - u')^2 + (v-v')^2 + |F(r',(1+r')^{-\delta'}) -
F(r,(1+r)^{-\delta})|^2 \\
   &\geq (r'-r)^2 + |F_{\infty}(r',(1+r')^{-\delta'}) -
F_{\infty}(r,(1+r)^{-\delta})|^2 - C'r^{-2\sigma} \\
& \geq - C'r^{-2\sigma} + (r'-r)^2 + c'
\left|r^{\delta-\delta'}\Big(\frac{r}{r'}\Big)^{3-\delta}-1\right|^2(r')^{2(3-\delta')}
\gg (r')^2
\end{align*}
which is a contradiction.
\end{proof}

\subsection{The free boundary super and subsolution.}

In correspondence with the form of the supersolution ansatz
\eqref{ansatz}, define the subsolution ansatz
$v:\mathcal{B}_{\Gamma_{\alpha}} \rightarrow \real$ by
\begin{equation}\label{subansatz}
    v(y,z) = -h_0^{\alpha}(y) + z h_1^{\alpha}(y) + z^2(-h_2^{\alpha}(y)) + z^3 h_3^{\alpha}(y) + z^5
    h_5^{\alpha}(y).
\end{equation}
Since we require $h_0^{\alpha}>0$ and $h_2^{\alpha}\ > 0$, we will
automatically have $v<w$ in $\mathcal{B}_{\Gamma_{\alpha}}$. Also,
\begin{equation*}
    0< w-v = 2 (h_0^{\alpha} + z^2h_2^{\alpha}) = O\left(\frac{\alpha^p}{1+\alpha
    r}\right).
\end{equation*}

\begin{prop}\label{Sec_ViscSup}
Fix $0<p<1$. There exist $C_1,C_2 >0$ and $\alpha_0 >0$ such that
for all small enough $\alpha \leq \alpha_0$, $w$ given by
\eqref{ansatz} satisfies
\begin{align}\label{punchline}
 & \Delta w  < 0 \quad \text{in}\quad \mathcal{B}_{\Gamma_{\alpha}}
    \notag \\
    & |\nabla w |^2  > 1 \quad \text{on}~ \{w=1\} \quad
    \text{and} \\ & |\nabla w |^2 < 1 \quad \text{on}~
    \{w=-1\},\notag
\end{align}
while
\begin{align}\label{subpunchline}
 & \Delta v  > 0 \quad \text{in}\quad
 \mathcal{B}_{\Gamma_{\alpha}}\notag
    \\
    & |\nabla v |^2  < 1 \quad \text{on}~ \{v=1\} \quad
    \text{and} \\ & |\nabla v |^2 > 1 \quad \text{on}~
    \{v=-1\}.\notag
\end{align}
Moreover, $0 < w-v \leq \frac{1}{2}$, $\del_{x_9}v >0$ and
$\del_{x_9}w>0$ in $\mathcal{B}_{\Gamma_{\alpha}}$.
\end{prop}
We immediately derive as a corollary:
\begin{coro}\label{punches}
Let $v$, $w$, $0<\alpha \leq \alpha_0$ be as in Proposition
\ref{Sec_ViscSup} above. Then the function $W:\real^9\rightarrow
\real$, given by
\[
    W(x) =  \left\{ \begin{array}{cl}w(x) & \quad \text{for}~ x\in \mathcal{B}_{\Gamma_{\alpha}}\cap \{|w|\leq 1\}  \\
    1 & \quad \text{for}~ x \in (\mathcal{B}_{\Gamma_{\alpha}}\cap \{|w|\leq
    1\})^c \cap \{x_9 > F(x')\} \\
    -1 & \quad \text{for}~ x \in (\mathcal{B}_{\Gamma_{\alpha}}\cap
\{|w|\leq
    1\})^c \cap \{x_9 < F(x')\}
    \end{array}\right.
\]
is a classical strict supersolution to \eqref{FBP}, while the
function $V:\real^9 \rightarrow \real$, given by
\begin{equation*}
    V(x) =  \left\{ \begin{array}{cl}v(x) & \quad \text{for}~ x\in \mathcal{B}_{\Gamma_{\alpha}}\cap \{|v|\leq 1\}  \\
    1 & \quad \text{for}~ x \in (\mathcal{B}_{\Gamma_{\alpha}}\cap \{|v|\leq
    1\})^c \cap \{x_9 > F(x')\} \\
    -1 & \quad \text{for}~ x \in (\mathcal{B}_{\Gamma_{\alpha}}\cap
\{|v|\leq
    1\})^c \cap \{x_9 < F(x')\}
    \end{array}\right.
\end{equation*}
is a classical strict subsolution. Moreover, $0 \leq W- V \leq
\frac{1}{2}$, both $V$ and $W$ are monotonically increasing in $x_9$
and strictly increasing in $x_9$ inside $\Omega_{\text{in}}(V)$,
$\Omega_{\text{in}}(W)$, respectively.
\end{coro}

\begin{proof}[Proof of Proposition \ref{Sec_ViscSup}]
Fix some $0<\delta < \frac{1}{2}$ and $0<\eps< \delta \tau$, where
$\frac{1}{2}<\tau < \frac{2}{3}$ is provided by Proposition \ref{h"}
and let $h'> 0$, $h''\geq 0$ be the $J_{\Gamma}$-supersolutions
given by Lemma \ref{h'} and \ref{h"}, respectively. Remember that we
only need to set the values of $h_0$ and $h_2'$ in order to
determine the ansatz \eqref{ansatz} completely. So, let
\begin{equation*}
    h_0 = h' + h'',
\end{equation*}
and
\begin{equation*}
    h_2' = \frac{1}{2} \Big(\frac{1}{1+r^{4+\eps}} + \frac{c\cos^2(2\theta)}{1+r^3}\Big),
\end{equation*}
where $c>0$ is such that $2c\cos^2(2\theta) \leq |\theta-\pi/4|$.
Set
\begin{align*}
    C_1 & = \|D^2_{\Gamma}h_0\|_{3,\infty}  + \|D_{\Gamma}h_0\|_{2,\infty} + \|h_0\|_{1,\infty}  \\
    C_2 & = \|D^2_{\Gamma}h_2'\|_{5,\infty}  + \|D_{\Gamma}h_2'\|_{4,\infty} +
    \|h_2'\|_{3,\infty}.
\end{align*}

Claim that for all small enough $\alpha > 0$, $\Delta w < 0$ in
$\mathcal{B}_{\Gamma_{\alpha}}$. This is a consequence of the
following computations.
\begin{itemize}
\item For $\alpha>0$ small enough,
\begin{align*}
    J_{\Gamma_{\alpha}}h_0^{\alpha} - z^2 H_{3,\alpha} & \leq - \frac{\alpha^{2+p}}{1+(\alpha r)^{4+\eps}} - \frac{\alpha^{2+p}|\theta-\pi/4|}{1+(\alpha r)^{3}} + O\Big(\frac{\alpha^3 |\theta-\pi/4|}{1+(\alpha
    r)^3}\Big) \\
    & = - \frac{\alpha^{2+p}}{1+(\alpha r)^{4+\eps}} -  \frac{\alpha^{2+p}}{2}\frac{|\theta-\pi/4|}{1+(\alpha
    r)^{3}},
\end{align*}
so that
\begin{equation}\label{lap_major}
    J_{\Gamma_{\alpha}}h_0^{\alpha} - z^2 H_{3,\alpha} +
    h_2'^{\alpha} \leq -\frac{\alpha^{2+p}}{4} \left( \frac{2}{1+(\alpha r)^{4+\eps}} +  \frac{|\theta-\pi/4|}{1+(\alpha
    r)^{3}}\right).
\end{equation}

\item In $S_{\alpha}(-1) = \{|\theta-\frac{\pi}{4}| \leq (1+\alpha r)^{-1}\}$,
Proposition \ref{h"} and \eqref{DeltaCompAlpha} imply
\begin{align}
    (\Delta_{\Gamma_{\alpha}(z)} + |A_{\alpha}|^2) h_0^{\alpha}
    =
    J_{\Gamma_{\alpha}}h_0^{\alpha}  + O\Big(\frac{\alpha^{3+p}}{1+(\alpha
    r)^{4+\delta \tau}}\Big) \label{lap_rem1}.
\end{align}
Then \eqref{lap_ansatz}, \eqref{lap_major} and \eqref{lap_rem1}
yield the desired
\begin{equation*}
    \Delta w(y,z) < 0 \quad \text{for } y\in S_{\alpha} ~ \text{and}~ (y,z)\in \mathcal{B}_{\Gamma_{\alpha}},
\end{equation*}
and all small enough $\alpha >0$.

\item In $S_{\alpha}^{c}(-1)= \{|\theta-\frac{\pi}{4}| > (1+ \alpha
r)^{-1}\}$, \eqref{lap_major} can be estimated further by
\begin{equation}\label{lap_major2}
     J_{\Gamma_{\alpha}}h_0^{\alpha} - z^2 H_{3,\alpha} +
    h_2'^{\alpha} \leq -\frac{\alpha^{2+p}}{4} \left(\frac{2}{1+(\alpha r)^{4+\eps}} + \frac{c'}{1+(\alpha
    r)^{4}}\right).
\end{equation}
Because of \eqref{DeltaCompAlpha} we have
\begin{align}
    (\Delta_{\Gamma_{\alpha}(z)} + |A_{\alpha}|^2) h_0^{\alpha} =  J_{\Gamma_{\alpha}}h_0^{\alpha}
    + O\Big(\frac{\alpha^{3+p}}{1+(\alpha r)^4}\Big)
    \label{lap_rem2}.
\end{align}
Thus, \eqref{lap_ansatz}, \eqref{lap_major2} and \eqref{lap_rem2}
yield
\begin{equation*}
    \Delta w(y,z) < 0 \quad \text{for } y\in S_{\alpha}^c ~ \text{and}~ (y,z)\in \mathcal{B}_{\Gamma_{\alpha}},
\end{equation*}
and all small enough $\alpha >0$.
\end{itemize}

To verify that necessary gradient conditions \eqref{subpunchline}
are also met, we need to check that for small enough $\alpha>0$,
$h_2'^{\alpha}$ majorizes both $|A_{\alpha}|^2(h_0^{\alpha})^2$ and
$|\nabla_{\Gamma_{\alpha(z_{\pm})}}h_0^{\alpha}|^2$ (see Lemma
\ref{supGrad}). Indeed,
\begin{itemize}
\item in $S_{\alpha}(-\frac{1}{2})=\{|\theta-\frac{\pi}{4}| \leq (1+\alpha r)^{-\frac{1}{2}}\}$
\begin{equation*}
    |A_{\alpha}|^2(h_0^{\alpha})^2 +
    |\nabla_{\Gamma_{\alpha(z_{\pm})}}h_0^{\alpha}|^2 = O\Big( \frac{\alpha^{2+2p}}{1+(\alpha r)^{4+2\delta\tau}}\Big)
\end{equation*}
is dominated by $h_2'^{\alpha} \geq
\frac{1}{2}\frac{\alpha^{2+p}}{1+(\alpha r)^{4+\eps}}$;
\item in $S_{\alpha}^c(-\frac{1}{2})=\{|\theta-\frac{\pi}{4}| > (1+\alpha r)^{-\frac{1}{2}}\}$
\begin{equation*}
    |A_{\alpha}|^2(h_0^{\alpha})^2 +
    |\nabla_{\Gamma_{\alpha(z_{\pm})}}h_0^{\alpha}|^2 = O\Big(\frac{\alpha^{2+2p}}{1+(\alpha
r)^4}\Big)
\end{equation*}
is dominated by $h_2'^{\alpha} \geq
\frac{1}{2}\frac{\alpha^{2+p}\cos^2(2\theta)}{1+(\alpha r)^{3}} \geq
\frac{c''\alpha^{2+p}}{1+(\alpha r)^{4}}$.
\end{itemize}

Checking that $v$ meets the conditions for a subsolution is
absolutely analogous. In view of Lemma \ref{ansmonotoni},
$\del_{x_9}w>0$ and similarly $\del_{x_9} v>0$.
\end{proof}

\section{The solution. Existence and
regularity.}\label{Section_FBexireg}

We have at our disposal a globally defined classical strict
subsolution $V$ to \eqref{FBP} lying below a classical strict
supersolution $W$ both of which are monotonically increasing in
$x_9$ (in fact, strictly increasing in their interphases
$\real^9_{\text{in}}$). In this section we will explain why this
engenders the existence of a classical solution $u$ to \eqref{FBP},
trapped in-between. Moreover, the solution will inherit some of the
nice properties of the barriers $V$, $W$, such as monotonicity in
$x_9$ and graph free boundaries $F^{+}(u)$ and $F^{-}(u)$.

We will construct $u$ as a \emph{global minimizer} of $I$,
constrained to lie between $V$ and $W$.
\begin{definition}
A function $u\in H^1_{\text{loc}}(\real^n)$ is a global minimizer of
$I$, constrained between $V\leq W$ if for any bounded right cylinder
$\Omega\subset\real^9$
\begin{equation*}
    I(u,\Omega) \leq I(v,\Omega) \quad \text{for all}~ v\in H^1(\Omega)~ \text{such that}\quad V\leq v \leq W
    ~ \text{and} ~ u-v\in H^1_0(\Omega).
\end{equation*}
\end{definition}

As usual, we obtain a global (constrained) minimizer $u$ as a
sequence of local (constrained) minimizers on expanding bounded
domains. For the purpose, we will verify that local minimizers are
Lipschitz continuous with a universal bound on the local Lipschitz
constant. This is done in the spirit of \cite{DS}.

Afterwards, we will show that a global minimizer $u$ which, in
addition, meets certain simple geometric constraints, is actually a
classical solution to our free boundary problem. This is achieved
almost for free --  by applying the regularity theory of minimizers
to the energy functional $I_0$, developed in \cite{DS} and
\cite{DSJ}, to the functions $1\pm u$.




\subsection{Existence of a local minimizer.}

Let $\Omega \subset \real^n$ be a cylinder
\begin{equation*}
    C_{R,h} = \{x = (x',x_n) \in \real^{n-1}\times\real: |x'|<R, |x_n|< h \}
\end{equation*}
and consider the minimization problem for the functional
\begin{equation*}
    I(v,\Omega) = \int_{\Omega} |\nabla v|^2 + \chrc{|v|<1},
\end{equation*}
where $v$ ranges over the following closed convex subset of
$H^1(\Omega)$:
\begin{equation*}
    S(\Omega) = \{v \in H^1(\Omega): V\leq v \leq W\quad \text{a.e.}\}.
\end{equation*}
Let us show that there exists $u\in S(\Omega)$ for which the infimum
of $I(\cdot, \Omega)$ over $S(\Omega)$ is attained.

\begin{prop}[Existence of monotone local minimizers]\label{ExiMin}

There exists $u\in S(\Omega)$ such that
\begin{equation*}
    I(u,\Omega) = m := \inf_{v\in S(\Omega)} I(v,\Omega).
\end{equation*}
Moreover, given that $V$ and $W$ are monotonically increasing in the
$x_n$-variable, $u$ can also be taken to be monotonically increasing
in the $x_n$-variable.
\end{prop}
\begin{proof}
For convenience use the simplified notation $I(v) = I(v,\Omega)$.
Obviously, the infimum $m$ is non-negative and finite:
\begin{equation*}
    0\leq m \leq C_0:=\min(I(V), I(W)).
\end{equation*}
Take a sequence $u_k\in S(\Omega)$ such that $C_0\geq I(u_k)\searrow
m$. Then
\begin{equation*}
    \|u_k\|^2_{H^1} = \|u_k\|^2_{L^2} + \|\nabla u_k\|^2_{L^2} \leq
    |\Omega| + C_0
\end{equation*}
is uniformly bounded, so by compactness we can extract a subsequence
(call it $u_k$ again) such that
\begin{equation*}
    u_k \to u \quad \text{in}~ L^2 ~ \text{and a.e.} \quad \text{and}
    \quad \nabla u_k \rightharpoonup \nabla u \quad \text{weakly in } L^2
\end{equation*}
for some $u\in S(\Omega)$. Claim that $I(u) = m$. It suffices to
show that $I$ is lower semicontinuous with respect to the weak-$H^1$
topology, i.e.
\begin{equation}\label{lowersemi}
    I(u) \leq \liminf_{k\to \infty} I(u_k)
\end{equation}
which is done analogously as in \cite{AC}.


We can produce a minimizer, which is monotonically increasing in the
$x_n$-variable by applying a rearrangement. A monotone-increasing
rearrangement in the $x_n$-variable, $f \rightarrow f^*$ satisfies
the following properties (cf. \cite{Kawohl}):
\begin{enumerate}
\item If $f$ is monotonically increasing in the $x_n$-variable, $f^* =
f$.
\item The functions $f$ and $f^*$ are equimeasurable, i.e. $|f^{-1}(O)| =
|(f^*)^{-1}(O)|$ for any open interval $O\subseteq \real$.
\item The mapping $f\rightarrow f^*$ is order-preserving, i.e. if $f\leq
g$ then $f^*\leq g^*$.
\item If $f\in H^1(C_{R,h})$, then $f^* \in H^1(C_{R,h})$ and
\[
    \|\nabla f^*\|^2_{L^2} \leq  \|\nabla f\|^2_{L^2}.
\]
\end{enumerate}

Since $V, W$ are monotonically increasing in the $x_n$-variable,
$V^* = V$ and $W^* = W$; thus $V \leq u^* \leq W$ by order
preservation under rearrangements. Moreover, $u^* \in H^1(\Omega)$,
so that $u^* \in S(\Omega)$ and because of properties $2$ and $4$
above,
\begin{equation*}
    m\leq I(u^*) \leq I(u) = m.
\end{equation*}
Thus, $u^*$ is a minimizer to $I$ over $S(\Omega)$, monotonically
increasing in the $x_n$-variable.
\end{proof}

\subsection{Lipschitz continuity of local minimizers.}

Employing standard arguments, we first establish continuity of local
minimizers before we prove Lipschitz continuity with a universal
bound on the local Lipschitz constant.

We adapt the technique of \emph{harmonic replacements} used by
\cite{AC}.
\begin{definition}
The harmonic replacement of $u$ in the ball $B \subset \Omega$ is
the unique function $v\in H^1(\Omega)$ that is harmonic in $B$ and
agrees with $u$ on $\Omega\setminus B$.
\end{definition}

Let
\begin{equation*}
    \mathcal{B}_{V,W} = \Omega_{\text{in}}(V)\cup
    \Omega_{\text{in}}(W)
\end{equation*}
and note that $V$ is subharmonic in $\mathcal{B}_{V,W}$ whereas $W$
is superharmonic in $\mathcal{B}_{V,W}$. 

Below we show that the function $u$, constructed in Proposition
\ref{ExiMin}, is continuous in $\mathcal{B}_{V,W}$.
\begin{prop}[Continuity]\label{contMin}
Let $D\Subset \mathcal{B}_{V,W} \subseteq \Omega$. Then the
minimizer $u$, constructed in Proposition \ref{ExiMin}, is in a
H\"older class $C^{\alpha}(D)$ for some $\alpha
>0$, depending on $D$. In particular, $u$ is continuous in
$\mathcal{B}_{V,W}$.
\end{prop}
\begin{proof}
Let $B_{\rho}\subseteq D$ denote a ball of radius $\rho$, centered
at some fixed point in $D$. Let $v_r$ be the harmonic replacement of
$u$ in the concentric $B_r\subseteq B_{\rho}$. Since $B_r\subset
\mathcal{B}_{V,W}$, where $V$,$W$ are subharmonic and superharmonic,
respectively, the weak maximum principle implies that $V\leq v_r\leq
W$ a.e. in $B_r$. Thus $v_r \in S(\Omega)$ and $I(u)\leq I(v_r)$.
Therefore,
\begin{equation*}
    \int_{B_r} |\nabla(u-v_r)|^2 = \int_{B_r} (|\nabla u|^2 - |\nabla
    v_r|^2) \leq 2|B_r| = c_0 r^n \qquad \forall~ 0<r\leq \rho.
\end{equation*}
for some dimensional constant $c_0$. Whence a standard dyadic
argument in the spirit of \cite[Theorem 5.3.6]{Morrey} yields
\begin{equation*}
    \fint_{B_{r/4}}|\nabla u|^2 \leq
    C(1+\rho^{-1})(1+\log^2(\rho/r)) \qquad \forall ~ 0<r\leq \rho,
\end{equation*}
from which the statement of the proposition follows as in
\cite[Theorem 3.5.2]{Morrey}.


\end{proof}

\begin{coro}\label{Harmcty}
The function $u$ is harmonic in $\Omega_{\text{in}}(u)=\{|u|<1\}$,
subharmonic in $\{u<1\}$ and superharmonic in $\{u>-1\}$.
\end{coro}
\begin{proof}
From the previous proposition we know that $u$ is continuous in
$\mathcal{B}_{V,W}$, therefore $\Omega_{\text{in}}(u)\subseteq
\mathcal{B}_{V,W}$ is an open set. Thus for any $x\in
\Omega_{\text{in}}(u)$ we can find a small enough closed ball
$\overline{B}=\overline{B}_r(x)\subseteq \Omega_{\text{in}}$. Let
$v$ be the harmonic replacement of $u$ in $B$. Since, $|u|<1$ on
$\del B$, the maximum principle implies that $|v|<1$. Combining the
latter with the fact that harmonic extensions minimize the Dirichlet
energy, we get that $I(u, B) \geq I(v,B)$. However, by minimality,
$I(u,B)\leq I(v,B)$. Hence, $I(u, B) = I(v,B)$, which in turn
implies that
\begin{equation*}
    \|\nabla u \|^2_{L^2(B)} = \|\nabla v\|^2_{L^2(B)}
\end{equation*}
So, $u$ is itself the minimizer of the Dirichlet energy, meaning
that $u$ is harmonic in $B$. Since $x\in \Omega_{\text{in}}$ is
arbitrary, we conclude that $u$ is harmonic in $\Omega_{\text{in}}$.

The fact that $u$ is subharmonic in $\{u<1\}$ and superharmonic in
$\{u>-1\}$ now follows from the mean-value characterization of
sub/super-harmonic functions.
\end{proof}

Before we proceed to establish Lipschitz continuity, let us state
the following definition related to the geometry of the pair of
barriers $V, W$:

\begin{definition}\label{Intertwined}
We call the subsolution-supersolution pair $(V,W)$ \textbf{nicely
intertwined in} the bounded domain $\Omega$ if
\begin{equation*}
    F^{+}(W)\cap\Omega \subset \Omega_{\text{in}}(V) \quad \text{and} \quad F^{-}(V)\cap\Omega
\subset \Omega_{\text{in}}(W).
\end{equation*}
so that $F^{+}(W)\cap\Omega$ stays a positive distance away from
$\{V = \pm 1\}\cap\Omega$ and $F^{-}(V)\cap\Omega$ stays a positive
distance away from $\{W = \pm 1\}\cap\Omega$. We say that
$V,W:\real^n\rightarrow \real$ are \textbf{nicely intertwined
globally} if for every $R>0$, there exists an $h_0 = h_0(R)$ large
enough, such that $(V,W)$ is nicely intertwined in all cylinders
$\Omega = C_{R,h}$ for $h\geq h_0(R)$.
\end{definition}
\begin{figure}[htbp]
    \centering
    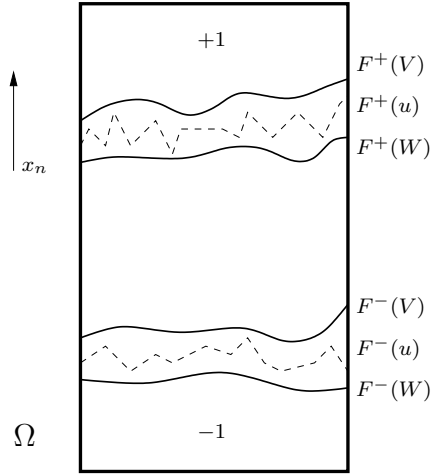
    \caption{The free boundaries of a nicely intertwined pair $(V,W)$ in a cylinder $\Omega$.}
\end{figure}

\begin{prop}\label{propLipMini} Let $D\Subset D' \Subset \Omega$ be compactly contained cylinders and suppose $(V, W)$ is nicely intertwined
in $\Omega$. Then there exists a constant $K$, depending on $n$,
$d(\del D, \del D')$, $d(F^{+}(W)\cap D', F^{-}(V)\cap D')$ and the
Lipschitz constant of $V,W$ in $D'$, such that $|\nabla u|\leq K$ in
$D$. That is, $u$ is Lipschitz-continuous in $D$.
\end{prop}
\begin{proof}
Since $u\in H^1(\Omega)$ and $|u|\leq 1$,
\begin{equation*}
    \nabla u = \nabla u \chrc{|u|<1} \quad \text{a.e.}
\end{equation*}
Thus, it suffices to bound the gradient at points
$x_0\in\Omega_{\text{in}}(u)\cap D$. Let $B_r = B_r(x_0)$ be the
largest ball contained in $\Omega_{\text{in}}\cap D'$, centered at
$x_0$. We may restrict our attention to the situation when $r<
\frac{d(\del D,\del D')}{2}$; for otherwise, using the gradient
estimate for harmonic functions,
\begin{equation*}
    |\nabla u(x_0)| \leq \frac{C}{r} \fint_{B_r} |u| \leq
    \frac{C}{d(\del D,\del D')}.
\end{equation*}
Obviously, in the situation when $r< \frac{d(\del D,\del D')}{2}$,
$B_r$ must touch the free boundary $F(u)$ and not the fixed boundary
$\del D'$.

Assume $B_r(x_0)$ touches $F^{+}(u)$ at a point $x_1$. We consider
two cases determined by how close $x_1$ is to $F^+(V)$.
\begin{itemize}
\item Assume $x_1$ is relatively close to $F^{+}(V)$:
\begin{equation*}
    d(x_1, F^{+}(V)) = |x_2 - x_1| \leq r/2 \quad \text{for some}~x_2 \in
    F^+(V).
\end{equation*}
Note that $|x_2 - x_0| \leq 3r/2 < d(D,D')$, thus the segment
between $x_0$ and $x_2$ is contained in $D'$. The Lipschitz
continuity of $V$ in $D'$ yields
\begin{align*}
    1-u(x_0) & \leq 1 - V(x_0) \leq \|\nabla
    V\|_{L^{\infty}(D')} |x_2 - x_0| \leq  \|\nabla
    V\|_{L^{\infty}(D')} 3r/2.
\end{align*}
Because $1-u\geq 0$ is harmonic in $B_r(x_0)$, Harnack's inequality
implies that
\begin{equation*}
    1-u \leq c (1-u)(x_0) \leq c'\|\nabla V\|_{L^{\infty}(D')} r
    \quad \text{in} \quad B_{r/2}.
\end{equation*}
Hence, by the gradient estimate for harmonic functions we get the
desired
\begin{equation*}
    |\nabla u|(x_0) = |\nabla (1-u)|(x_0) \leq \frac{C}{r}
    \fint_{B_{r/2}(x_0)} (1-u) \leq C \|\nabla V\|_{L^{\infty}(D')}.
\end{equation*}

\item Assume that $d(x_1, F^+(V)) > r/2$. Certainly, $B_{r/2}(x_1)\subseteq
D'$, as $r\leq d(\del D,\del D')/2$. We may also assume that $r\leq
L= d(F^{+}(W)\cap D', F^{-}(V)\cap D')$, for otherwise the gradient
estimate for harmonic functions will immediately give us
\begin{equation*}
    |\nabla u|(x_0) \leq \frac{C}{L}.
\end{equation*}
With these assumptions in mind we see that $B_{r/2}(x_1) \subseteq
\Omega_{\text{in}}(V)\cap D'$, so that $u>-1$ on $B_{r/2}(x_1)$.

Let $v$ be the harmonic replacement of $u$ in $B_{r/2}(x_1)$. By the
strong maximum principle $|v| < 1$, so by minimality,
\begin{equation}\label{LipMini}
    \int_{B_{r/2}(x_1)}|\nabla (u-v)|^2 \leq
    \int_{B_{r/2}}(\chrc{|v|<1} - \chrc{|u|<1}) = |B_{r/2}(x_1)\cap
    \{u=1\}|.
\end{equation}

Now the argument for Lipschitz continuity of \cite{AC} goes through.
It is based on the following bound for the measure of
$B_{r/2}(x_1)\cap
    \{u=1\}$:
\begin{equation}\label{AClipIneq}
|B_{r/2}(x_1)\cap \{u=1\}| \Big(\fint_{\del
B_{r/2}(x_1)}(1-v)\Big)^2 \leq C r^2 \int_{B_{r/2}(x_1)} |\nabla
(u-v)|^2.
\end{equation}
Hence, \eqref{LipMini} and \eqref{AClipIneq} imply
\begin{equation*}
\fint_{\del B_{r/2}(x_1)}(1-v) \leq Cr.
\end{equation*}
Let $x_3$ be a point on the segment between $x_0$ and $x_1$, which
is at a distance $r/4$ from $x_1$. According to Corollary
\ref{Harmcty}, $u$ is superharmonic in $B_{r/2}(x_1)$, so
$1-u(x_3)\leq 1-v(x_3)$. On the other hand, using a Poisson kernel
estimate
\begin{equation*}
    1 - v(x_3) \leq C \fint_{\del B_{r/2}(x_1)} (1-v) \leq C r
\end{equation*}
for some dimensional constant $C$. Then by Harnack inequality,
\begin{equation*}
    \sup_{B_{4r/5}(x_0)}(1-u) \leq
    C'(1-u(x_3))\leq C'r.
\end{equation*}
Applying a gradient estimate, we can conclude $|\nabla u|(x_0) \leq
C$.
\end{itemize}
The case when the ball $B_r(x_0)$ touches $F^{-}(u)$ is treated
analogously.
\end{proof}

\subsection{Construction of a global minimizer.}\label{Sect_Glob}

 Take an increasing sequence of cylinders $\Omega_k = C_{R_k,h_k}$
with $R_k,h_k\nearrow \infty$ and let $u_k$ be the minimizers to
$I(\cdot, \Omega_k)$ over $S(\Omega_k)$ constructed in Proposition
\ref{ExiMin}. If $(V,W)$ is a nicely intertwined pair, Proposition
\ref{propLipMini} implies that (for all large enough $k$) $u_k$ are
uniformly Lipschitz-continuous on compact subsets of $\real^n$.
Therefore, one can extract a subsequence (call it again $\{u_k\}$)
which converges uniformly to a globally defined, locally Lipschitz
continuous function $u:\real^n \to \real$, so that in addition
\begin{equation*}
    \nabla u_k \rightharpoonup \nabla u \quad \text{weakly in} \quad
L^{\infty}_{\text{loc}}(\real^n).
\end{equation*}

\begin{prop}\label{GlobMini} Assume that $V, W$ is a pair of a globally defined subsolution and supersolution to \eqref{FBP}, which are monotonically increasing in $x_n$ and nicely intertwined
with $V\leq W$. Then the locally Lipschitz-continuous function
$u:\real^n\to \real$, produced above, is monotonically increasing in
the $x_n$-variable, satisfies $V\leq u \leq W$ and is harmonic in
$\{u>0\}$. Moreover, for any cylinder $\Omega \subset \real^n$ $u$
minimizes $I(\cdot, \Omega)$ among all competitors $v\in S(\Omega)$
such that $v-u\in H^1_0(\Omega)$.
\end{prop}
\begin{proof}
The first three properties follow from the uniform convergence $u_k
\to u$ on compact sets. Let us concentrate on the minimization
property: assume that there exists a cylinder $\Omega$ and a
competitor $v\in S(\Omega)$, $v-u\in H_0^1(\Omega)$ such that
\begin{equation*}
    I(v,\Omega) \leq I(u,\Omega) - \delta
\end{equation*}
for some $\delta >0$. Denote by $\mathcal{N}_t(\Omega)$ the
$t$-thickening of $\Omega$:
\begin{equation*}
    \mathcal{N}_t(\Omega) = \{x:\in \real^n: \text{dist}(x,\bar{\Omega}) <
    t\}.
\end{equation*}
For $k$ large enough so that $\Omega\Subset \Omega_k$ construct the
following competitor $v_k: \Omega_k \to \real$ for $u_k$:
\begin{equation*}
    v_k(x) = \left\{
    \begin{array}{cl}
    v(x) & \text{in} \quad \Omega \\
   \left(1- \frac{d(x, \bar{\Omega})}{t} \right) u(x) + \frac{d(x,\bar{\Omega})}{t} u_k(x) &
    \text{in} \quad A_{t} := \mathcal{N}_{t}(\Omega) \setminus \Omega \\
    u_k(x) & \text{in} \quad \Omega_k \setminus \mathcal{N}_{t}(\Omega) \\
    \end{array}\right.
\end{equation*}
for some small enough $t>0$ which will be chosen later. It is easy
to check that $v_k\in S(\Omega_k)$, therefore
\begin{align}\label{wantacontra}
    0 & \leq I(v_k, \Omega_k)-I(u_k, \Omega_k) \leq \notag \\
    & \leq \big(I(v,\Omega) - I(u,\Omega)\big) + \big(I(u,\Omega) - I(u_k, \Omega)\big)
    + \big(I(v_k, A_t) - I(u_k, A_t)\big) \notag \\
& \leq - \delta + \big(I(u,\Omega) - I(u_k, \Omega)\big)
    + \big(I(v_k, A_t) - I(u_k, A_t)\big).
\end{align}
By the lower semicontinuity of $I(\cdot,\Omega)$ there exists a
subsequence $u_{k_l}$ such that
\begin{equation*}
    I(u, \Omega) - I(u_{k_l},\Omega) < \delta/2.
\end{equation*}
Now we claim that we can choose $t$ so small that for all $l$ large
enough
\begin{equation}\label{Annulusesti}
I(v_{k_l}, A_t) - I(u_{k_l}, A_t) < \delta/2
\end{equation}
which will lead to a contradiction in \eqref{wantacontra}. Indeed,
$|\nabla u_k |\leq K$ is uniformly bounded on some fixed large
cylinder $\Omega' \supseteq \mathcal{N}_t(\Omega)$, so
\begin{equation*}
    I(u_{k_l},A_t) \leq (K+1)|A_t| \leq C t.
\end{equation*}
Also, $|\nabla u| \leq K$ on $\Omega'$
\begin{align*}
    |\nabla v_{k_l}| & = \left|\nabla\Big( u + \frac{d(x,\Omega)}{t} (u_{k_l}-u) \Big)
\right| \leq |\nabla u| + \frac{1}{t} |u_{k_l}-u| + |\nabla u_{k_l}
- \nabla u| \\
& \leq 3 K + \frac{1}{t} |u_{k_l}-u|.
\end{align*}
Thus, if $\eps_k = \sup_{\Omega'}|u_k - u| $
\begin{equation*}
I(v_{k_l},A_t) \leq C'|A_{t}|(1 + \eps_{k_l}/t + (\eps_{k_l}/t)^2)
\leq C t.
\end{equation*}
for all $l$ large enough, so that $\eps_{k_l}\leq t$. Thus, if we
choose $t< \delta/(2C)$, the estimate \eqref{Annulusesti} will be
satisfied for all $l$ large enough.
\end{proof}


\subsection{Regularity of global minimizers.}

As mentioned in the introduction, there is an intimate connection
between the energy functional $I$ and the standard one-phase energy
functional
\begin{equation*}
    I_0(u, \Omega) = \int_{\Omega} |\nabla u|^2 + \chrc{\{u>0\}}
    \qquad u\in H^1(\Omega).
\end{equation*}
as well as between the notions of viscosity (sub-/super-) solutions
to \eqref{FBP} and \eqref{FBP_0}. Recall,

\begin{definition}\label{defVisc0}
A viscosity solution to \eqref{FBP_0} is a non-negative continuous
function $u$ in $\Omega$ such that
\begin{itemize}
\item $\Delta u = 0$ in $\Omega_p(u)$;
\item If there is a tangent ball $B$ to $F_p(u)$ at some $x_0\in F_p(u)$
from either the positive or zero side, then
\begin{equation*}
    u(x) = \langle x-x_0, \nu \rangle^+ + o(|x-x_0|) \quad
\text{as}\quad x\to x_0,
\end{equation*}
where $\nu$ is the unit normal to $\del B$ at $x_0$ directed into
$\Omega_p(u)$.
\end{itemize}
Equivalently, a viscosity solution cannot be touched from above by a
strict classical supersolution or from below by strict classical
subsolution at a free boundary point.
\end{definition}

\begin{definition}
A viscosity subsolution (resp. supersolution) to \eqref{FBP_0} is a
non-negative continuous function $v$ in $\Omega$ such that
\begin{itemize}
\item $\Delta u \geq 0$ in $\Omega_p(u)$;
\item If there is a tangent ball from the positive side $B\subset \Omega_p(v)$ (resp.
zero side $\Omega_n(v)$) to $F_0(v)$ at some $x_0\in F_0(v)$ and
$\nu$ denotes the unit inner (resp. outer) normal to $\del B$ at
$x_0$, then
\begin{equation*}
    v(x) = \alpha \langle x-x_0, \nu\rangle^+ + o(|x-x_0|) \quad
\text{as}\quad x\to x_0,
\end{equation*}
for some $\alpha \geq 1$ (resp. $\alpha \leq 1$).
\end{itemize}
If $\alpha$ is strictly greater (resp. smaller) than $1$, then $v$
is called a \textbf{strict} viscosity subsolution (resp.
supersolution).
\end{definition}

\begin{remark}\label{NotaBeneMini}
Indeed, suppose $u$ minimizes $I(\Omega)$ among all $v \in
H^1(\Omega)$, such that $V\leq v \leq W$, where $V$ is a subsolution
and $W$ is a supersolution to \eqref{FBP} and let $D \subseteq
\Omega$ be a (regular enough) domain such that either
\[
D\cap \{V=-1\} = \emptyset \quad \text{or} \quad D \cap \{W=1\} =
\emptyset.
\]
In the first case, we readily see that $1-u$ minimizes $I_0(v_0, D)$
among all admissible $1-W \leq v_0 \leq 1-V$, with $1 - W$ being a
subsolution and $1 - V$ -- a supersolution to \eqref{FBP_0} in $D$.
Similarly, in the second case, $u_0 + 1 $ minimizes $I_0(v_0, D)$
among all admissible $V+1 \leq v_0 \leq W+1$, with $V+1$ being a
subsolution and $W+1$ -- a supersolution to \eqref{FBP_0} in $D$.
\end{remark}

Below we will collect the regularity results concerning constrained
minimizers of $I_0$, developed by \cite{DS}. For completeness, we
will lay out the natural sequence of establishing regularity:
starting from weaker notions and ending with the optimum, classical
regularity.

When a viscosity solution $u$ to \eqref{FBP_0} arises from a
minimization problem, $u$ exhibits a non-degenerate behaviour at the
free boundary in the sense that $u$ grows linearly away from its
free boundary in the positive phase. To make that statement precise
we need the following definition.
\begin{definition}
A continuous nonnegative function $u$ is non-degenerate along its
free boundary $F_p(u)$ in $\Omega$ if for every $G\Subset \Omega$,
there exists a constant $K = K(G) >0$ such that for every $x_0\in
F_p(u)\cap G$ and every ball $B_r(x_0)\subseteq G$,
\begin{equation*}
    \sup_{B_r(x_0)} u \geq K r.
\end{equation*}
\end{definition}
On the way to establishing strong regularity properties for free
boundary $F_p(u)$ one needs certain weaker, measure-theoretic
notions of regularity.
\begin{definition}
The free boundary $F_p(u)$ satisfies the density property (D) if for
every $G\Subset \Omega$ there exists a constant $c = c(G)>0$ such
that for every ball $B_r\subseteq G$, centered at a free boundary
point,
\begin{equation*}
    c\leq \frac{|B_r \cap \Omega_p(u)|}{|B_r|} \leq 1-c.
\end{equation*}
\end{definition}
Here we should also recall the notion of nontangentially accessible
(NTA) domains \cite{JK}, which admit the application of the powerful
boundary Harnack principles.
\begin{definition}
A bounded domain $D\subset \real^n$ is NTA if for some constants
$M>0$ and $r_0>0$ it satisfies the following three conditions:
\begin{itemize}
\item (Corkscrew condition). For any $x\in \del D$, $r<r_0$, there
exists $y=y_r(x) \in D$ such that $M^{-1}r < |y-x| < r$ and
$\text{dist}(y, \del D) > M^{-1}r$;
\item (Density condition). The Lebesgue density of $D^c = \real^n \setminus D$ at every point $x\in D^c$ is uniformly bounded from
below by some positive $c>0$
\begin{equation*}
    \frac{|B_r(x)\cap D^c|}{|B_r(x)|} \geq c \quad \forall~ x\in D^c
\quad 0<r<r_0;
\end{equation*}
\item (Harnack chain condition) If $x_1, x_2 \in D$,
$\text{dist}(x_i,\del D)>\eps$, $i=1,2$ and $|x_1 - x_2| < m \eps$
there exists a sequence of $N = N(m)$ balls $\{B_{r_j}\}_{j=1}^N$ in
$D$, such that $x_1 \in B_{r_1}$, $x_2 \in B_{r_N}$, successive
balls intersect and $M^{-1} r_j < \text{dist}(B_{r_j},\del D) < M
r_j$, $j = 1, \ldots, N$.
\end{itemize}
\end{definition}

We can now state the regularity results proved in \cite{DS} and
\cite{DSJ} concerning constrained minimizers of $I_0$. We will say
that the triple of functions $(V_0, u_0, W_0)$ defined on a vertical
right cylinder $\Omega = C_{R,h}$ satisfies the hypotheses
H$(\Omega)$ if
\begin{itemize}
\item $V_0$ is a strict classical subsolution and $W_0$ is a strict
classical supersolution to \eqref{FBP_0} in $\Omega$ such that
$V_0\leq W_0$, $\del_{x_n} V_0
> 0$ in $\{V_0>0\}$ and $\del_{x_n} W_0 >0$ in $\{W_0 > 0\}$.
Moreover, $F_0(V_0), F_0(W_0)$ are a positive distance away from the
top and bottom sections of $\Omega$.

\item The function $u_0$ is monotonically increasing in $x_n$ and minimizes $I_0(\cdot,
\Omega)$ among all competitors $v\in H^1(\Omega)$ such that $v-u \in
H^1_0(\Omega)$ and $V_0 \leq v \leq W_0$ in $\Omega$.
\end{itemize}



\begin{theorem}[\cite{DS}, \cite{DSJ}]\label{Regularity}
If $(V_0, u_0, W_0)$ satisfies the hypotheses H$(\Omega)$, then:
\begin{itemize}
\item \cite{DS}, $u$ is Lipschitz-continuous and non-degenerate along its free boundary $F_p(u)$;
\item \cite{DS}, $F_p(u)$ satisfies the density property (D);
\item \cite{DS}, $F_p(u)$ touches neither $F_p(V_0)$ nor $F_p(W_0)$;
\item \cite{DS}, $u$ is a viscosity solution to \eqref{FBP_0};
\item \cite{DS}, For any vertical cylinder $D\Subset \Omega$ the positive phase $D\cap \Omega_p(u)$ is an NTA domain;
\item \cite{DS}, The free boundary $F_p(u)\cap C_{\frac{3R}{4}, h}$ is given by the graph of a
continuous function $\phi$, $F_p(u) = \{(x',x_n): |x'|<3R/4, x_n =
\phi(x')\}$.
\item \cite{DSJ}, If $\max_{|x'|<3R/4} |\phi(x)| \leq h-\eps$, and $\eps \ll R < h$ then
\begin{equation*}
    \sup_{|x'|<\eps/2}|\nabla \phi|\leq C,
\end{equation*}
where $C$ depends on the dimension $n$, the Lipschitz constant of
$u$, on $h, \eps$ and the NTA constants of $\Omega_p(u)\cap C_{3R/4,
h}$. By the work of Caffarelli \cite{Caf}, this implies $\phi(x')$
is smooth in $\{|x'|\leq \eps/4\}$.
\end{itemize}
\end{theorem}

Let us revert our attention to the original problem. According to
Proposition \ref{Sec_ViscSup}, we are in possession of a pair of a
strict classical supersolution $W:\real^9 \to \real$ and a strict
classical subsolution $V:\real^9 \to \real$, such that $V\leq W$,
both are monotonically increasing in the $x_9$-variable (strictly
increasing in that direction when away from their $\pm 1$ phases),
and are, in addition, nicely intertwined (see Definition
\ref{Intertwined}). By Proposition \ref{GlobMini}, we can then
construct a globally defined, monotonically increasing in $x_9$
function $u:\real^9 \to \real$, such that $u$ minimizes $I(\cdot,
\Omega)$ among $v\in S(\Omega)$ for any vertical cylinder $\Omega$.
Taking into account the observations we made in Remark
\ref{NotaBeneMini}, we can utilize the regularity results (Theorem
\ref{Regularity}) once we simply show that around every free
boundary point $x_+\in F^+(u)$ there exists a vertical cylinder
$\Omega_+\ni x_+$, such that $\Omega_+ \cap \{V = -1\} = \emptyset$
and around every free boundary point $x_-\in F^-(u)$ there exists a
cylinder $\Omega_-\ni x_-$, such that $\Omega_- \cap \{W = 1\} =
\emptyset$. Furthermore, we'll need to show $F^+(u) \cap \Omega_+$
and $F^{-}(u) \cap \Omega_-$ stay a positive distance away from the
top and bottom of $\Omega_+$, respectively $\Omega_-$. This is the
content of the next lemma.

\begin{lemma} Let $V, W$ be the strict subsolution/supersolution provided by Corollary \ref{punches} and $u$ -- the function
constructed in \S \ref{Sect_Glob}. For every $y'\in \real^8$ there
exist an $y_n\in\real$ and an $R>0$ small enough such that
\begin{equation*}
    -\frac{1}{2}\leq V(x',y_n)\leq W(x',y_n) \leq \frac{1}{2}
\end{equation*}
for all $x'\in \real^8$ with $|x'-y'|\leq R$. Thus, if
\begin{align*}
    \Omega_{+}(y') & = \{(x',x_n): |x'-y'|<R,~ 0< x_n-y_n <
h\}\\
\Omega_-(y') & =\{(x',x_n): |x'-y'|<R,~ 0 < y_n-x_n < h\}
\end{align*}
and $h>0$ is large enough, the monotonicity of $V$ and $W$ in the
$x_n$-direction guarantees that
\begin{itemize}
\item $\overline{\Omega_+(y')} \cap \{V=-1\} = \emptyset$ and $F^+(V),
F^+(W)$ exit from the side of $\Omega_+$;
\item $\overline{\Omega_-(y')} \cap \{W=1\} = \emptyset$ and $F^-(V),
F^-(W)$ exit from the side of $\Omega_-$.
\end{itemize}
In particular, $(1-W, 1-u, 1-V)(x', -x_n)$ satisfy hypotheses
H$(\Omega_+(y'))$, while $(V+1, u+1, W+1)(x',x_n)$ satisfy
hypotheses H$(\Omega_-(y'))$.

\end{lemma}
\begin{proof}
Fix $y'\in \real^8$. Since $W(y',x_n) = \pm 1$ for all large
positive (negative) $x_n$, there certainly exists a $y_n$ such that
$W(y',y_n)=\frac{1}{4}$. For a small enough $R>0$, we can ensure
$0\leq W(x',y_n) \leq \frac{1}{2}$ whenever $|x'-y'|<R$. Since, $W$
and $V$ were constructed so that
\[
    0\leq W(x)-V(x) \leq \frac{1}{2} \qquad \forall~ x\in \real^9 \quad (\text{Corollary \ref{punches}})
\]
we see that
\begin{equation*}
    -\frac{1}{2} \leq V(x',y_n) \leq W(x',x_n) \leq \frac{1}{2}
\qquad |x'-y'|<R.
\end{equation*}
\end{proof}

Localizing at any free boundary point, we immediately invoke the
lemma above and the regularity results (Theorem \ref{Regularity}),
establishing the desired Theorem \ref{theBread}.


\appendix
\section{Supersolutions for $J_{\Gamma_{\infty}}$}\label{append1}

The two objectives of this appendix are
\begin{itemize}
\item To state the results of Del Pino, Kowalczyk and Wei \cite{PKW}
concerning supersolutions for the linearized mean curvature operator
on $F_{\infty}$:
\begin{equation*}
    H'[F_{\infty}](\phi) = \left.\frac{d}{dt}\right|_{t=0} H[F_{\infty}+
    t\phi] = \text{div}\left(\frac{\nabla \phi}{\sqrt{1 + |\nabla F_{\infty}|^2}} - \frac{(\nabla F_{\infty}\cdot \nabla \phi)\nabla F_{\infty}}{(1 + |\nabla F_{\infty}|^2)^{3/2}} \right)
\end{equation*}
where $H[\cdot]$ is the mean curvature operator (MCO), and to
describe the relation between $H'[F_{\infty}]$ and the Jacobi
operator $J_{\Gamma_{\infty}}$ on $\Gamma_{\infty}$.

\item To obtain the refined estimate $H_{\infty,3} = O\left(\frac{|\theta-\pi/4|}{1+r^3}\right)$.
\end{itemize}
We recall that in polar coordinates $(r,\theta)$,
$F_{\infty}(r,\theta) = r^3g(\theta)$ and the function $g(\theta)$
is smooth and satisfies:
\begin{enumerate}
\item $g(\theta) = -g(\frac{\pi}{2} - \theta)$
\item $g(\theta) = (\theta - \frac{\pi}{4})(1 + c_3(\theta -
\frac{\pi}{4} )^2 + \cdots) \quad \text{near} \quad \theta =
\frac{\pi}{4}$
\item $g'(\theta) > 0 \quad \text{for}~ \theta\in(0,\frac{\pi}{2})
\quad \text{and}\quad g'(0)=g'(\pi/2)=0.$
\end{enumerate}

\subsection{The relation between the Jacobi operator and the linearized MCO}

For a domain $U\subset\real^8$, let $S:U \rightarrow \real$ be an
arbitrary $C^2(U)$ function and denote by $\Sigma = \{(x',x_9)\in
U\times\real: x_9 = S(x')\}$ its graph. Denote the standard
projection onto $\real^9$ by
\[
 \pi:\real^8\times \real \rightarrow \real^8.
\]
We can identify functions $\phi$ defined on $U$ with functions
$\phi_{\Sigma}$ on $\Sigma$ in the usual way:
\begin{equation*}
    \phi_{\Sigma} = \phi\circ\pi.
\end{equation*}
We'll abuse notation and use the same symbol $\phi$ to denote both.
A long, but straightforward computation yields the following
interesting formula relating the linearized MCO associated with $S$
to the Jacobi operator $J_{\Sigma}:= \Delta_{\Sigma} +
|A_{\Sigma}|^2$ on the graph $\Sigma$:
\begin{equation}\label{JacobiandLinearizedMC}
    J_{\Sigma}\left(\frac{\phi}{\sqrt{1+|\nabla S|^2}}\right) =
H'[S](\phi) - \frac{\nabla (H[S]) \cdot \nabla S}{\sqrt{1+|\nabla
S|^2}}\frac{\phi}{\sqrt{1+|\nabla S|^2}}.
\end{equation}
Note that if $\Sigma$ is a minimal graph, i.e. $H[S]=0$, we recover
the well-known relation
\begin{equation}\label{jacoandLinMCmini}
    J_{\Sigma}\left(\frac{\phi}{\sqrt{1+|\nabla S|^2}}\right) =
    H'[S](\phi).
\end{equation}

For our purposes, we would like to estimate the size of the error
term in \eqref{JacobiandLinearizedMC} when $S = F_{\infty}$.

\subsection{Computation of $H[F_{\infty}]$ and $|\nabla (H[F_{\infty}])|$.}
First, we compute $H[F_{\infty}]$. Since
$\text{div}\left(\frac{\nabla F_{\infty}}{|\nabla F_{\infty}|
}\right)=0$,
\begin{align*}
    H[F_{\infty}] & = \text{div}\left({\frac{\nabla F_{\infty}}{\sqrt{1+ |\nabla F_{\infty}|^2}}}\right) =
-\text{div}\left(\frac{\nabla F_{\infty}}{|\nabla
F_{\infty}|}-\frac{\nabla
F_{\infty}}{\sqrt{1+ |\nabla F_{\infty}|^2}}\right) \\
& = -\text{div}\left(\frac{\nabla F_{\infty}}{|\nabla
F_{\infty}|\sqrt{1+|\nabla
F_{\infty}|^2}(|\nabla F_{\infty}|+ \sqrt{1+|\nabla F_{\infty}|^2})}\right) = \\
& = -\frac{\nabla F_{\infty}}{|\nabla F_{\infty}|}\cdot
\nabla\left(\frac{1}{\sqrt{1+|\nabla F_{\infty}|^2}(|\nabla
F_{\infty}|+ \sqrt{1+|\nabla F_{\infty}|^2})}\right) = \frac{\nabla
F_{\infty}}{|\nabla F_{\infty}|}\cdot\frac{\nabla Q}{Q^2},
\end{align*}
where $Q(x') := \sqrt{1+|\nabla F_{\infty}|^2}(|\nabla F_{\infty}|+
\sqrt{1+|\nabla F_{\infty}|^2})$. Note that $Q$ is bounded from
below by
\[
    Q(x') \geq 2 |\nabla F_{\infty}|^2 = 2 r^4(9g(\theta)^2+g'(\theta)^2)\geq 2mr^4,
\]
where $m =
\min_{\theta\in[0,\pi/2]}\big(9g(\theta)^2+g'(\theta)^2\big)
> 0$. Also,
\begin{align*}
    |\nabla |\nabla F_{\infty}||^2 & = (2r \sqrt{9g^2 + g'^2})^2 + \left( r g'\frac{9g' + g''}{\sqrt{9g^2 +
g'^2}}\right)^2 = O(r^2) \\
    |\nabla \sqrt{1+|\nabla F_{\infty}|^2}|^2 & = \left||\nabla F_{\infty}|\frac{\nabla |\nabla F_{\infty}|}{\sqrt{1+|\nabla
F_{\infty}|^2}}\right|^2 =|\nabla F_{\infty}|^2 \frac{|\nabla
|\nabla F_{\infty}||^2}{1+|\nabla F_{\infty}|^2} = O(r^2).
\end{align*}
Thus,
\begin{equation}\label{nablaQ}
    |\nabla Q| = O(r^3)
\end{equation}
and
\begin{equation*}
    |H[F_{\infty}]| \leq \frac{|\nabla Q|}{Q^2} = O(r^{-5}).
\end{equation*}

To compute $|\nabla H[F_{\infty}]|$, observe that

\begin{equation}\label{nablaHprem}
 \nabla (H[F_{\infty}]) = \frac{\nabla (\nabla F_{\infty} \cdot \nabla Q)}{|\nabla
F_{\infty}|Q^2} - H[F_{\infty}] \frac{\nabla|\nabla
F_{\infty}|}{|\nabla F_{\infty}|} - 2 H[F_{\infty}] \frac{\nabla
Q}{Q}.
\end{equation}

The last two summands are obviously $O(r^{-6})$. Let us bound
\begin{align*}
|\nabla (\nabla F_{\infty} \cdot \nabla Q)|^2 = \Big(\del_r\big(
(F_{\infty})_r Q_r &+ r^{-2} (F_{\infty})_{\theta}
Q_{\theta}\big)\Big)^2 \\
& + r^{-2}\Big(\del_{\theta}\big((F_{\infty})_r Q_r + r^{-2}
(F_{\infty})_{\theta} Q_{\theta}\big)\Big)^2.
\end{align*}
Because of \eqref{nablaQ}, $Q_r, r^{-1}Q_{\theta} = O(r^3)$.
Furthermore,
\begin{equation*}
    F_{\infty,rr}, r^{-1}F_{\infty, r\theta}, r^{-2}F_{\infty, \theta\theta} =
O(r),
\end{equation*}
and
\begin{equation*}
    Q_{rr}, r^{-1}Q_{r\theta}, r^{-2}Q_{\theta\theta} =
O(r^2).
\end{equation*}
Thus, $|\nabla (\nabla F_{\infty} \cdot \nabla Q)|^2 = O(r^8)$ and
the first summand in \eqref{nablaHprem} is then
\[
\frac{\nabla (\nabla F_{\infty} \cdot \nabla Q)}{|\nabla
F_{\infty}|Q^2} = O\Big(\frac{r^4}{r^{10}}\Big) = O(r^{-6}),
\]
as well. We conclude
\begin{equation}\label{nablaH}
    |\nabla (H[F_{\infty}])| = O(r^{-6}).
\end{equation}

\subsection{Supersolutions for
$J_{\Gamma_{\infty}}$}\label{PKWsupsinfinity}

In \cite[\S 7.2]{PKW} the authors study the linearized MCO
$H'[F_{\infty}]$ and show that it admits two types of supersolutions
away from the origin. We will call those Type 1 and Type 2 in
parallel with the labels we used in Section \S
\ref{subsec_Supersols}. Because of \eqref{nablaH}, formula
\eqref{JacobiandLinearizedMC} becomes
\begin{equation}\label{JacobiandlinMCOimprov}
    J_{\Gamma_{\infty}}\left(\frac{\phi}{\sqrt{1+|\nabla F_{\infty}|^2}}\right) =
H'[F_{\infty}](\phi) + O\left(\frac{r^{-6}|\phi|}{\sqrt{1+|\nabla
F_{\infty}|^2}}\right)
\end{equation}
so that one can then cook up supersolutions for the Jacobi operator
$J_{\Gamma_{\infty}}$.

\begin{itemize}

\item Type 1 supersolution for the linearized MCO (cf. the proof of Lemma 7.2 in \cite{PKW}): \\
There exists a smooth function $\phi_1 = \phi_1(r,\theta) =
r^{-\eps}q_1(\theta)$ with $q_1(\theta)
>0$ and even about $\theta = \pi/4$ that satisfies the differential
inequality
\begin{equation*}
    H'[F_{\infty}](\phi_1) \leq - \frac{1}{r^{4+\eps}} \qquad r>r_0
\end{equation*}
for sufficiently large $r_0$. Thus, \eqref{JacobiandlinMCOimprov}
implies that $h_1 = \frac{\phi_1}{\sqrt{1+|\nabla F_{\infty}|^2}}
\in C^{\infty}(\Gamma_{\infty})$ satisfies
\begin{equation}\label{diffen1}
    J_{\Gamma_{\infty}}h_1 \leq - \frac{1}{r^{4+\eps}}+ O(r^{-8-\eps}) \leq
-\frac{1}{1+ r^{4+\eps}}
\end{equation}
for sufficiently large $r>r_0$.

\item Type 2 supersolution for the linearized MCO (cf. the proof of Lemma 7.3 in \cite{PKW}):\\
For every $\frac{1}{3} < \tau < \frac{2}{3}$ there exists a function
$\phi_2 = \phi_2(r,\theta) = r q_2(\theta)$, defined for
$\theta\in\{\pi/4 < \theta \leq \pi/2 \}$, such that
\begin{equation*}
    H'[F_{\infty}](\phi_2) \leq - \frac{g(\theta)^{\tau}}{r^{3}}, \quad \theta \in
    (\frac{\pi}{4},\frac{\pi}{2}], \quad r>r_0
\end{equation*}
for sufficiently large $r_0$. Moreover, $q_2(\theta)$ is smooth in
$(\frac{\pi}{4},\frac{\pi}{2}]$ and has the following expansion near
$\theta = \frac{\pi}{4}^+$:
\[
    q_2(\theta) = (\theta - \frac{\pi}{4})^{\tau}(a_0 + a_2(\theta - \frac{\pi}{4})^2+
\cdots) \quad \text{where}\quad a_0>0.
\]
Therefore, $h_2 = \frac{\phi_2}{\sqrt{1+|\nabla F_{\infty}|^2}}$
satisfies the differential inequality
\begin{equation}\label{differen2}
    J_{\Gamma_{\infty}} h_2 \leq - \frac{g(\theta)^{\tau}}{r^{3}} +
O(r^{-7}g(\theta)^{\tau}) \leq - \frac{g(\theta)^{\tau}}{1+ r^{3}}
\end{equation}
in $\theta \in (\frac{\pi}{4},\frac{\pi}{2}]$ for sufficiently large
$r>r_0$.

\end{itemize}

\subsection{Computation of $H_{\infty,3}$} We will compute the
second fundamental form $A$ of the graph $\Gamma_{\infty}$ and then
estimate the sizes of the principal curvatures.

Let $y = (\hat{u} r \cos\theta, \hat{v} r\sin\theta, r^3
g(\theta))\in \Gamma_{\infty}$, with $\hat{u}, \hat{v}\in
S^3\subset\real^4$ and consider local parametrizations
$\tilde{u}(t_1, t_2, t_3)$, and $\tilde{v}(s_1,s_2,s_3)$ of $S^3$
around $\hat{u}$ and $\hat{v}$, respectively, such that
\begin{align*}
   \tilde{u}(0) & = \hat{u} \qquad \del_{t_i}\tilde{u}(0) = \tau_i, \qquad i=1,2,3\\
   \tilde{v}(0) & = \hat{v} \qquad \del_{s_i}\tilde{v}(0) = \sigma_i, \qquad i=1,2,3
\end{align*}
where $\{\tau_i\}, \{\sigma_i\}$ are orthonormal bases for
$T_{\hat{u}}S^3$ and $T_{\hat{v}}S^3$, respectively. Then
\begin{equation}\label{localpara}
    P(r,\theta,t_i,s_i) = (\tilde{u}r \cos\theta, \tilde{v}r \sin\theta, r^3 g(\theta))
\end{equation}
defines a local parametrization of $\Gamma_{\infty}$ near $y$.

In the system of coordinates $\{r,\theta, t_i, s_i\}$ the metric
tensor near $y$ takes the form $\mathsf{g} = \mathsf{g}_2 \oplus
\mathsf{g}_{\text{sym}}$, where
\begin{equation}\label{metrictensor}
        \mathsf{g}_2 = \left(\begin{array}{cc}
1+9r^4g^2 & 3r^5 g g'  \\
3r^5 g g' & r^2(1+g'^2 r^4) \\
\end{array}\right), \mathsf{g}_{\text{sym}} = \left(\begin{array}{cc}
(r^2\cos^2\theta)U_3  &  \\
& (r^2\sin^2\theta)V_3
\end{array}\right).
\end{equation}
In the expression for $\mathsf{g}_{\text{sym}}$ above, $U_3, V_3$
are $3\times 3$ matrices that depend only on $\{t_k\},\{s_k\}$ with
$U_3(0) = V_3(0) = I_3$, the identity $3\times 3$ matrix. We will
also need the inverse of $\mathsf{g}$, $\mathsf{g}^{-1} =
\mathsf{g}_2^{-1} \oplus \mathsf{g}_{\text{sym}}^{-1}$, where
\begin{equation}\label{metrictensorinv}
\begin{aligned}
    \mathsf{g}_2^{-1}(y) &= \frac{1}{\sigma}\left(\begin{array}{cc}
1+r^4g'^2 & -3r^3 g g' \\
-3r^3 g g' & r^{-2}+9g^2 r^2 \\
\end{array}\right) \qquad \sigma:= 1+r^4(9 g^2 + g'^2) \\
\mathsf{g}_{\text{sym}}^{-1}(y) &= r^{-2}\left(\begin{array}{cc}
(\cos\theta)^{-2} I_3 & \\
 & (\sin\theta)^{-2} I_3 \\
\end{array}\right).
\end{aligned}
\end{equation}
The unit normal $\nu(y)$ is given by
\begin{align*}
    \nu(y) & = -\frac{1}{\sqrt{\sigma}}\big((F_{\infty})_u \tilde{u}, (F_{\infty})_v
    \tilde{v}, -1\big) = \\
    & =-\frac{1}{\sqrt{\sigma}}\big(r^2(3g\cos\theta - g'\sin\theta) \tilde{u},
    r^2(3g\cos\theta + g'\sin\theta)
    \tilde{v}, -1\big).
\end{align*}
We calculate the second fundamental form $(A_{\infty})_{ij} =
-\del_i P \cdot \del_j \nu = \del_{ij}P\cdot \nu$ at $y$ to be
$A_{\infty} = (A_{\infty})_2 \oplus (A_{\infty})_{\text{sym}}$,
where
\begin{align*}
(A_{\infty})_2(y) & = \frac{1}{\sqrt{\sigma}}\left(\begin{array}{cc}
6r g & 2r^2g' \\
2r^2g' & r^3(3g + g'')
\end{array}\right)
\\
(A_{\infty})_{\text{sym}}(y) & =
\frac{1}{\sqrt{\sigma}}\left(\begin{array}{cc}
r\cos\theta (F_{\infty})_u  I_3 & \\
 & r\sin\theta(F_{\infty})_v I_3
\end{array}\right).
\end{align*}
The principal curvatures of $\Gamma_{\infty}$ at $y$ are the
eigenvalues of the matrix $A_{\infty}\mathsf{g}^{-1}$,
i.e. the eigenvalues $\mu_1$, $\mu_2$ (each of multiplicity $3$) of
$(A_{\infty})_{\text{sym}}
    (\mathsf{g}_{\text{sym}})^{-1}$:
\begin{align*}
    \mu_1 &= \frac{1}{\sqrt{\sigma}}
    \frac{(F_{\infty})_u}{r\cos\theta} =  \frac{r}{\sqrt{\sigma}}(3g - g'\tan\theta) \\
    \mu_2 &=
    \frac{1}{\sqrt{\sigma}}\frac{(F_{\infty})_v}{r\sin\theta} =\frac{r}{\sqrt{\sigma}}(3g + g'\cot\theta)
\end{align*}
and the eigenvalues $\lambda_1, \lambda_2$ of
\begin{equation*}
    (A_{\infty})_2 (\mathsf{g}_2)^{-1} =
\frac{1}{\sigma^{3/2}}\left(\begin{array}{cc}
O(r(\theta - \frac{\pi}{4})) & *\\
* & O(r^5 (\theta-\frac{\pi}{4}))
\end{array}\right).
\end{equation*}
Since $g'\cot\theta = O(1)$ and $g'\tan\theta = O(1)$, we see that
$\mu_1$ and $\mu_2$ are $O((1+r)^{-1})$. Note further that
\begin{equation}\label{mu1plusmu2}
    \mu_1 + \mu_2 = \frac{r}{\sqrt{\sigma}}(6g +2g'\cot2\theta) =
    O\Big(\frac{\theta-\pi/4}{1+r}\Big).
\end{equation}
On the other hand, $\lambda_{1,2} = O((1+r)^{-1})$ as well, since
\begin{align}
    \lambda_1 + \lambda_2 & = \text{Trace}((A_{\infty})_2
    (\mathsf{g}_2)^{-1}) = O\Big(\frac{\theta-\pi/4}{1+r}\Big) \label{lam1pluslam2}\\
    \lambda_1 \lambda_2 & = \det (A_{\infty})_2 \det(\mathsf{g}_2)^{-1}
    = O(1)\frac{1}{r^2\sigma} = O(r^{-6}).\notag
\end{align}

Now \eqref{mu1plusmu2} and \eqref{lam1pluslam2}, combined with the
fact that the principal curvatures are all of order $O((1+r)^{-1})$
imply that
\begin{equation*}
    H_{\infty,3} = 3 (\mu_1 + \mu_2)(\mu_1^2 - \mu_1\mu_2 + \mu_2^2)
    + (\lambda_1 + \lambda_2)(\lambda_1^2 - \lambda_1\lambda_2 +
    \lambda_2^2) = O\Big(\frac{\theta-\pi/4}{1+r^3}\Big).
\end{equation*}

\section{Bounds on the gradient and hessian of $h(r,\theta)$ on
$\Gamma_{\infty}$.}\label{append2}

Once again we will make use of the $\{r,\theta, t_i, s_i\}$ system
of coordinates \eqref{localpara} in order to estimate the first and
second covariant derivatives of a function $h=h(r,\theta) \in
C^2(\Gamma_{\infty})$ which depends only on $r$ and $\theta$.

\begin{lemma}\label{gradhess} The gradient and hessian of $h = h(r,\theta) \in C^2(\Gamma_{\infty})$
satisfy:
\begin{align}
    |D_{\Gamma_{\infty}} h| & = O\Big(|\del_r h| + (r^{-1}\vartheta +
    r^{-3})|\del_{\theta}h|\Big) \label{gradrtita} \\
    |D^2_{\Gamma_{\infty}}h | & = O\Big(|\del^2_{r} h| + (r^{-1}\vartheta +
    r^{-3}) |\del^2_{r\theta}h| + (r^{-1}\vartheta +
    r^{-3})^2 |\del^2_{\theta}h|\Big) + \notag \\
    & + r^{-1}O\Big(|\del_r h| + (r^{-1}\vartheta +
    r^{-3})|\del_{\theta} h|\Big) + O\Big((r^{-6} + r^{-2}\vartheta^2)\left|\frac{\del_{\theta}h}{g'}\right|\Big)  \label{hessrtita}
\end{align}
where $\vartheta:=|\theta-\pi/4|$ and the constants in the
$O$-notation depend on $g$.
\end{lemma}
\begin{proof}

In order to carry out the computations, we adopt the standard
Einstein index notation. That way, we write
\begin{align*}
    |D_{\Gamma_{\infty}} h|^2 = h^i h_i \quad \text{and} \quad  |D^2_{\Gamma_{\infty}} h|^2 = h_i{}^j
    h_j{}^i,
\end{align*}
where $i,j$ range over the list of coordinates  $\{r,\theta,
\{t_k\},\{s_k\}\}$
\[
h_i = \del_i h, \quad h^{i} = \mathsf{g}^{ik}h_k  \quad
\text{with}\quad \mathsf{g}^{ij} = (\mathsf{g}^{-1})_{ij}
\]
and
\begin{equation*}
    h_i{}^j = \del_i h^j + \Gamma^j_{ik} h^k.
\end{equation*}
In the expression above, $\Gamma^j_{ik}$ are, of course, the
Christophel symbols:
\begin{equation}
\Gamma^j_{ik} = \frac{1}{2}\mathsf{g}^{jl}\big(\del_i\mathsf{g}_{lk}
+ \del_k\mathsf{g}_{il} - \del_l\mathsf{g}_{ik}
\big).\label{Christo}
\end{equation}
Since $h = h(r,\theta)$, we have $|D_{\Gamma_{\infty}} h|^2 = h^r
h_r + h^{\theta}h_{\theta}$. Using \eqref{metrictensorinv} we
calculate
\begin{align*}
    h^r & = \frac {1+r^4g'^2}{\sigma} h_r -  \frac{3r^3 g g'}{\sigma} h_{\theta} \\
    h^{\theta} & =  -\frac{3r^3 g g'}{\sigma} h_r +
    \frac{r^{-2}+9g^2 r^2}{\sigma}
    h_{\theta}
\end{align*}
and taking into account that $g(\theta) = O(\vartheta)$ and $9g^2 +
g'^2$ is uniformly bounded from above and from below by positive
constants, we conclude
\begin{equation*}
    |D_{\Gamma_{\infty}}h|^2 = O\Big(|h_r|^2 + r^{-1}\vartheta
    |h_r||h_{\theta}| + (r^{-3}+ r^{-1}\vartheta)^2 |h_{\theta}|^2
    \Big)
\end{equation*}
so that \eqref{gradrtita} is verified.

The computation of the hessian is slightly more involved. For
convenience we will denote by Greek letters $\alpha, \beta, \gamma$,
etc. indices that correspond to coordinates $r,\theta$, and by Latin
$l,m, n$, etc. indices that correspond to coordinates
$\{t_i\},\{s_i\}$. First, note that the ``cross term" contribution
$h_{\alpha}{}^l h_l{}^{\alpha} = 0$, because
\begin{equation*}
    h_l{}^{\alpha} = \del_l h^{\alpha} + \Gamma_{l\beta}^{\alpha}
    h^{\beta} = 0,
\end{equation*}
as $\del_l h^{\alpha} = 0$ and
\begin{equation*}
    \Gamma_{l\beta}^{\alpha} = \frac{1}{2} \mathsf{g}^{\alpha\gamma}\big(\del_l\mathsf{g}_{\gamma\beta}
+ \del_{\beta}\mathsf{g}_{l\gamma} -
\del_{\gamma}\mathsf{g}_{l\beta} \big) = 0.
\end{equation*}
Thus, $|D^2_{\Gamma_{\infty}}h|^2 = S_2 + S_{\text{sym}}$, where
\begin{equation*}
    S_2:= h_{\alpha}{}^{\beta}h_{\beta}{}^{\alpha}, \qquad S_{\text{sym}} := h_l{}^m
h_m{}^l.
\end{equation*}
Let us first deal with $S_{\text{sym}}$. We see that
\begin{equation*}
    S_{\text{sym}} = \Gamma_{l\alpha}^m \Gamma_{m\beta}^l h^{\alpha}
    h^{\beta},
\end{equation*}
where
\begin{equation*}
    \Gamma_{l\alpha}^m = \frac{1}{2}\mathsf{g}^{mm}\big(\del_l \mathsf{g}_{m\alpha} + \del_{\alpha}\mathsf{g}_{lm} -\del_m \mathsf{g}_{l\alpha}
    \big) = \frac{1}{2}\mathsf{g}^{mm} \del_{\alpha}
    \mathsf{g}_{lm}.
\end{equation*}
Noting that
\begin{equation*}
        \del_{\alpha}\mathsf{g}_{\text{sym}}(y) =
\left(\begin{array}{cc}
\del_{\alpha}(r^2\cos^2\theta) I_3 &  \\
&\del_{\alpha}(r^2\sin^2\theta) I_3
\end{array}\right),
\end{equation*}
we obtain
\begin{align*}
    S_{\text{sym}} & = (h^r)^2 \frac{6}{r^2} + 2 h^r h^{\theta} \frac{3}{r}(\cot\theta -
    \tan\theta) + (h^{\theta})^2 3(\cot^2\theta + \tan^2\theta).
\end{align*}
A more refined estimation of $h^{\theta}$ gives
\begin{equation*}
    h^{\theta} = O\Big(|g'|\big(|h_r|/r + (r^{-6} + r^{-2}\vartheta^2) |h_{\theta}/g'|\big)\Big)
\end{equation*}
and since
\begin{equation*}
\cot\theta - \tan\theta = O(1/g') \qquad \cot^2\theta + \tan^2\theta
= O(1/g'^2)
\end{equation*}
we conclude
\begin{equation}
 S_{\text{sym}} = O(|h_r|^2 r^{-2} + r^{-2}\vartheta^2 |h_{\theta}|^2 + (r^{-6} + r^{-2}\vartheta^2)^2
 |h_{\theta}/g'|^2)  \label{Ssym}.
\end{equation}

Now we proceed with the computation of $S_2$. For the purpose we
need to calculate the Christophel symbols
$\Gamma_{\alpha\beta}^{\gamma}$. The derivatives of $\mathsf{g}_2$
are
\begin{equation*}
    \del_r\mathsf{g}_2 = \left(\begin{array}{cc}
36r^3g^2 & 15r^4 g g'  \\
15r^4 g g' & 2r + 6 g'^2 r^5 \\
\end{array}\right)
\qquad
 \del_{\theta}\mathsf{g}_2 = \left(\begin{array}{cc}
18r^4 gg' & 3 r^5 (g'^2 + gg'')  \\
3 r^5 (g'^2 + gg'') & 2 r^6 gg' \\
\end{array}\right)
\end{equation*}
which we then plug in \eqref{Christo} to obtain:
\begin{align*}
\left(\begin{array}{c}
\Gamma_{rr}^r  \\
\Gamma_{rr}^{\theta} \\
\end{array}\right) & = \frac{\mathsf{g}_2^{-1}}{2}\left(\begin{array}{c}
\del_r \mathsf{g}_{rr}  \\
2 \del_r \mathsf{g}_{\theta r} - \del_{\theta} g_{rr}
\end{array}\right)  =
\frac{6}{\sigma} \left(\begin{array}{c}
3r^3 g^2  \\
r^2 gg' \\
\end{array}\right) = \left(\begin{array}{c}
O(r^{-1}\vartheta^2)  \\
O(r^{-2} \vartheta)
\end{array}\right)
 \\
\left(\begin{array}{c}
\Gamma_{\theta r}^r  \\
\Gamma_{\theta r}^{\theta} \\
\end{array}\right) & = \frac{\mathsf{g}_2^{-1}}{2}\left(\begin{array}{c}
\del_{\theta} \mathsf{g}_{rr}  \\
\del_r \mathsf{g}_{\theta\theta}\\
\end{array}\right) =\frac{1}{\sigma} \left(\begin{array}{c}
6r^4 gg'  \\
r^{-1} + 3r^3(g'^2 + 3g^2) \\
\end{array}\right) = \left(\begin{array}{c}
O(\vartheta)  \\
O(r^{-1})
\end{array}\right) \\
\left(\begin{array}{c}
\Gamma_{\theta \theta}^r  \\
\Gamma_{\theta \theta}^{\theta} \\
\end{array}\right) & = \frac{\mathsf{g}_2^{-1}}{2}\left(\begin{array}{c}
2\del_{\theta} \mathsf{g}_{r\theta}-\del_r\mathsf{g}_{\theta\theta}  \\
\del_{\theta} \mathsf{g}_{\theta\theta}\\
\end{array}\right) =\frac{1}{\sigma} \left(\begin{array}{c}
r^5(3gg'' - g'^2)-r  \\
r^4g'(g''-3g) \\
\end{array}\right) = \left(\begin{array}{c}
O(r)  \\
O(\vartheta)
\end{array}\right). \\
\end{align*}
For convenience define the following expressions that measure the
magnitude of the first and second derivatives of $h$:
\begin{align*}
    \mathcal{F}(h_{\alpha}) & := |h_r| + (\vartheta r^{-1} +
    r^{-3})|h_{\theta}| \\
    \mathcal{S}(\del^2_{\alpha\beta}h, h_{\alpha}) & := |\del^2_{r}h| +(\vartheta r^{-1} +
    r^{-3})|\del^2_{\theta r}h|+ (\vartheta r^{-1} +
    r^{-3})^2|\del^2_{\theta}h| + r^{-1}\mathcal{F}(h_{\alpha})
\end{align*}
A straightforward computation yields:
\begin{align*}
    h_{r}{}^r & = \del_r h^r + \Gamma_{rr}^r h^{r} +
    \Gamma_{r\theta}^r h^{\theta} = O\big(\mathcal{S}(\del^2_{\alpha\beta}h,
    h_{\alpha})\big) \\
    h_{\theta}{}^{\theta} & = \del_{\theta} h^{\theta} + \Gamma_{\theta r}^{\theta} h^{r} +
    \Gamma_{\theta\theta}^r h^{\theta} =  O\big(\mathcal{S}(\del^2_{\alpha\beta}h,
    h_{\alpha})\big) \\
    h_r{}^{\theta} & =\del_{r} h^{\theta} + \Gamma_{rr}^{\theta} h^{r} +
    \Gamma_{r\theta}^{\theta} h^{\theta} = \\
    & =     O\big(r^{-7} h_{\theta} +
    r^{-6}\del^2_{r\theta}h + \vartheta r^{-1}\mathcal{S}(\del^2_{\alpha\beta}h,
    h_{\alpha})
    \big) + O(\vartheta r^{-2} \mathcal{F}(h_{\alpha})) = \\
    & = O \big((r^{-3} + \vartheta r^{-1}) \mathcal{S}(\del^2_{\alpha\beta}h,
    h_{\alpha})\big) \\
    h_{\theta}{}^{r} & =\del_{\theta} h^{r} + \Gamma_{\theta r}^{r} h^{r} +
    \Gamma_{\theta \theta}^r h^{\theta} = \\
    & = O\big(\del^2_{r\theta} h + r^{-1}\vartheta \del^2_{\theta}h + \vartheta h_{r} + r^{-1}
    h_{\theta}\big) + O\big(\vartheta h_r + (r^{-5} + \vartheta^2 r^{-1})h_{\theta}
    \big) = \\
    & = O \big((r^{-3} + \vartheta r^{-1})^{-1} \mathcal{S}(\del^2_{\alpha\beta}h,
    h_{\alpha})\big),
\end{align*}
whence we conclude
\begin{equation}\label{Stwo}
    S_2 = O\big(\mathcal{S}^2(\del^2_{\alpha\beta}h,
    h_{\alpha})\big).
\end{equation}
Equations  \eqref{Ssym} and \eqref{Stwo} yield the estimate
\eqref{hessrtita} for the hessian.
\end{proof}

\bibliography{FB_Bib}

\begin{thebibliography}{10}

\bibitem{AlbAmbCabre}
Giovanni Alberti, Luigi Ambrosio, and Xavier Cabr{\'e}.
\newblock On a long-standing conjecture of {E}. {D}e {G}iorgi: symmetry in 3{D}
  for general nonlinearities and a local minimality property.
\newblock {\em Acta Appl. Math.}, 65(1-3):9--33, 2001.
\newblock Special issue dedicated to Antonio Avantaggiati on the occasion of
  his 70th birthday.

\bibitem{AC}
H.~W. Alt and L.~A. Caffarelli.
\newblock Existence and regularity for a minimum problem with free boundary.
\newblock {\em J. Reine Angew. Math.}, 325:105--144, 1981.

\bibitem{ACF}
Hans~Wilhelm Alt, Luis~A. Caffarelli, and Avner Friedman.
\newblock Variational problems with two phases and their free boundaries.
\newblock {\em Trans. Amer. Math. Soc.}, 282(2):431--461, 1984.

\bibitem{AmbroCabre}
Luigi Ambrosio and Xavier Cabr{\'e}.
\newblock Entire solutions of semilinear elliptic equations in {$\bold R^3$}
  and a conjecture of {D}e {G}iorgi.
\newblock {\em J. Amer. Math. Soc.}, 13(4):725--739, 2000.

\bibitem{BdGG}
E.~Bombieri, E.~De~Giorgi, and E.~Giusti.
\newblock Minimal cones and the {B}ernstein problem.
\newblock {\em Invent. Math.}, 7:243--268, 1969.

\bibitem{Caf}
Luis~A. Caffarelli.
\newblock A {H}arnack inequality approach to the regularity of free boundaries.
  {I}. {L}ipschitz free boundaries are {$C^{1,\alpha}$}.
\newblock {\em Rev. Mat. Iberoamericana}, 3(2):139--162, 1987.

\bibitem{CafCordo1}
Luis~A. Caffarelli and Antonio C{\'o}rdoba.
\newblock Uniform convergence of a singular perturbation problem.
\newblock {\em Comm. Pure Appl. Math.}, 48(1):1--12, 1995.

\bibitem{CafCordo2}
Luis~A. Caffarelli and Antonio C{\'o}rdoba.
\newblock Phase transitions: uniform regularity of the intermediate layers.
\newblock {\em J. Reine Angew. Math.}, 593:209--235, 2006.

\bibitem{DeGio}
Ennio De~Giorgi.
\newblock Convergence problems for functionals and operators.
\newblock In {\em Proceedings of the {I}nternational {M}eeting on {R}ecent
  {M}ethods in {N}onlinear {A}nalysis ({R}ome, 1978)}, pages 131--188, Bologna,
  1979. Pitagora.

\bibitem{DS}
Daniela De~Silva.
\newblock Existence and regularity of monotone solutions to a free boundary
  problem.
\newblock {\em Amer. J. Math.}, 131(2):351--378, 2009.

\bibitem{DSJ}
Daniela De~Silva and David Jerison.
\newblock A gradient bound for free boundary graphs.
\newblock {\em Comm. Pure Appl. Math.}, 64(4):538--555, 2011.

\bibitem{PKW}
M.~Del~Pino, M.~Kowalczyk, and J.~Wei.
\newblock On {D}e {G}iorgi conjecture in dimension $n\geq 9$.
\newblock {\em to appear in Ann. of Math.
  http://annals.math.princeton.edu/articles/2296}.

\bibitem{GhouGui}
N.~Ghoussoub and C.~Gui.
\newblock On a conjecture of {D}e {G}iorgi and some related problems.
\newblock {\em Math. Ann.}, 311(3):481--491, 1998.

\bibitem{JK}
David~S. Jerison and Carlos~E. Kenig.
\newblock Boundary behavior of harmonic functions in nontangentially accessible
  domains.
\newblock {\em Adv. in Math.}, 46(1), 1982.

\bibitem{Kawohl}
Bernhard Kawohl.
\newblock {\em Rearrangements and convexity of level sets in {PDE}}, volume
  1150 of {\em Lecture Notes in Mathematics}.
\newblock Springer-Verlag, Berlin, 1985.

\bibitem{modica}
Luciano Modica.
\newblock {$\Gamma $}-convergence to minimal surfaces problem and global
  solutions of {$\Delta u=2(u^{3}-u)$}.
\newblock In {\em Proceedings of the {I}nternational {M}eeting on {R}ecent
  {M}ethods in {N}onlinear {A}nalysis ({R}ome, 1978)}, pages 223--244, Bologna,
  1979. Pitagora.

\bibitem{Morrey}
Charles~B. Morrey, Jr.
\newblock {\em Multiple integrals in the calculus of variations}.
\newblock Die Grundlehren der mathematischen Wissenschaften, Band 130.
  Springer-Verlag New York, Inc., New York, 1966.

\bibitem{Savin}
Ovidiu Savin.
\newblock Regularity of flat level sets in phase transitions.
\newblock {\em Ann. of Math. (2)}, 169(1):41--78, 2009.

\bibitem{SavinCDM}
Ovidiu Savin.
\newblock Phase transitions, minimal surfaces and a conjecture of {D}e
  {G}iorgi.
\newblock In {\em Current developments in mathematics, 2009}, pages 59--113.
  Int. Press, Somerville, MA, 2010.

\bibitem{Simon}
Leon Simon.
\newblock Entire solutions of the minimal surface equation.
\newblock {\em J. Differential Geom.}, 30(3):643--688, 1989.

\bibitem{Simons}
James Simons.
\newblock Minimal varieties in riemannian manifolds.
\newblock {\em Ann. of Math. (2)}, 88:62--105, 1968.

\end{thebibliography}
\end{document}